\documentclass{article}
\usepackage[cp1251]{inputenc}
\usepackage[english]{babel}
\usepackage{amsmath}
\usepackage{amsthm}
\usepackage{mathtext}
\usepackage{amssymb}
\usepackage{amscd}
\usepackage{graphics}
\usepackage{euscript}
\usepackage{bbm}
\usepackage[matrix,arrow,curve]{xy}

\theoremstyle{plain}
\newtheorem{thm}{Theorem}[section]
\newtheorem*{thm*}{Theorem}
\newtheorem{prop}[thm]{Proposition}
\newtheorem{cor}[thm]{Corollary}
\newtheorem*{prop*}{Proposition}
\newtheorem{lemma}[thm]{Lemma}

\theoremstyle{definition}
\newtheorem{defin}[thm]{Definition}
\newtheorem{exam}[thm]{Example}
\newtheorem{rem}[thm]{Remark}
\newtheorem{quest}[thm]{Question}

\DeclareMathOperator{\Qsym}{Qsym}
\DeclareMathOperator{\Sym}{Sym}
\DeclareMathOperator{\rank}{rank}
\DeclareMathOperator{\rk}{rk}
\DeclareMathOperator{\id}{id}
\DeclareMathOperator{\Ker}{Ker}
\DeclareMathOperator{\im}{Im}
\DeclareMathOperator{\conv}{conv}
\DeclareMathOperator{\aff}{aff}

\DeclareMathOperator{\Hom}{Hom}

\DeclareMathOperator{\Coh}{H}
\DeclareMathOperator{\C}{H}
\DeclareMathOperator{\pt}{pt}
\DeclareMathOperator{\Lyn}{LYN}
\DeclareMathOperator{\R}{R}
\DeclareMathOperator{\F}{F}
\DeclareMathOperator{\Ps}{P}
\DeclareMathOperator{\Q}{Q}

\newcommand {\Qs}{\mathcal{Q}sym}

\newcommand {\ib}[1]{\textit{\textbf{#1}}}

\begin{document}
\title{Polytopes, Hopf algebras  and Quasi-symmetric functions}
\author{Victor M.~Buchstaber \and Nickolai Erokhovets}
\date{}
\maketitle
\begin{abstract}
In this paper we use the technique of Hopf algebras and quasi-symmetric
functions to study the combinatorial polytopes. Consider the free abelian
group $\mathcal{P}$ generated by all combinatorial polytopes. There are two
natural bilinear operations on this group defined by a direct product $\times
$ and a join $\divideontimes$ of polytopes. It turns out that
$(\mathcal{P},\times)$ is a commutative associative bigraded ring of
polynomials with graduations $(2n,2(m-n))$, where $n$ is the dimension, and
$m$ is the number of facets of a polytope.  $(\mathcal{P},\divideontimes)$ is
a commutative associative ring without a unit. If we add the empty set
$\varnothing$ as a $(-1)$-dimensional  polytope, then it gives a unit, and we
obtain the commutative associative threegraded ring of polynomials
$\mathcal{RP}$ with graduations $(2(n+1),2(v-n-1),2(m-n-1))$, where $v$ is
the number of vertices of a polytope. The ring $\mathcal{RP}$ has the
structure of a graded Hopf algebra and the natural correspondence $P\to L(P)$
that sends a polytope to its face lattice embeds $\mathcal{RP}$ into the
Rota-Hopf algebra $\mathcal{R}$ of graded posets as a Hopf subalgebra. It
turns out that $\mathcal{P}$ has a natural Hopf comodule structure over the
Hopf algebra $\mathcal{RP}$. Faces operators $d_k$ that send a polytope to
the sum of all its $(n-k)$-dimensional faces define on both rings the Hopf
module structures over the Hopf algebra $\mathcal{U}$ that is universal in
the category of Leibnitz-Hopf algebras with the antipode $U_k\to(-1)^k U_k$.
This structure gives a ring homomorphism $\R\to\Qs\otimes\R$, where $\R$ is
$\mathcal{P}$ or $\mathcal{RP}$. Composing this homomorphism with the
characters $P^n\to\alpha^n$ of $\mathcal{P}$, $P^n\to\alpha^{n+1}$ of
$\mathcal{RP}$, and with the counit we obtain the ring homomorphisms
$f\colon\mathcal{P}\to\Qs[\alpha]$,
$f_{\mathcal{RP}}\colon\mathcal{RP}\to\Qs[\alpha]$, and
$\F^*:\mathcal{RP}\to\Qs$, where $F$ is the Ehrenborg transformation. We
describe the images of these homomorphisms in terms of functional equations,
prove that these images are rings of polynomials over $\mathbb Q$, and find
the relations between the images, the homomorphisms and the Hopf comodule
structures. For each homomorphism $f,\;f_{\mathcal{RP}}$, and $\F$ the images
of two polytopes coincide if and only if they have equal flag $f$-vectors.
Therefore algebraic structures on the images give the information about flag
$f$-vectors of polytopes. The homomorphism $f$ is an isomorphism on the
graded group $BB$ generated by the polytopes introduced by M.~Bayer and
L.~Billera to find the linear span of flag $f$-vectors of convex polytopes.
This gives the group $BB$ a structure of the ring isomorphic to
$f(\mathcal{P})$. Developing this approach one can build new combinatorial
invariants of convex polytopes that can not be expresses in terms of flag
numbers.

\end{abstract}
\section{Introduction}
Convex polytopes is a classical object of convex geometry. In recent times
the solution of problems of convex geometry involves results of algebraic
geometry and topology, commutative and homological algebra. There are
remarkable results lying on the crossroads of the polytope theory, theory of
complex manifolds, equivariant topology and singularity theory (see the
survey \cite{BR}). The example of successful collaboration of the polytope
theory and differential equations can be found in \cite{Buch}.

To study the combinatorics of convex polytopes we consider the free abelian
group $\mathcal{P}$ generated by all combinatorial convex polytopes. It turns
out that this object has a rich algebraic nature. The direct product $\times$
defines on $\mathcal{P}$ the structure of a bigraded associative commutative
ring $(\mathcal{P},\times)$ with graduations $(2n,\,2(m-n))$, where $n$ is
the dimension, and $m$ is the number of facets of a polytope. The bigraded
structure gives the decomposition of $\mathcal{P}$ into a direct sum of
finitely generated free abelian groups
$$
\mathcal{P}=\mathcal{P}^0+\sum\limits_{m\geqslant 2}\sum\limits_{n=1}^{m-1}\mathcal{P}^{2n,\,2(m-n)}.
$$
The duality operator $*$ defines on $\mathcal{P}$ another multiplication
$P\circ Q=(P^*\times Q^*)^*$.

The join of two polytopes $P\divideontimes Q$ is a bilinear operation of
degree $+2$. Then $(\mathcal{P},\divideontimes)$ is a commutative associative
ring without a unit.  Set $\mathcal{RP}=\mathbb Z\varnothing
\oplus\mathcal{P}$. The empty set $\varnothing$ may be considered as a
$(-1)$-dimensional polytope such that $\varnothing \divideontimes P=P$. Then
the free abelian group $\mathcal{RP}$ has the structure of a threegraded
commutative associative ring with graduations
$(2(n+1),\,2(v-n-1),\,2(m-n-1))$, where $v$ is the number of vertices of a
polytope. It turns out that $\mathcal{RP}$ is a graded Hopf algebra with
graduation $2(n+1)$ and the comultiplication
$$
\Delta(P)=\sum\limits_{\varnothing\subseteq F\subseteq P}F\otimes (P/F)
$$

The correspondence $P\to L(P)$, where $L(P)$ is a face lattice of a polytope
$P$ defines an embedding of $\mathcal{RP}$ as a Hopf subalgebra into the Hopf
algebra $\mathcal{R}$ of graded posets introduces by Joni and Rota in
\cite{JR}.

In the focus of our interest we put the ring of linear operators
$\mathcal{L}(\R)=\Hom_{\mathbb Z}(\R,\R)$, where $\R=\mathcal{P}$ or
$\mathcal{RP}$, and its subring $\mathcal{D}$ generated by the operators
$d_k,\,k\geqslant 1$ that map an $n$-dimensional polytope to the sum of all
its $(n-k)$-dimensional faces. $d=d_1$ is a derivation, so $\R$ is a
differential ring.

The ring $\mathcal{D}$ has a canonical structure of a Leibnitz-Hopf algebra
(see Section 3.2). We prove that $\mathcal{D}$ is isomorphic to the universal
Hopf algebra in the category of Leibnitz-Hopf algebras with the antipode
$\chi(d_k)=(-1)^kd_k$.
$$
\mathcal{D}\simeq\mathcal{Z}/J_{\mathcal{D}}
$$
where $\mathcal{Z}=\mathbb Z\langle{Z_1,Z_2,\dots\rangle}$ is the universal
Leibnitz-Hopf algebra -- the free associative Hopf algebra with the
comultiplication
$$
\Delta Z_k=\sum\limits_{i+j=k}Z_i\otimes Z_j,\quad Z_0=1,
$$
and $J_{\mathcal{D}}$ is the two-sided Hopf ideal generated by the relations
$\sum\limits_{i+j=n}(-1)^iZ_iZ_j=0,\quad n\geqslant 1$.

This algebra appears in various application of theory of Hopf algebras in
combinatorics: over the rationals it is isomorphic to the graded dual of the
odd subalgebra $S_{-}(\Qsym[t_1,t_2,\dots],\zeta_{\mathcal{Q}})$ \cite[Remark
6.7]{ABS}, to the algebra of forms on chain operators \cite[Proposition
3.2]{BLiu}, to the factor algebra of the algebra of Piery operators on a
Eulerian poset by the ideal generated by the Euler relations \cite[Example
5.3]{BMSW}.

The action of $\mathcal{D}$ on $\R$ satisfies the property
$D_{\omega}(PQ)=\mu\circ(\Delta D_{\omega})(P\otimes Q)$, so the ring $\R$
has the structure of a Milnor module over $\mathcal{D}$.

The graded dual Hopf algebra to $\mathcal{Z}$ is the algebra of
quasi-symmetric functions $\Qsym[t_1,t_2,\dots]$. The character
$\xi_{\alpha}\colon\mathcal{P}\to\mathbb Z[\alpha]$, $P^n\to\alpha^n$ induces
the homomorphism $\mathcal{P}\to\mathcal{D}^*[\alpha]$, and the composition
with the embedding $\mathcal{D}^*\subset\Qsym[t_1,t_2,\dots]$ gives the ring
homomorphism
$$
f:\mathcal{P}\to\Qsym[t_1,t_2,\dots][\alpha]
$$
We call $f(\alpha,t_1,t_2,\dots)$ the generalized $f$-polynomial. In can be
shown that $f(\alpha,t_1,0,0,\dots)=f_1(\alpha,t)$ is a homogeneous
$f$-polynomial in two variables introduced in \cite{Buch}. For  simple
polytopes we have $f(\alpha,t_1,t_2,\dots)=f_1(\alpha,t_1+t_2+\dots)$.

The homomorphism $f$ can be considered in a more general context. The
structure of a right Milnor module on the ring $\R$ over the universal Hopf
algebra $\mathcal{Z}$ defines a left Hopf comodule structure $\R\to\Qs\otimes
\R$. Then any ring homomorphism $\R\to\mathcal{R}_1$ defines the ring
homomorphism $\R\to\Qs\otimes\mathcal{R}_1$. In particular, the character
$P^n\to\alpha^n$ defines the ring homomorphism
$f\colon\mathcal{P}\to\Qs[\alpha]$, the character
$\varepsilon_{\alpha}(P^n)=\alpha^{n+1}$ defines a ring homomorphism
$f_{\mathcal{RP}}\colon\mathcal{RP}\to\Qs[\alpha]$, the counit $\varepsilon$
defines a ring homomorphism $\F^*\colon\mathcal{RP}\to\Qs$ such that $\F(P)$
is the Ehrenborg $F$-quasi-symmetric function of a graded poset $L(P)$. The
image of two polytopes under any of homomorphisms $f,f_{\mathcal{RP}}$, or
$\F^*$ coincide if and only if the polytopes have equal flag $f$-numbers.

It turns out that there is a nice algebraic structure including the rings
$\mathcal{P}$ and $\mathcal{RP}$. Namely, the formula
$$
\Delta_{\mathcal{RP}}P=\sum\limits_{F\subseteq P}F\otimes (P/F)
$$
defines on the ring $\mathcal{P}$ the structure of a graded Hopf comodule
over the Hopf algebra $\mathcal{RP}$ with respect to graduations defined by
dimension of a polytope.  This approach clarifies the structures mentioned
above and allows us to build new ring homomorphisms from the ring of
polytopes $\mathcal{P}$ to rings of polynomials such that the image of a
polytope is not determined by its flag $f$-vector. On this way one can build
new combinatorial invariants of polytopes that can not be expressed in terms
of flag numbers. This will be the topic of our next article.

The paper has the following structure:

Section 2 describes the main definitions and constructions involving convex
polytopes.
    \begin{itemize}
    \item[-]part $2.1$ contains the main definitions and considers the
        basic constructions involving polytopes;
    \item[-]in part $2.2$ the ring of polytopes $\mathcal{P}$ is
        defined.
    \item[-]in part $2.3$ we establish the structure of a graded Hopf
        algebra on the ring $\mathcal{RP}$ and consider it as a Hopf
        subalgebra of the Rota-Hopf algebra $\mathcal{R}$.
    \item[-]part $2.4$ is devoted to linear operators defined on the
        rings $\mathcal{P}$ and $\mathcal{RP}$. In particular, the faces
        operators $d_k$ are defined.
    \item[-]part $2.5$  recalls important facts about flag
        $f$-vectors and emphasizes some peculiarities we need in this
        paper.
    \end{itemize}

Section 3 is devoted to necessary definitions and important facts
    about quasi-symmetric functions and Hopf algebras.
\begin{itemize}
\item[-] part $3.1$ contains the definition of quasi-symmetric functions;
\item[-] part $3.2$ includes the definition of Leibnitz-Hopf algebras;
\item[-] part $3.3$ -- of Lie-Hopf algebras;
\item[-] part $3.4$ describes the idea of the proof of the
    shuffle-algebra theorem.
\end{itemize}

Section 4 contains topological realizations of the Hopf algebras
    we study.

In Section 5 the structure theorem for the ring $\mathcal{D}$ is
    proved.

In Section 6 we define and study the generalized $f$-polynomial.
    \begin{itemize}
    \item[-]part $6.1$ contains the construction of $f$;
    \item[-]in part $6.2$ we find the functional equations that describe
        the image of the ring $\mathcal{P}$ in
        $\Qsym[t_1,t_2,\dots][\alpha]$. These equations are equivalent to
        the generalized Dehn-Sommerville equations discovered by M.~Bayer
        and L.~Billera in \cite{BB}.
    \item[-]part $6.3$ describes the properties that define the
        generalized $f$-polynomial in a unique way.
    \end{itemize}

In Section 7 we study the Hopf algebra $\mathcal{D}^*$ as a Hopf subalgebra
in $\Qs$.
\begin{itemize}
\item[-]in part $7.1$ the main theorem is proved;
\item[-]part $7.2$ contains the applications of the main theorem. In
    particular, the image of the ring of polytopes $\mathcal{P}$ in
    $\mathcal{D}^*[\alpha]$ under the natural mapping $\varphi_{\alpha}$
    is found.
\end{itemize}

In Section $8$  the multiplicative structure of the ring
$f(\mathcal{P}\otimes \mathbb Q)$ is described.  We prove that
$f(\mathcal{P}\otimes \mathbb Q)$ is a free polynomial algebra with dimension
of the $2n$-th graded component equal to the $n$-th Fibonacci number $c_n$
($c_0=c_1=1$, $c_{n+1}=c_n+c_{n-1}$, $n\geqslant 1$). This gives the
decomposition of the Fibonacci series into the infinite product
$$
\frac{1}{1-t-t^2}=\sum\limits_{n=0}^{\infty}c_nt^n=\prod\limits_{i=1}^{\infty}\frac{1}{(1-t^i)^{k_i}},
$$
where $k_i$ is the number of generators of degree $2i$. The infinite product
converges absolutely in the interval $|t|<\frac{\sqrt{5}-1}{2}$. The numbers
$k_n$ satisfy the inequalities $k_{n+1}\geqslant k_n\geqslant N_n-2$, where
$N_n$ is the number of decompositions of $n$ into the sum of odd numbers.
Moreover, standard methods allow one to find the explicit formula for $k_n$,
which in the case of prime $p$ has the form
$$
k_p=\sum\limits_{j=1}^{[\frac{p}{2}]}\frac{{p-j\choose j}}{p-j}
$$

In Section $9$ we introduce the multiplicative structure on the graded group
$BB$ generated by the Bayer-Billera polytopes arising from the isomorphism
with the ring $f(\mathcal{P})$.

Nowadays Hopf algebras is one of the central tools in combinatorics.  There
is a well-known Hopf algebra $\mathcal{R}$ of posets introduces by Joni and
Rota in \cite{JR}. Various aspects of this algebra were studied in
\cite{Ehr,ABS,Sch1,Sch2}. The generalization of the Rota-Hopf algebra was
proposed in  \cite{RS}. In \cite{Ehr} R.~Ehrenborg introduced the
$F$-quasi-symmetric function, which gives a Hopf algebra homomorphism from
the Rota-Hopf algebra to $\Qsym[t_1,t_2,\dots]$.  Section $10$ contains the
following results:
\begin{itemize}
\item[-] in part $10.1$ we consider the Hopf comodule structures arising
    from the Hopf module structures on the rings $\mathcal{P}$ and
    $\mathcal{RP}$ over the Hopf algebras $\mathcal{Z}$ and
    $\mathcal{D}$;
\item[-] in part $10.2$ we describe the ring homomorphisms that arise
    from this structures. As an example, we show how the generalized
    $f$-polynomial and the Ehrenborg $F$-quasi-symmetric function appear
    in this context;
\item[-] part $10.3$ contains the definition and the main properties of
    the natural Hopf comodule structure on the ring $\mathcal{P}$ over
    the Hopf algebra $\mathcal{RP}\subset\mathcal{R}$;
\item[-] part $10.4$ describes the interrelations between the Hopf
    comodule structures and the homomorphisms arising in this section. In
    particular, it is shown that the Hopf comodule structures from parts
    $10.1$ and $10.3$  are related by the Ehrenborg transformation.
\end{itemize}

Section $11$ is devoted  to the cone and the bipyramid operators $C$ and $B$
used by M.~Bayer and L.~Billera \cite{BB} to find the linear span of flag
$f$-vectors of polytopes. We show that the operations $C$ and $B$ can be
defined on the ring $\Qsym[t_1,t_2,\dots][\alpha]$ in such a way that
$f(CP)=Cf(P)$ and $f(BP)=Bf(P)$. The same idea works for the ring
$\mathcal{RP}$ and the homomorphisms $f_{\mathcal{RP}}$, and $\F^*$.
\begin{itemize}
\item[-]in part $11.1$ we consider the case of $\mathcal{P}$;
\item[-]in part $11.2$ -- the case of $\mathcal{RP}$.
\end{itemize}

Section $12$ deals with the problem of the description of flag $f$-vectors of
polytopes. For simple polytope  $g$-theorem gives the full description of the
set of flag $f$-vectors. As it is mentioned in \cite{Z1} even for
$4$-dimensional non-simple polytopes the corresponding problem is extremely
hard. We show how this problem can be stated in terms of the ring of
polytopes.

This article contains the results and corrections of some inaccuracies made
in the work \cite{BuchE}.
\section{Polytopes}
\subsection{Definitions and constructions}
This section contains main definitions and constructions from the polytope
theory. For details see the books \cite{Gb}, \cite{Z2}, \cite{BP}.

There are two algorithmically different ways to define a {\itshape convex
polytope}.
\begin{defin}
A {\itshape $\mathcal{V}$-polytope} is a convex hull of a finite set of
points in some $\mathbb R^d$.
\end{defin}
\begin{defin}
An {\itshape $\mathcal{H}$-polyhedron} is an intersection of finitely many
closed halfspaces in some $\mathbb R^d$. An {\itshape $\mathcal{H}$-polytope}
is a bounded $\mathcal{H}$-polyhedron.
\end{defin}

The classical fact that these two definitions are equivalent is proved in
different books, for example in \cite[Theorem 1.1]{Z2}.

The {\itshape dimension} of a convex polytope is the dimension of its affine
hull. Without loss of generality in most cases we will assume that an
$n$-dimensional polytope $P$ lies in the space $\mathbb R^n$.

It is not difficult to prove the following lemma:
\begin{prop}
Any $n$-dimensional convex polytope $P^n$ up to a translation can be
represented in the form:
\begin{equation}\label{Polytope}
P^n=\{\ib{x}\in\mathbb R^n:\langle \ib{a}_i,\ib{x}\rangle+1\geqslant
0,\;i=1,\dots,m\},
\end{equation}
where the system of inequalities is irredundant, that is the absence of any
one of them changes the set $P^n$.
\end{prop}
In this case the polytope $P^n$ has exactly $m$ {\itshape facets}
$$
F_i=P^n\cap\{\ib{x}\in\mathbb R^n:\langle \ib{a}_i,\ib{x}\rangle+1=0\},\quad i=1,\dots,m.
$$
\begin{proof}
It is enough to make by a translation the origin $\mathbf{0}$ the inner point
of $P$.
\end{proof}

Two polytopes $P^n$ and $Q^n$ are said to be {\itshape combinatorially
equivalent} if there exists a bijection between their face lattices $L(P)$
and $L(Q)$ that preserves an inclusion relation. A {\itshape combinatorial
polytope} is an equivalence class of combinatorially equivalent convex
polytopes.

The notion of the general position from the point of view of two different
definitions of a convex polytope give two special classes of polytopes.
\begin{defin}
A polytope $P^n$ is called {\itshape simple} if any vertex of $P$ belongs to
exactly $n$ facets.

A polytope $P^n$ is called {\itshape simplicial} if any facet of $P$ is an
$(n-1)$-dimensional simplex (i.e. contains exactly $n$ vertices).
\end{defin}

\begin{defin}[dual polytope]
Let $P^n$ be an $n$-dimensional polytope (\ref{Polytope}). Then a {\itshape
dual {\upshape (or {\itshape polar})} polytope} is defined as
$$
(P^n)^*=\{\ib{y}\in\mathbb R^n:\langle \ib{y},\ib{x}\rangle+1\geqslant
0,\;\forall \ib{x}\in P^n\}.
$$
It can be shown (see, e.g. \cite{Z2}) that
$$
(P^n)^*=\conv\{\ib{a}_i,\,i=1,\dots,m\},
$$
\end{defin}
and $(P^*)^*=P$.

There is a bijection $F\longleftrightarrow F^{\diamondsuit}$ between the
$i$-faces of $P$ and the $(n-1-i)$-faces of $P^*$ such that
$$
F\subset G\Leftrightarrow G^{\diamondsuit}\subset
F^{\diamondsuit}.
$$

For any simple polytope $P$ its polar polytope $P^*$ is simplicial and vice
versa.

\begin{defin}[face polytope]
For the face $F$ of the polytope $P$ let us define a {\itshape face polytope
$P/F$} as
$$
P/F=(F^{\diamondsuit})^*
$$
It is easy to see that the face polytope has dimension $\dim P-\dim F-1$ and
the face lattice
$$
L(P/F)=[F,P]=\{G\in L(P)\colon F\subseteq G\subseteq P\}.
$$
\end{defin}
\begin{exam}
Let $P^n$ be a simple polytope. Then for any proper $k$-dimensional face $F$
of $P$ we have $P/F=\Delta^{n-k-1}$.
\end{exam}

\begin{defin}[direct product]
For two polytopes $P^{n_1}\subset\mathbb\mathbb R^{n_1}$ and
$Q^{n_2}\subset\mathbb R^{n_2}$ with $m_1$ and $m_2$ facets respectively the
{\itshape direct product}
$$
P^{n_1}\times Q^{n_2}=\{(\ib{x},\,\ib{y})\subset\mathbb R^{n_1}\times\mathbb R^{n_2}\colon \ib{x}\in P^{n_1},\,\ib{y}\in\mathbb R^{n_2}\}
$$
is an $(n_1+n_2)$-dimensional convex polytope. Its faces are direct products
of faces of $P^{n_1}$ and $Q^{n_2}$. In particular, it has $m_1+m_2$ facets.
\end{defin}
\begin{defin}[operation $\circ$]
For two polytopes
\begin{gather*}
P^{n_1}=\{\ib{x}\in\mathbb R^{n_1}:\;\langle
\ib{a}_i,\ib{x}\rangle+1\geqslant 0,\;i=1,\dots,m_1\},\\
Q^{n_2}=\{\ib{y}\in\mathbb R^{n_2}:\;\langle
\ib{b}_j,\ib{y}\rangle+1\geqslant 0,\;j=1,\dots,m_2\}
\end{gather*}
let us define the polytope
$$
P^{n_1}\circ Q^{n_2}=\conv(P^{n_1}\times\{\mathbf 0\}\cup\{\mathbf 0\}\times Q^{n_2})\subset\mathbb R^{n_1}\times\mathbb R^{n_2}
$$
\end{defin}
\begin{prop}
For two polytopes $P^{n_1}$ and $Q^{n_2}$ represented in the form
(\ref{Polytope}) we have
$$
(P^{n_1}\times Q^{n_2})^*=(P^{n_1})^*\circ(Q^{n_2})^*
$$
\end{prop}
\begin{proof}
The polytope $P^{n_1}\times Q^{n_2}$ is defined by the inequalities
$$
\langle \ib{a}_i,\ib{x}\rangle+1\geqslant 0,\,i=1,\dots,m_1,\quad\langle \ib{b}_j,\ib{y}\rangle+1\geqslant 0,\,j=1,\dots,m_2,
$$
therefore,
$$
(P^{n_1}\times Q^{n_2})^*=\conv\{(\ib{a}_i,\mathbf{0}),(\mathbf{0},\ib{b}_j),\;i=1,\dots,m_1,\;j=1,\dots,m_2\}=(P^{n_1})^*\circ (Q^{n_2})^*
$$
\end{proof}

Since the product of two simple polytopes is again a simple polytope, we see
that for any two simplicial polytopes $P$ and $Q$ the polytope $P\circ Q$ is
again simplicial.

\begin{exam}[bipyramid]
Let us define a  {\itshape bipyramid} as $BP=I\circ P$. Then we have
$(BP)^*=I\times P^*$.
\end{exam}

\begin{defin}[join]
For the space $\mathbb R^n$ let $\Delta^{n-1}$ be the regular simplex
$\conv(e_1,\dots,e_n)$, where \linebreak $e_i=(0,\dots,0,1,0,\dots,0)$ is the
$i$-th basis vector. Let $P^{n_1}\subset \mathbb R^{n_1+1}$, and
$Q^{n_2}\subset \mathbb R^{n_2+1}$ be polytopes lying in the regular
simplices $\Delta^{n_1}$ and $\Delta^{n_2}$. A {\itshape join}
$P\divideontimes Q$ is defined as
$$
P\divideontimes Q=\conv(P\times\{\mathbf{0}\}\cup\{\mathbf{0}\}\times Q)\subset\mathbb R^{n_1+1}\times\mathbb R^{n_2+1}
$$
It is easy to see that the polytope $P\divideontimes Q$ lies in the regular
simplex $\Delta^{n_1+n_2+1}$.
\end{defin}
From this construction is evident that the join is an associative operation.
\begin{rem}\label{Join}
In the definition of the join we used the polytopes that lie in the regular
simplices. In fact, the affine type of $P\divideontimes Q$ does not depend on
the positions of $P$ and $Q$ in the regular simplices, moreover, it depends
only on the affine types of $P$ and $Q$.

Indeed, let $\{\ib{v}_1,\dots,\ib{v}_{n_1+1}\}$ be the set of vertices of $P$
in general position, and $\{\ib{w}_1,\dots,\ib{w}_{n_2+1}\}$ be the set of
vertices of $Q$ in general position. Since the origin $O=\mathbf 0\in\mathbb
R^{n_1+1}\subset\mathbb R^{n_1+1}\times\mathbb R^{n_2+1}$ does not belong to
$\aff P$, the vectors
$\overrightarrow{O\ib{v}_1},\dots,\overrightarrow{O\ib{v}_{n_1+1}}$ are
linearly independent. Similarly the vectors
$\overrightarrow{O\ib{w}_1},\dots,\overrightarrow{O\ib{w}_{n_2+1}}$ form a
basis in $\mathbb R^{n_2+1}$. Then the vectors
$\overrightarrow{O\ib{v}_1},\dots,\overrightarrow{O\ib{v}_{n_1+1}},\overrightarrow{O\ib{w}_1},\dots,\overrightarrow{O\ib{w}_{n_2+1}}$
are linearly independent in $\mathbb R^{n_1+1}\times \mathbb R^{n_2+1}$, and
the set of points
$\ib{v}_1,\dots,\ib{v}_{n_1+1},\ib{w}_1,\dots,\ib{w}_{n_2+1}$ is in general
position. This implies that $\dim P^{n_1}\divideontimes Q^{n_2}=n_1+n_2+1$.
Now for any two realizations $(P\in\Delta^{n_1},Q\in\Delta^{n_2})$ and
$(P'\in\Delta^{n_1},Q'\in\Delta^{n_2})$ the correspondence
$\ib{v}_i\to\ib{v}_i',\,\ib{w}_j\to\ib{w}_j'$ defines an affine isomorphism
$\aff P\divideontimes Q\simeq\aff P'\divideontimes Q'$ such that $P\to P'$,
$Q\to Q'$. Then $\conv(P\cup Q)\to\conv(P'\cup Q')$, so $P\divideontimes
Q\simeq P'\divideontimes Q'$. Since the affine types of $P$ and $Q$ are
uniquely defined by the sets of points $\{\ib{v}_i\}$ and $\{\ib{w}_j\}$, we
see, that the affine type of $P\divideontimes Q$ depends only on the affine
types of $P$ and $Q$.
\end{rem}

\begin{prop}
$P^{n_1}\divideontimes Q^{n_2}$ is an $(n_1+n_2+1)$-dimensional polytope.

Faces of $P\divideontimes Q$ up to an affine equivalence are exactly
$F\divideontimes G$, where $\varnothing\subseteq F\subseteq
P,\;\varnothing\subseteq G\subseteq Q$ are faces of $P$ and $Q$ respectively.
\end{prop}
\begin{proof}
For the proof see Appendix A.
\end{proof}

\begin{cor}
The face lattice  $L(P\divideontimes Q)$ is the direct product of the face
lattices $L(P)$ and $L(Q)$. Therefore the join operation is well defined on
combinatorial polytopes.
\end{cor}

\begin{rem}
The {\itshape join} $P\divideontimes Q$ can be defined in more invariant
terms -- as the convex hull of $P$ and $Q$ provided they are placed into
affine subspaces of some $\mathbb R^N$ such that their affine hulls $\aff(P)$
and $\aff(Q)$ are skew.
\end{rem}

\begin{exam}
For the simplices $\Delta^k$ and $\Delta^l$ we have
$\Delta^k\divideontimes\Delta^l=\Delta^{k+l+1}$.
\end{exam}

\begin{prop}\label{Dual}
For two combinatorial polytopes $P$ and $Q$
$$
(P\divideontimes Q)^*=P^*\divideontimes Q^*
$$
\end{prop}
\begin{proof}
Indeed, the faces of the polytope $(P\divideontimes Q)^*$ have the form
$(F\divideontimes G)^{\diamondsuit}$, and $(F\divideontimes
G)^{\diamondsuit}\subseteq(F'\divideontimes G')^{\diamondsuit}$ if and only
if $F'\divideontimes G'\subseteq F\divideontimes G$, i.e. $F'\subseteq F,
G'\subseteq G$, which is equivalent to
$F^{\diamondsuit}\subseteq(F')^{\diamondsuit},
G^{\diamondsuit}\subseteq(G')^{\diamondsuit}$. Therefore the correspondence
$(F\divideontimes G)^{\diamondsuit}\longleftrightarrow
F^{\diamondsuit}\divideontimes G^{\diamondsuit}$ is a combinatorial
equivalence of the polytopes $(P\divideontimes Q)^*$ and $P^*\divideontimes
Q^*$.
\end{proof}

\begin{exam}[cone]
Let us define a {\itshape cone} $CP$ as $\pt*P$, where $\pt$ is a point. Then
$(CP)^*=C(P^*)$.
\end{exam}

\subsection{Ring of polytopes}
\begin{defin}
Denote by $\mathcal{P}^{2n}$ the free abelian group generated by all
$n$-dimensional combinatorial polytopes. For $n>1$ we have the direct sum
$$
\mathcal{P}^{2n}=\sum\limits_{m\geqslant n+1}\mathcal{P}^{2n,\,2(m-n)},
$$
where $P^{n}\in\mathcal{P}^{2n,\,2(m-n)}$ if it is a polytope with $m$
facets, and $\rank \mathcal{P}^{2n,\,2(m-n)}<\infty$ for any fixed $n$ and
$m$ such that $n<m$. The product of polytopes $P\times Q$ turns the direct
sum
$$
\mathcal{P}=\sum\limits_{n\geqslant 0}\mathcal{P}^{2n} =\mathcal{P}^0+\sum\limits_{m\geqslant2}\sum\limits_{n=1}^{m-1}\mathcal{P}^{2n,\,2(m-n)}
$$
into a bigraded commutative associative ring, a {\itshape ring of polytopes}.
The unit is $P^0=\pt$, a point.

Let us denote
$\mathcal{P}^{[2n]}=\mathcal{P}^0\oplus\mathcal{P}^2\oplus\dots\oplus\mathcal{P}^{2n}$.

The direct product $P\times Q$ of simple polytopes $P$ and $Q$ is a simple
polytope as well. Thus the abelian subgroup $\mathcal{P}_s\subset
\mathcal{P}$ generated by all simple polytopes is a subring in $\mathcal{P}$.
\end{defin}

\begin{defin}
A polytope $P^n$ is called {\itshape indecomposable} if it can not be
decomposed into a direct product $P_1\times P_2$ of two polytopes of positive
dimensions.
\end{defin}

\begin{prop}
$\mathcal{P}$ is a polynomial ring generated by indecomposable combinatorial
polytopes.
\end{prop}
For the proof see Appendix B.

Since the polytope $P=P_1\times P_2$ is simple if and only if the polytopes
$P_1$ and $P_2$ are simple, $\mathcal{P}_s$ is a polynomial ring generated by
all indecomposable simple polytopes.

Since $(P^*)^*=P$, the correspondence $P\to P^*$ defines an involutory linear
isomorphism  of the ring $\mathcal{P}$.

Any isomorphism of graded abelian groups $A\colon\mathcal{P}\to\mathcal{P}$
defines a new multiplication
$$
P\times_{A}Q=A^{-1}(AP\times AQ)
$$
such that $A$ is a graded ring isomorphism
$(\mathcal{P},\times_{A})\to(\mathcal{P},\times)$.

In particular, $*$ defines the multiplication
$$
P\circ Q=\left(P^*\times Q^*\right)^*
$$

\begin{defin}
Let us define an {\itshape alpha-character}
$\xi_{\alpha}\colon\mathcal{P}\to\mathbb Z[\alpha]$:
$$
\xi_{\alpha}(P^n)=\alpha^n
$$
\end{defin}

\subsection{The Rota-Hopf algebra}
\begin{defin}
Let $\Ps$ be a finite poset with a minimal element $\hat 0$ and a maximal
element $\hat 1$.

An element $y$ in $\Ps$ {\itshape covers} another element $x$ in $\Ps$ if
$x<y$ and there is no $z$ in $\Ps$ such that $x<z<y$.

A poset $\Ps$ is called {\itshape graded} if there exists a rank function
$\rho\colon\Ps\to\mathbb Z$ such that $\rho(\hat0)=0$ and $\rho(y)=\rho(x)+1$
if $y$ covers $x$. It is easy to see that if a rank function exists, then it
is determined in a unique way.

Set $\rho(\Ps)=\rho(\hat 1)$, and $\deg(\Ps)=2\rho(\Ps)$.

Two finite graded posets are isomorphic if there exists an order preserving
bijection between them.

Let us denote by $\mathcal{R}$ the graded free abelian group with basis the
set of all isomorphism classes of finite graded posets.
\end{defin}

$\mathcal{R}$ has the structure of a {\itshape graded connected Hopf
algebra}:
\begin{itemize}
\item The multiplication $\Ps\cdot \Q$ is a cartesian product $\Ps\times
    \Q$ of posets $\Ps$ and $\Q$: let $x,u\in \Ps$, and $y,v\in \Q$. Then
$$
(x,y)\leqslant (u,v)\;\mbox{ if and only if $x\leqslant u$ and $y\leqslant v$}
$$

\item The unit element in $\mathcal{R}$ is the poset with one element
    $\hat 0=\hat 1$

\item The comultiplication is
$$
\Delta(\Ps)=\sum\limits_{\hat 0\leqslant z\leqslant \hat 1}[\hat 0,z]\otimes[z,\hat 1],
$$
where $[x,y]$ is the subposet $\{z\in \Ps|x\leqslant z\leqslant y\}$.

\item The counit $\varepsilon$ is
$$
\varepsilon(\Ps)=\begin{cases}
1,&\mbox{if }\hat 0=\hat 1;\\
0,& \mbox{else }
\end{cases}
$$

\item The antipode $\chi$ is
$$
\chi(\Ps)=\sum\limits_{k\geqslant 0}\sum\limits_{C_k}(-1)^k[x_0,x_1]\cdot[x_1,x_2]\dots[x_{k-1},x_k],
$$
where $C_k=(\hat 0=x_0<x_1<\dots<x_k=\hat 1)$
\end{itemize}
The Hopf algebra of graded posets was originated in the work by Joni and Rota
\cite{JR}.  Variations of this construction were studied in
\cite{Ehr,ABS,Sch1,Sch2}. The generalization of this algebra can be found in
\cite{RS}.

\begin{exam} The simplest Boolean algebra $B_1=\{\hat 0,\hat 1\}$ is the face lattice of the point $\pt$.
\begin{gather*}
\Delta(B_1)=1\otimes B_1+B_1\otimes 1\\
\chi(B_1)=-B_1
\end{gather*}
\end{exam}

\begin{defin}
There is a natural involutory graded ring isomorphism $*$ of the ring
$\mathcal{R}$. For a graded poset $\Ps$ let $\Ps^*$ be a graded poset
consisting of the same elements but with the reverse inclusion relation, i.e.
$x\leqslant_{\Ps} y$ if and only is $x\geqslant_{\Ps^*} y$. Then
$\rho(\Ps^*)=\rho(P)$, and $\Delta(\Ps^*)=*\otimes *(\tau_{\mathcal{R}}\Delta
\Ps)$, where $\tau_{\mathcal{R}}$ is a ring homomorphism
$\mathcal{R}\otimes\mathcal{R}\to\mathcal{R}\otimes\mathcal{R}$ that
interchanges the tensor factors $\tau_{\mathcal{R}}(x\otimes y)=y\otimes x$.
\end{defin}

\begin{defin}
A join of polytopes defines a bilinear operation of degree $+2$ on the ring
$\mathcal{P}$. It is associative and commutative, so
$(\mathcal{P},\divideontimes)$ is a  commutative associative ring without a
unit.

Let us add a formal unit $\varnothing$ of degree $-2$ corresponding to the
empty set. Thus $\varnothing\divideontimes P=P\divideontimes\varnothing=P$.

Let us denote the ring $(\mathcal{P}\oplus\mathbb
Z\varnothing,\divideontimes)$ by $\mathcal{RP}$. Then $\mathcal{RP}$ is a
commutative associative ring.

Given two polytopes: $P^{n_1}$ with $v_1$ vertices and $m_1$ facets, and
$Q^{n_2}$ with $v_2$ vertices and $m_2$ facets, the polytope
$P^{n_1}\divideontimes Q^{n_2}$ has dimension $n_1+n_2+1$, while its number
of vertices is $v_1+v_2$, and its number of facets is $m_1+m_2$.

Therefore $\mathcal{RP}$ is a threegraded ring with graduations $2(n+1)$,
$2(v-n-1)$, and $2(m-n-1)$, where $n=\dim P$, $v$ is the number of vertices,
and $m$ is the number of facets. The duality operator $*$ interchanges the
graduations $2(v-n-1)$ and $2(m-n-1)$.
\end{defin}

\begin{rem}
Let us recall that in the ring $\mathcal{P}$ the polytope $P$ has two
graduations $2n$ and $2(m-n)$. We have $v(P\times Q)=v(P)\cdot v(Q)$, so the
number of vertices does not give the third graduation in this case.
\end{rem}

\begin{defin}
Let us denote by $\mathcal{SP}$ the graded abelian subgroup consisting of all
self-dual elements of the ring $\mathcal{RP}$. According to Proposition
\ref{Dual} it is a subring in $\mathcal{RP}$.
\end{defin}

\begin{exam}
$\mathcal{SP}$ contains the subring generated by simplices and all polygons
$M^2_n$.
\end{exam}

The correspondence $P\to L(P)$ defines the homomorphism of abelian groups
$L:\mathcal{P}\oplus\mathbb Z\varnothing\to\mathcal{R}$. It is an injection.
Since $L(P\divideontimes Q)=L(P)\cdot L(Q)$, it is a homomorphism of graded
rings $\mathcal{RP}\to\mathcal{R}$. Moreover, $L(P^*)=L(P)^*$.

\begin{defin}
Since
$$
\Delta(L(P))=\sum\limits_{\varnothing\subseteq F\subseteq P}[\varnothing,F]\otimes[F,P]=\sum\limits_{\varnothing\subseteq F\subseteq P}L(F)\otimes L(P/F),
$$
we see that $L(\mathcal{RP})$ is a Hopf subalgebra in $\mathcal{R}$,
therefore it induces a Hopf algebra structure on the ring $\mathcal{RP}$:
\begin{itemize}
\item The comultiplication in $\mathcal{RP}$ is
$$
\Delta(P)=\sum\limits_{\varnothing\subseteq F\subseteq P}F\otimes P/F;
$$
\item The counit is
$$\varepsilon(P)=\begin{cases}
1,&\mbox{ if } P=\varnothing,\\
0,&\mbox{else};
\end{cases}
$$
\item The antipode $\chi$ is defined as
$$
\chi(P)=\sum\limits_{k\geqslant 0}\sum\limits_{\varnothing=F_0\subset F_1\subset\dots\subset F_k=P}(-1)^k(F_1/F_0)\divideontimes(F_2/F_1)\divideontimes\dots\divideontimes(F_{k}/F_{k-1}).
$$
\end{itemize}
\end{defin}

\begin{rem}
Let
$\tau_{\mathcal{RP}}:\mathcal{RP}\otimes\mathcal{RP}\to\mathcal{RP}\otimes\mathcal{RP}$
be the ring homomorphism that interchanges the tensor factors:
$\tau_{\mathcal{RP}}(x\otimes y)=y\otimes x$. Then $\Delta (P^*)=*\otimes
*(\tau_{\mathcal{RP}}\Delta P)$.
\end{rem}

\begin{prop}
$\mathcal{RP}$ is a ring of polynomials generated by all join-indecomposable
polytopes.
\end{prop}
See Appendix B for the proof.

\begin{defin}
For the ring $\mathcal{RP}$ let us define an {\itshape $\alpha$-character}
$\varepsilon_{\alpha}$ by the formula
$$
\varepsilon_{\alpha}(P^n)=\alpha^{n+1}=\alpha^{\rho (P)}
$$
Then $\varepsilon_{0}$ is the counit $\varepsilon$.
\end{defin}

\begin{defin}
Let $\mathcal{B}$ be the abelian subgroup in $\mathcal{RP}$ generated by
$\varnothing $, and all the simplices $\Delta^n, n\geqslant 0$. Then
$L(\Delta^{n-1})=B_n$ is a Boolean algebra $\{\hat 0,\hat 1\}^n$. We have
$\Delta^k\divideontimes\Delta^l=\Delta^{k+l+1}$, so it is a subring in
$\mathcal{RP}$.

Let us denote $x=\Delta^0=\pt$. Then $\Delta^{n-1}=x^n$. We have $\Delta
x=1\otimes x+x\otimes 1$, $\chi(x)=-x$, so $\mathcal{B}$ is a Hopf subalgebra
isomorphic to the Hopf algebra of polynomials $\mathbb Z[x]$ with the
comultiplication $\Delta x=1\otimes x+x\otimes 1$.
\end{defin}

\begin{defin}
We will denote by $\R$ the ring $\mathcal{P}$ or the ring $\mathcal{RP}$, and
by $\R^{[2n]}$ the sum
$$
\R^0\oplus\R^{2}\oplus\R^4\oplus\dots\oplus\R^{2n}
$$
\end{defin}

\subsection{Operators}

\begin{defin}
Let us define the ring of linear operators $\mathcal{L}(\R)=\Hom_{\mathbb
Z}(\R,\R)$ with the multiplication given by a composition of operators.
\end{defin}

\begin{defin}
For $k\geqslant 0$ let us define {\itshape faces operators} $d_k$ that sent a
polytope $P^n$ to the sum of all its $(n-k)$-dimensional faces.
$$
d_kP^n=\sum\limits_{F^{n-k}\subseteq P^n}F^{n-k}
$$
In particular, $d_{n+1}P^n=
\begin{cases}\varnothing& \mbox{ in } \mathcal{RP},\\
0&\mbox{ in }\mathcal{P}
\end{cases}$,\quad
$d_k\varnothing=0,\;k\geqslant 1$.
\end{defin}
\begin{exam}
We have $d_1\pt=\varnothing$ in $\mathcal{RP}$, and $0$ in $\mathcal{P}$.
\end{exam}

\begin{defin}
Let us define by $\mathcal{D}(\R)\subset\mathcal{L}(\R)$ the ring of
operators generated by faces operators $d_k$, $k\geqslant 0$, and by
$\mathcal{D}(\R^{2n},\R^{2(n-k)})$ the set of all operators in
$\mathcal{D}(\R)$ of degree $-2k$ acting on the abelian group $\R^{2n}$.
\end{defin}

\begin{prop}
For any two polytopes $P$ and $Q$, and $d=d_1$ we have:
$$
d(P\times Q)=(dP)\times Q+P\times (dQ) \mbox{ in }\mathcal{P},\mbox{ and }d(P\divideontimes Q)=(dP)\divideontimes Q+P\divideontimes(dQ)\mbox{ in }\mathcal{RP}.
$$
Therefore $\mathcal{P}$ and $\mathcal{RP}$ are differential rings.
\end{prop}

\begin{proof}
Each facet of $P\times Q$ has the form $F\times Q$ or $P\times G$, where
$F,G$ are facets of $P$ and $Q$. Therefore,
$$
d(P\times Q)=\sum\limits_{F\subset P}F\times Q+\sum\limits_{G\subset Q}P\times G=(dP)\times Q+P\times (dQ)
$$
Each facet of $P\divideontimes Q$ has the form $F\divideontimes Q$ or
$P\divideontimes G$,  where $F,G$ are facets of $P$ and $Q$. Therefore,
$$
d(P\divideontimes Q)=(dP)\divideontimes Q+P\divideontimes(dQ)
$$
\end{proof}

\begin{exam}
$xP=\pt\divideontimes P=CP$ is a {\itshape cone}. For a polytope $P$ of
positive dimension we have
$$
d(CP)=d(\pt\divideontimes P)=(d\pt)\divideontimes P+\pt\divideontimes(dP)=\varnothing\divideontimes P+C(dP)=P+C(dP)
$$
Thus, the multiplication by $x=\pt$ in the ring $\mathcal{RP}$ defines a
{\itshape cone operator} $C\in\mathcal{L}(\mathcal{RP})$ of degree $+2$ such
that
$$
[d,C]=\id
$$
On the other hand, in the ring $\mathcal{P}$ the same relation holds for all
the polytopes except for $P=\pt$. In this case we have $d(C\pt)=dI=2\pt$,
while $C(d\pt)+\pt=\pt$. Therefore, in the ring $\mathcal{P}$ we have
$$
[d,C]=\id+\xi_0
$$
\end{exam}

\begin{exam}
We have
$$
d(\Delta^k\divideontimes\Delta^l)=(k+1)\Delta^{k-1}\divideontimes\Delta^l+(l+1)\Delta^k\divideontimes\Delta^{l-1}=(k+l+2)\Delta^{k+l}=d\Delta^{k+l+1}
$$
\end{exam}

\begin{exam}
Since $\Delta^k=x^{k+1}$, we see that $dx^{k+1}=(k+1)x^k$. Therefore $\mathbb
Z[x]$ is a differential subring in $(\mathcal{RP},d)$ with $d=\frac{d}{dx}$.
\end{exam}

\begin{defin}
The correspondence $P\to\ BP$ defines a linear {\itshape bipyramid operator}
of degree $+2$ on the rings $\mathcal{P}$ and $\mathcal{RP}$. By definition
$B\varnothing=\pt$.
\end{defin}

\begin{rem}
It is easy to see that if $P$ is simple and \emph{self-dual}, i.e. $P^*$ is
combinatorially equivalent to $P$, then $P$ is a simplex. There are many
examples of self-dual non-simple polytopes, with simplest infinite family
given by $k$-gonal pyramids for $k\ge4$. The next example gives a more
interesting \emph{regular} self-dual polytope.

\begin{exam}[24-cell]
Let $Q$ be the 4-polytope obtained by taking the convex hull of the following
24 points in $\mathbb{R}^4$: endpoints of 8 vectors $\pm {\bf e}_i$, \ $1\le
i\le 4$, and 16 points of the form
$(\pm\frac12,\pm\frac12,\pm\frac12,\pm\frac12)$. It can be shown that all the
facets of $Q$ are octahedra.

The polar polytope $Q^*$ is specified by the following 24 inequalities:
\begin{equation}\label{24ineq}
  {}\mathbin\pm x_i+1\ge0\;\text{ for }1\le i\le 4,\quad\text{and
  }\;
  \textstyle{\frac12}({}\mathbin\pm x_1\mathbin\pm x_2\mathbin\pm x_3\mathbin\pm x_4)+1\ge0.
\end{equation}
Each of these inequalities turns into equality in exactly one of the
specified 24 points, so it defines a supporting hyperplane whose intersection
with $Q$ is only one point. This implies that $Q$ has exactly 24 vertices.
The vertices of $Q^*$ may be determined using the ``elimination process''
to~\eqref{24ineq} (see~\cite[Section 1.2]{Z2}), and as the result we obtain
24 points of the form ${}\mathbin\pm{\bf e}_i\mathbin\pm{\bf e}_j$ for $1\le
i<j\le 4$. Each supporting hyperplane defined by~\eqref{24ineq} contains
exactly 6 vertices of~$Q^*$, which form an octahedron. So both $Q$ and $Q^*$
have 24-vertices and 24 octahedral facets. In fact, both $Q$ and $Q^*$
provide examples of a \emph{regular 4-polytope} called a \emph{24-cell}. It
is the only regular self-dual polytope different from a simplex. For more
details on 24-cell and other regular polytopes see~\cite{C}.
\end{exam}
In this example we have $dQ=24BI^2$, and $(dQ)^*=24I^3$. Therefore for a
self-dual element $P\in\mathcal{SP}$, $dP$ can be non self-dual. So
$\mathcal{SP}$ is not a differential subring in $\mathcal{RP}$.

However there is another derivation defined on $\mathcal{SP}$. Since the
duality $*$ is a ring homomorphism,  $\delta=*d*$ is a derivation of the ring
$\mathcal{RP}$, as well as $d$.

\begin{prop}
$d+\delta$ is a derivation of the ring $\mathcal{SP}$
\end{prop}
\begin{proof}
Since for any $P\in\mathcal{SP}$, $P^*=P$, we have
$$
\left((d+\delta)P\right)^*=*(d+*d*)P=*dP+**dP=*d*P+dP=(\delta+d)P.
$$
\end{proof}
\end{rem}

\begin{prop}\label{MM}
\begin{align*}
d_k(P\times Q)&=\sum\limits_{i+j=k}(d_iP)\times(d_jQ)\quad\mbox{in }\mathcal{P}\\
d_k(P\divideontimes
Q)&=\sum\limits_{i+j=k}(d_iP)\divideontimes(d_jQ)\quad\mbox{in }\mathcal{RP}.
\end{align*}
\end{prop}
\begin{proof}
We have
\begin{multline*}
d_k(P^n\times Q^r)=\sum\limits_{F^{n-i}\subseteq P^n,\,G^{r-j}\subseteq
Q^r,\,i+j=k}F^{n-i}\times G^{r-j}=
\sum\limits_{i+j=k}\sum\limits_{F^{n-i}\subseteq
P^n}F^{n-i}\times\left(\sum\limits_{G^{r-j}\subseteq
Q^r}G^{r-j}\right)=\\
=\sum\limits_{i+j=k}\left(\sum\limits_{F^{n-i}\subseteq
P^n}F^{n-i}\right)\times\left(\sum\limits_{G^{r-j}\subseteq
Q^r}G^{r-j}\right)=\sum\limits_{i+j=k}(d_iP)\times (d_jQ)
\end{multline*}

The same argument works also for $P^n\divideontimes Q^r$. The only difference
is that faces of the form $\varnothing \divideontimes G^{r-j}$, $i=n+1,
i+j=k$, and $F^{n-i}\divideontimes \varnothing$, $j=r+1,i+j=k$ are allowed.
Here $\varnothing=d^{n+1}P^n=d^{r+1}Q^r$.
\end{proof}

\begin{defin}
Let us define an operator $\Phi:\R\to\R[t]$, $\Phi(t)\in\mathcal{D}(\R)[[t]]$
by the formula
$$
\Phi(t)=1+dt+d_2t^2+\dots+d_kt^k+\dots
$$
\end{defin}

\begin{defin}
Let $f_i(P^n)$ be the number of $i$-dimensional faces of $P^n$.
\end{defin}

\begin{prop}\label{PM}
We have $\Phi(-t)\Phi(t)=1$ on the ring $\R$, that is for any $n\geqslant 1$
$$
d_n-dd_{n-1}+\dots+(-1)^{n-1}d_{n-1}d+(-1)^nd_n=0.
$$
\end{prop}
\begin{proof}
The relation is clear for $\varnothing\in\mathcal{RP}$ and for
$\pt\in\mathcal{P}$. For $\pt\in\mathcal{RP}$ we have
$$
\Phi(-t)\Phi(t)\pt=(1-dt+\dots)(\pt+\varnothing t)=\pt-\varnothing
t+\varnothing t=\pt
$$
Let us consider the polytope $P^n$ of positive dimension.
$$
\Phi(-t)\Phi(t)=\sum\limits_{k=0}^{\infty}\left(\sum\limits_{i=0}^k(-1)^id_id_{k-i}\right)t^k.
$$
If $k>n+1$ then the coefficient of $t^k$ in the series $\Phi(-t)\Phi(t)P^n$
is equal to $0$. If $k=0$, then it is equal to $P^n$. Let $1\leqslant
k\leqslant n$. Then
\begin{multline*}
\left(dd_{k-1}-d_2d_{k-2}+\dots+(-1)^{k-2}d_{k-1}d\right)P^n=\\
=\sum\limits_{F^{n-k}\subset P^n}\left(\sum\limits_{F^{n-k+1}\supset
F^{n-k}}1-\sum\limits_{F^{n-k+2}\supset
F^{n-k}}1+\dots+(-1)^{k-2}\sum\limits_{F^{n-1}\supset
F^{n-k}}1\right)F^{n-k}=\\
=\sum\limits_{F^{n-k}\subset
P^n}\left(f_0(P^n/F^{n-k})-f_1(P^n/F^{n-k})+\dots+(-1)^{k-2}f_{k-2}(P^n/F^{n-k})\right)F^{n-k}
\end{multline*}
Since  $P^n/F^{n-k}$ is a $(k-1)$-dimensional polytope, the Euler formula
gives the relation
$$
f_0(P^n/F^{n-k})-f_1(P^n/F^{n-k})+\dots+(-1)^{k-2}f_{k-2}(P^n/F^{n-k})=1+(-1)^k.
$$
Thus we obtain
$$
\left(dd_{k-1}-d_2d_{k-2}+\dots+(-1)^{k-2}d_{k-1}d\right)P^n=(1+(-1)^k)d_kP^n\quad\Leftrightarrow\quad\sum\limits_{i=0}^k(-1)^id_id_{k-i}P^n=0
$$
The coefficient of $t^{n+1}$ is equal to $0$ in $\mathcal{P}$, and to
$(1-f_0+f_1+\dots+(-1)^{n}f_{n-1}+(-1)^{n+1})\varnothing$ in $\mathcal{RP}$.
Then the Euler formula implies that it is equal to $0$.
\end{proof}

\begin{prop}\label{Euler}
\begin{align*}
\xi_{-\alpha}\Phi(\alpha)&=\xi_{\alpha}\quad\mbox{in the ring }\mathcal{P}\\
\varepsilon_{-\alpha}\Phi(\alpha)&=\varepsilon_{0}\quad\mbox{in the ring }\mathcal{RP}.
\end{align*}
\end{prop}
\begin{proof}
Let us note that $\xi_{\alpha}d_kP^n=f_{n-k}(P^n)\alpha^{n-k}$. Then for any
polytope $P^n$ we have
$$
\xi_{-\alpha}\Phi(\alpha)P^n=(-\alpha)^n+(-\alpha)^{n-1}f_{n-1}\alpha+\dots+f_0\alpha^n=\left(\sum\limits_{i=0}^n(-1)^if_i\right)\alpha^n=\alpha^n=\xi_{\alpha}P^n,
$$
since the Euler formula gives the relation:
$$
f_0-f_1+\dots+(-1)^{n-1}f_{n-1}+(-1)^n=1
$$

On the other hand, in the ring $\mathcal{RP}$ we have
$\varepsilon_{\alpha}d_kP^n=f_{n-k}(P^n)\alpha^{n-k+1}$, so
$$
\varepsilon_{-\alpha}\Phi(\alpha)P^n=(-\alpha)^{n+1}+(-\alpha)^nf_{n-1}\alpha+\dots+(-\alpha)f_0\alpha^n+\alpha^{n+1}
$$
According to the Euler formula this expression is equal to $0$ for all the
polytopes but $\varnothing$. Thus
$\varepsilon_{-\alpha}\Phi(\alpha)=\varepsilon_{0}$.
\end{proof}

\begin{exam}
On the ring of simple polytopes $\mathcal{P}_s\subset\mathcal{P}$ we have the
relations
$\left.k!d_k\right|_{\mathcal{P}_s}=\left.d^k\right|_{\mathcal{P}_s}$,
therefore $\mathcal{D}(\mathcal{P}_s)$ is isomorphic to the divided power
ring $d_kd_l={k+l\choose k}d_{k+l}$, and
$\left.\Phi(t)\right|_{\mathcal{P}_s}=e^{dt}$.

Let us note that this equality is not valid for
$\mathcal{P}_s\subset\mathcal{RP}$. The problem is that for a simple polytope
$P^n$ we have $d_k=\frac{d^k}{k!}$ for all $1\leqslant k\leqslant n$. But
$d_{n+1}P^n=\varnothing$, while
$$
\frac{d^{n+1}}{(n+1)!}P^n=\frac{d}{n+1}\left(\frac{d^n}{n!}P^n\right)=\frac{d}{n+1}\left(d_nP^n\right)=\frac{d}{n+1}f_0(P^n)\pt=\frac{f_0(P^n)}{n+1}\varnothing
$$
\end{exam}

Another example of operators in $\mathcal{L}(\R)$ is given by the
multiplication by elements of $\R$:
$$
[P](Q)=PQ
$$
\begin{prop}
The following relation holds
$$
d_k[P]=\sum\limits_{i=0}^k[d_iP]d_{k-i}
$$
\end{prop}
\begin{proof}
$$
d_k(PQ)=\sum_{i=0}^{k}(d_iP)d_{k-i}Q
$$
\end{proof}
Thus any operator in the ring $\R\mathcal{D}(\R)$ generated by
$\mathcal{D}(\R)$ and $\{[P],P\in\R\}$ can be expressed as a sum of operators
$$
\sum\limits_{P,\,\omega}[P]D_{\omega},\quad P\in\R,\; D_{\omega}\in\mathcal{D}(\R).
$$

Let us remind that $[\pt]=x$ is a cone operator in
$\mathcal{L}(\mathcal{RP})$.

\subsection{Flag $f$-vectors}
\begin{defin}
Let $P^n$ be an $n$-dimensional polytope and
$S=\{a_1,\dots,a_k\}\subset\{0,1,\dots,n-1\}$.\\
A {\itshape flag number} $f_{S}=f_{a_1,\,\dots,\,a_k}$ is the number of
increasing sequences of faces
$$
F^{a_1}\subset F^{a_2}\subset\dots\subset F^{a_k},\quad\dim F^{a_i}=a_i.
$$
It is easy to see that for $S=\{i\}$ the number $f_{\{i\}}=f_i$ is just the
number of $i$-dimensional faces. We have already defined this number above.
Let us denote $l(S)=k$. Then $0\leqslant k\leqslant n$.

The collection $\{f_{S}\}$ of all the flag numbers is called a {\itshape flag
$f$-vector} (or {\itshape extended $f$-vector}) of the polytope $P^n$. By the
definition $f_{\varnothing}=1$.
\end{defin}
\begin{rem}
For the set $\{a_1,\,\dots,\,a_k\}\subset \{0,1,\dots,n-1\}$ let us denote
$a_0=-1$, and $a_{k+1}=n$. Then
$$
f_{-1,\,a_1,\,\dots,\,a_k}=f_{-1,\,a_1,\,\dots,\,a_k,\,n}=f_{a_1,\,\dots,\,a_k,\,n}=f_{a_1,\,\dots,\,a_k}
$$
This notation corresponds to the fact that $\varnothing$ is the only face of
$P$ of dimension $-1$, and any other face contains $\varnothing$, while the
polytope $P$ itself is the only face of dimension $n$, and it contains any
other face.
\end{rem}

Flag $f$-vectors have been extensively studied by M.~Bayer and L.~Billera in
\cite{BB}, where the generalized Dehn-Sommerville relations were proved:
\begin{thm*}[\cite{BB}, Theorem 2.1]
Let $P^n$ be an $n$-dimensional polytope, and $S\subset\{0,\dots,n-1\}$. If
$\{i,k\}\subseteq S\cup\{-1,n\}$ such that $i<k-1$ and
$S\cap\{i+1,\dots,k-1\}=\varnothing$, then
\begin{equation}\label{BBR}
\sum\limits_{j=i+1}^{k-1}(-1)^{j-i-1}f_{S\cup\{j\}}=(1-(-1)^{k-i-1})f_S.
\end{equation}
\end{thm*}
\begin{defin}
For $n\geqslant 1$ let $\Psi^n$ be the set of subsets
$S\subset\{0,1,\dots,n-2\}$ such that $S$ contains no two consecutive
integers.
\end{defin}
It easy to show by induction that the cardinality of $\Psi^n$ is equal to the
$n$-th Fibonacci number $c_n$ ($c_n=c_{n-1}+c_{n-2}$, $c_0=1,\;c_1=1$).

Let us remind that for any polytope $P$ we have defined  a {\itshape cone}
(or a {\itshape pyramid}) $CP$ and a {\itshape bipyramid} (or a {\itshape
suspension}) $BP$. These two operations are defined on combinatorial
polytopes and can be extended to linear operators on the rings $\mathcal{P}$
and $\mathcal{RP}$. By definition $B\varnothing=\pt=C\varnothing$.

The face lattice of the polytope $CP^n$ is:
$$
\{\varnothing\};\;\{F^0_i,
C\varnothing\};\;\{F^1_i,CF^0_j\};\;\dots;\;\{F^{n-1}_i,CF^{n-2}_j\};\;\{P^n,CF^{n-1}_i\};\{CP^n\},
$$
where $F_i^k$ are $k$-dimensional faces of $P^n$, and
\begin{gather*}
F^{a_1}\subset F^{a_2}\;\mbox{in $CP$} \;\Leftrightarrow\; F^{a_1}\subset
F^{a_2}\;\mbox{in $P$},\\
F^{a_1}\subset CF^{a_2}\;\Leftrightarrow \;F^{a_1}\subset
F^{a_2},\\
CF^{a_1}\subset CF^{a_2}\;\Leftrightarrow\;F^{a_1}\subset F^{a_2}.
\end{gather*}

The polytope $BP^n$ has the face lattice:
$$
\{\varnothing\},\{C_{-1}\varnothing, C_1\varnothing
,F^0_i\},\{F^1_i,C_{-1}F^0_j,C_1F^0_k\},\dots,\{F_k^{n-1},C_{-1}F^{n-2}_j,
C_1F^{n-2}_k\},\{C_{-1}F^{n-1}_i,C_1F^{n-1}_j\},\{BP^n\}.
$$
where $C_{-1}$ and $C_1$ are the lower and the upper cones, and
\begin{gather*}
F^{a_1}\subset F^{a_2}\;\mbox{in $BP$}\Leftrightarrow\; F^{a_1}\subset F^{a_2}\;\mbox{in $P$},\\
F^{a_1}\subset C_sF^{a_2}\;\Leftrightarrow F^{a_1}\;\subset F^{a_2};\\
C_sF^{a_1}\subset C_tF^{a_2}\;\Leftrightarrow\; s=t\mbox{
and}\;F^{a_1}\subset F^{a_2};
\end{gather*}
\begin{defin}
For $n\geqslant 1$ let  $\Omega^n$ be the set of $n$-dimensional polytopes
that arises when we apply  words in $B$ and $C$ that end in $C^2$ and contain
no adjacent $B$'s to the empty set $\varnothing$.
\end{defin}
Each word of length $n+1$ from the set $\Omega^n$ either has the form $CQ$,
$Q\in\Omega^{n-1}$, or $BCQ$, $Q\in\Omega^{n-2}$, so cardinality of the set
$\Omega^n$ satisfies the Fibonacci relation
$|\Omega^n|=|\Omega^{n-1}|+|\Omega^{n-2}|$. Since $|\Omega^1|=|\{C^2\}|=1$,
and $|\Omega^2|=|\{C^3,BC^2\}|=2$, we see that $|\Omega^n|=c_n=|\Psi^n|$.

M.~Bayer and L.~Billera proved the following fact:
\begin{thm*}
Let $n\geqslant 1$. Then
\begin{enumerate}
\item For all $T\subseteq\{0,1,\dots,n-1\}$ there is a nontrivial linear
    relation expressing $f_T(P)$ in terms of $f_S(P),\;S\in\Psi^n$, which
    holds for all $n$-dimensional polytopes (see \cite[Proposition
    2.2]{BB}).
\item The extended $f$-vectors of the $c_n$ elements of $\Omega^n$ are
    affinely independent \cite[Proposition 2.3]{BB}. Thus the flag
    $f$-vectors
$$
\{f_S(P^n)\}_{S\in\Psi^n}:P^n\mbox{ -- an $n$-dimensional
polytope}
$$
span an $(c_n-1)$-dimensional affine hyperplane defined by the equation
$f_{\varnothing}=1$ (see \cite[Theorem 2.6]{BB})
\end{enumerate}
\end{thm*}

\begin{rem} The first part of the theorem follows from the generalized
Dehn-Sommerville relations (\ref{BBR}). In fact, we have
$f_{S,\,n-2,\,n-1}=2f_{S,\,n-2}$, and
$$
f_{S,\,i,\,n-1}=(1+(-1)^{n-i})f_{S,\,i}+(-1)^{n-i-1}\sum\limits_{j=i+1}^{n-2}(-1)^{j-i-1}f_{S,\,i,\,j},\quad i<n-2.
$$
Therefore, we can get rid of the index $n-1$. Then
$f_{S_1,\,i-1,\,i,\,i+1,\,S_2}=2f_{S_1,\,i-1,\,i+1,\,S_2}$. Therefore, we can
get rid of triples of consecutive indices.

Now for each set $S$ in the sum we obtain let us take the last pair of
consecutive indices $\{i,i+1\}$. Then $f_S=f_{S_1,\,j,\,i,\,i+1,\,S_2}$.
$$
f_{S_1,\,j,\,i,\,i+1,\,S_2}=(1+(-1)^{i+1-j})f_{S_1,\,j,\,i+1,\,S_2}+(-1)^{i-j}\sum\limits_{k=j+1}^{i-1}(-1)^{k-j-1}f_{S_1,\,j,\,k,\,i+1,\,S_2}
$$
Then the last pair of consecutive indices for each set in the sum consists of
smaller numbers. Then after the sequence if such steps and removing the
triples of consecutive indices on each step we will obtain flag numbers of
the form $f_{i,\,i+1,\,S_2}$, where the set $\{i+1\}\cup S_2$ contains no
consecutive indices. Then
$$
f_{i,\,i+1,\,S_2}=(1+(-1)^i)f_{i+1,\,S_2}-(-1)^i\sum\limits_{j=0}^{i-1}(-1)^jf_{j,\,i+1,\,S_2}
$$
where all the sets $\{j,\,i+1\}\cup S_2$ belong to $\Psi^n$. This finishes
the proof.
\end{rem}

In fact, a little bit more stronger statement, which can be easily extracted
from the original Bayer-Billera's proof, is true.

Let us identify the words in $\Omega^n$ with the sets in $\Psi^n$ in such a
way that the word \linebreak $C^{n+1-a_k}BC^{a_k-a_{k-1}-1}B\dots BC^{a_1-1}$
corresponds to the set $\{a_1-3,\dots,a_k-3\}$. Let us set $C<B$ and order
the words lexicographically. Consider a matrix $K^n$ of sizes $c_n\times
c_n$, $k_{Q,\,S}=f_{S}(Q),\;Q\in\Omega^n,\,S\in \Psi^n, $.
\begin{prop}
$$
\det(K^n)=1.
$$
\end{prop}
\begin{proof}
For $n=1$ the matrix $K^1=(f_{\varnothing}(I))=(1)$.

For $n=2$
$$
K^2=\begin{pmatrix}
1&f_{0}(\Delta^2)\\
1&f_{0}(I^2)\\
\end{pmatrix}=
\begin{pmatrix}
1&3\\
1&4\\
\end{pmatrix};\quad\det K^2=1
$$
Let us prove the statement by induction. Consider the matrix $K^n$:
$$
K^n=\begin{pmatrix}
K_{11}& K_{12}\\
K_{21}& K_{22}\\
\end{pmatrix}
$$
where the block $K_{11}$ corresponds to the words of the form
$CQ,\;Q\in\Omega^{n-1}$ and the sets $S$, that do not contain $n-2$.
Similarly, the block $K_{22}$ corresponds to the words of the form
$BCQ,\;Q\in\Omega^{n-2}$ and the sets containing $n-2$.

Each increasing sequence of faces of the polytope $CP^{n-1}$ without two
faces of adjacent dimensions has the form
$$
F^{l_1}\subset\dots\subset F^{l_i}\subset
CF^{l_{i+1}}\subset\dots\subset CF^{l_k},
$$
so for $\{a_1,\dots,a_k\}\in\Psi^n$
\begin{equation}\label{X2}
f_{a_1,\,\dots,\,a_k}(CP^{n-1})=f_{a_1-1,\,a_2-1,\,\dots,\,a_k-1}(P^n)+\dots+f_{a_1,\,\dots,\,a_i,\,a_{i+1}-1,\,\dots,\,a_k-1}(P^n)+\dots+f_{a_1,\,\dots,\,a_k}(P^n),
\end{equation}
where $f_{-1,\,a_2,\,\dots,\,a_k}=f_{a_2,\,\dots,\,a_k}$.

The sets of the form $\{a_1,\dots,a_l,a_{l+1}-1,\dots,a_k-1\}$ may not belong
to $\Psi^{n-1}$. But we can express the corresponding flag numbers in terms
of $\{f_S,\; S\in\Psi^{n-1}\}$ using the generalized Dehn-Sommerville
relations. Let $a_k<n-2$. Each relation has the form:
$$
f_{S_1,\,i,\,k-1,\,k,\,S_2}=(-1)^{k-i}\left(\left(1-(-1)^{k-i-1}\right)f_{S_1,\,i,\,k,\,S_2}-\sum\limits_{j=i+1}^{k-2}(-1)^{j-i-1}f_{S_1,\,i,\,j,\,k,\,S_2}\right),
$$
where all the sets on the right side are lower than the set on the left side.
Thus for $a_k<n-2$
$$
f_{a_1\,\dots,\,a_k}(CP^n)=f_{a_1,\,\dots,\,a_k}(P^n)+\mbox{lower
summands},
$$
so the matrix $K_{11}$ can be represented as $K^{n-1}T$, where $T$ is an
upper unitriangular matrix. In particular, $\det K_{11}=\det K^{n-1}\cdot\det
T=1$ by induction.

\begin{lemma}
Let $P$ be an $(n-2)$-dimensional polytope, and $S\in\Psi^n$. Then for the
polytopes $BCP$ and $CBP$
$$
f_{S}(BCP)=
\begin{cases}
f_{S}(CBP),&\mbox{if $n-2\notin S$},\\
f_{S}(CBP)+f_{S\setminus\{n-2\}}(P),&\mbox{if $n-2\in S$},
\end{cases}\label{BC-CB}
$$
\end{lemma}
\begin{proof}
The polytope $BCP$ has faces
$$
\{F,\,\varnothing\subseteq F\subseteq P\};\;\{CF,\,\varnothing\subseteq F\subsetneq P\};\;\{C_aF,\,\varnothing\subseteq F\subseteq P\};\;\{C_aCF,\,\varnothing\subseteq F\subsetneq P\};\;BCP,
$$
where $a=\pm1$

Then each increasing sequence of faces corresponding to the set $S$ has one
of the forms
\begin{itemize}
\item $F^{l_1}\subset\dots\subset F^{l_i}$
\item $F^{l_1}\subset\dots\subset F^{l_i}\subset
    CF^{l_{i+1}}\subset\dots\subset CF^{l_j}$,
\item $F^{l_1}\subset\dots\subset F^{l_i}\subset
    C_aF^{l_{i+1}}\subset\dots\subset C_aF^{l_j}$
\item $F^{l_1}\subset\dots \subset F^{l_i}\subset CF^{l_{i+1}}\subset
    \dots\subset CF^{l_j}\subset C_aCF^{l_{j+1}}\subset\dots\subset
    C_aCF^{l_k}$
\item $F^{l_1}\subset\dots\subset F^{l_i}\subset C_aF^{l_{i+1}}\subset
    \dots\subset C_aF^{l_j}\subset C_aCF^{l_{j+1}}\subset\dots\subset
C_aCF^{l_k}$
\item $F^{l_1}\subset\dots\subset F^{l_j}\subset
    C_aCF^{l_{j+1}}\subset\dots\subset C_aCF^{l_k}$
\end{itemize}

The polytope $CBP$ has  faces
$$
\{F,\,\varnothing\subseteq F\subsetneq P\};\;\{C_aF,\,\varnothing\subseteq F\subsetneq P\};\;BP;\{CF,\,\varnothing\subseteq F\subsetneq P\};\;\{CC_aF,\,\varnothing\subseteq F\subsetneq P\};\;\;CBP,
$$
where $a=\pm1$

Then each increasing sequence of faces corresponding to the set $S$ has one
of the forms

\begin{itemize}
\item $F^{l_1}\subset\dots\subset F^{l_i}$, $F^{l_i}\ne P$,
\item $F^{l_1}\subset\dots\subset F^{l_i}\subset
    CF^{l_{i+1}}\subset\dots\subset CF^{l_j}$,
\item $F^{l_1}\subset\dots\subset F^{l_i}\subset
    C_aF^{l_{i+1}}\subset\dots\subset C_aF^{l_j}$
\item $F^{l_1}\subset\dots \subset F^{l_i}\subset CF^{l_{i+1}}\subset
    \dots\subset CF^{l_j}\subset CC_aF^{l_{j+1}}\subset\dots\subset
    C_aCF^{l_k}$
\item $F^{l_1}\subset\dots\subset F^{l_i}\subset C_aF^{l_{i+1}}\subset
    \dots\subset C_aF^{l_j}\subset CC_aF^{l_{j+1}}\subset\dots\subset
CC_aF^{l_k}$
\item $F^{l_1}\subset\dots\subset F^{l_j}\subset
    CC_aF^{l_{j+1}}\subset\dots\subset CC_aF^{l_k}$
\end{itemize}

We can interchange $CC_a$ and $C_aC$  to obtain a one-to-one correspondence
between the increasing sequences containing both $C$ and $C_a$. There is a
natural bijection between sequences containing exactly one of these
operations, as well as between sequences containing neither operations nor
$P$. So $f_S(BCP)=f_S(CBP)$, if $n-2\notin S$.

The only difference appears in the case when the sequence of faces in the
polytope $BCP$ has the form
$$
F^{a_1}\subset\dots\subset F^{a_{k-1}}\subset P
$$
The sequences of this type give exactly $f_{a_1,\,\dots,\,a_{k-1}}(P)$, so
$f_{S}(BCP)=f_{S}(CBP)+f_{S\setminus\{n-2\}}(P)$, if $n-2\in S$.
\end{proof}

Let us consider the polytope $BCQ$ corresponding to one of the lower rows of
$K^n$. If $Q$ starts with $C$, then there is a row $CBQ$ in the upper part of
the matrix. Let us subtract the row $CBQ$ from the row $BCQ$. Then Lemma
\ref{BC-CB} implies that the resulting row is
\begin{equation}\label{X3}
k'_{BCQ,\,S}=
\begin{cases}
0,& \mbox{if }n-2\notin S;\\
f_{S\setminus\{n-2\}}(Q),&\mbox{if } n-2\in S.
\end{cases}
\end{equation}

If $Q$ starts with $B$, then there is no row $CBQ$ in the matrix. But since
$\det K^{n-1}=1$ and all flag $f$-numbers of $(n-1)$-dimensional polytopes
can be expressed in terms of $\{f_{S},\;S\in\Psi^{n-1}\}$, the flag
$f$-vector of any $(n-1)$-dimensional polytope is an integer combination of
flag $f$-vectors of the polytopes from $\Omega^{n-1}$. So the flag $f$-vector
of the polytope $BQ$ can be expressed as an integer combination of flag
$f$-vectors of the polytopes $Q'\in\Omega^{n-1}$:
$$
f(BQ)=\sum\limits_{Q'}n_{Q'}f(Q')
$$
Using Formula (\ref{X2}) we obtain:
$$
f(CBQ)=\sum\limits_{Q'}n_{Q'}f(CQ')
$$
If we subtract the corresponding integer combination of the rows of the upper
part of the matrix from the row $BCQ$, we obtain the row (\ref{X3}).

Thus we see, that using elementary transformations of rows the matrix
$K^{n-1}$ can be transformed to the matrix
$$
\begin{pmatrix}
K^{n-1}T&K_{12}\\
0&K^{n-2}
\end{pmatrix}
$$
By the inductive assumption $\det K^{n-1}=\det K^{n-2}=1$, so $\det K^n=1$.
\end{proof}
\begin{rem}
In the proof we follow the original Bayer and Billera's idea, except for the
fact that they use the additional matrix of face numbers of the polytopes
$Q\in\Psi^n$, and our proof is direct.
\end{rem}
\begin{cor}
The flag $f$-vector of any $n$-dimensional polytope $P^n$ is an integer
combination of the flag $f$-vectors of the polytopes $Q\in\Omega^n$.
\end{cor}
\begin{proof}
Indeed, any flag $f$-number is a linear combination of $f_S,\;S\in\Psi^n$.
Since $\det K^n=1$, the vector $\{f_S(P),\;S\in\Psi^n\}$ is an integer
combination of the vectors $\{f_S(Q),\;S\in\Psi^n\},\;Q\in\Omega^n$.

This implies that the whole vector $\{f_S(P)\}$ is an integer combination of
the vectors $\{f_S(Q)\},\;Q\in\Omega^n$ with the same coefficients.
\end{proof}
This fact is  important, when we try to describe the image of the generalized
$f$-polynomial in the ring of quasi-symmetric functions (see below).

Generalized $f$-vectors are also connected with a very interesting
construction of a $cd$-index invented by J.~Fine (see \cite[Prop. 2]{BK}, see
also the papers by R.~Stanley \cite{St3} and M.~Bayer, A.~Klapper \cite{BK}).

\section{Hopf algebras}
In this part we follow mainly the notations of \cite{H}. See also \cite{BR}
and \cite{CFL}.

\subsection{Quasi-symmetric functions}
\begin{defin}
A {\itshape composition} $\omega$ of a number $n$ is an ordered set
$\omega=(j_1,\dots,j_k),\;j_i\geqslant 1$, such that $n=j_1+\dots+j_k$. Let
us denote $|\omega|=n$, $l(\omega)=k$.

Let us denote by $()$ the empty composition of the number $0$. Then $|()|=0$,
$l\left(()\right)=0$.
\end{defin}

\begin{defin} Let $t_1,t_2,\dots$ be a finite or an infinite set of
variables, $\deg t_i=2$. For a composition $\omega=(j_1,\dots,j_k)$ consider
a {\itshape quasi-symmetric monomial}
$$
M_{\omega}=\sum\limits_{l_1<\dots<l_k}t_{l_1}^{j_1}\dots
t_{l_k}^{j_k},\quad M_{()}=1.
$$
\end{defin}
Degree of the monomial $M_{\omega}$ is equal to $2|\omega|=2(j_1+\dots+j_k)$.

For any two monomials $M_{\omega'}$ and $M_{\omega''}$ their product in the
ring of polynomials $\mathbb Z[t_1,t_2,\dots]$ is equal to
$$
M_{\omega'}M_{\omega''}=\sum\limits_{\omega}\left(\sum\limits_{\Omega'+\Omega''=\omega}1\right)M_{\omega},
$$
where for the compositions $\omega=(j_1,\dots,j_k)$,
$\omega'=(j_1',\dots,j_{l'}')$, $\omega''=(j_1'',\dots,j_{l''}'')$, $\Omega'$
and $\Omega''$ are all the $k$-tuples such that
$$
\Omega'=(0,\dots,j_1',\dots,0,\dots,j_{l'}'\dots,0),\quad \Omega''=(0,\dots,j_1'',\dots,0,\dots,j_{l''}''\dots,0),
$$

This multiplication rule of compositions is called {\itshape the overlapping
shuffle multiplication}.

For example,
\begin{enumerate}
\item
$$
M_{(1)}M_{(1)}=\left(\sum\limits_{i}t_i\right)\left(\sum\limits_{j}t_j\right)=\sum\limits_{i}t_i^2+2\sum\limits_{i<j}t_it_j=M_{(2)}+2M_{(1,\,1)}.
$$
This corresponds to the decompositions
$$
(2)=(1)+(1),\; (1,1)=(1,0)+(0,1)=(0,1)+(1,0).
$$
\item
\begin{multline*}
M_{(1)}M_{(1,\,1)}=\left(\sum\limits_{i}t_i\right)\left(\sum\limits_{j<k}t_jt_k\right)=\sum\limits_{i<j}t_i^2t_j+\sum\limits_{i<j}t_it_j^2+3\sum\limits_{i<j<k}t_it_jt_k=\\
=M_{(2,\,1)}+M_{(1,\,2)}+3M_{(1,\,1,\,1)}.
\end{multline*}
This corresponds to the decompositions
\begin{gather*}
(2,1)=(1,0)+(1,1),\;(1,2)=(1,0)+(1,1),\\
(1,1,1)=(1,0,0)+(0,1,1)=(0,1,0)+(1,0,1)=(0,0,1)+(1,1,0).
\end{gather*}
\item
\begin{multline*}
M_{(1,\,1)}M_{(1,\,1)}=\left(\sum\limits_{i<j}t_it_j\right)\left(\sum\limits_{k<l}t_kt_l\right)=\sum\limits_{i<j}t_i^2t_j^2+2\sum\limits_{i<j<k}t_i^2t_jt_k+\\
+2\sum\limits_{i<j<k}t_it_j^2t_k+2\sum\limits_{i<j<k}t_it_jt_k^2+6\sum\limits_{i<j<k<l}t_it_jt_kt_l=\\
=M_{(2,\,2)}+2M_{(2,\,1,\,1)}+2M_{(1,\,2,\,1)}+2M_{(1,\,1,\,2)}+6M_{(1,\,1,\,1,\,1)}.
\end{multline*}
This corresponds to the decompositions
\begin{gather*}
(2,2)=(1,1)+(1,1),\;(2,1,1)=(1,1,0)+(1,0,1)=(1,0,1)+(1,1,0),\\
(1,2,1)=(1,1,0)+(0,1,1)=(0,1,1)+(1,1,0),\;(1,1,2)=(1,0,1)+(0,1,1)=(0,1,1)+(1,0,1),\\
(1,1,1,1)=(1,1,0,0)+(0,0,1,1)=(1,0,1,0)+(0,1,0,1)=(1,0,0,1)+(0,1,1,0)=\\
=(0,1,1,0)+(1,0,0,1)=(0,1,0,1)+(1,0,1,0)=(0,0,1,1)+(1,1,0,0).
\end{gather*}
\end{enumerate}

Thus finite integer combinations of quasi-symmetric monomials form a ring.
This ring is called {\itshape a ring of quasi-symmetric functions} and is
denoted by $\Qsym[t_1,\dots,t_n]$, where $n$ is the number of variables. In
the case of an infinite number of variables it is denoted by
$\Qsym[t_1,t_2,\dots]$ or $\Qs$.

The diagonal mapping $\Delta\colon\Qs\to\Qs\otimes\Qs$
$$
\Delta M_{(a_1,\,\dots,\,a_k)}=\sum\limits_{i=0}^kM_{(a_1,\,\dots,\,a_i)}\otimes M_{(a_{i+1},\,\dots,\,a_k)}
$$
defines on $\Qs$ the structure of a graded Hopf algebra.

\begin{prop}\label{Crn}
A polynomial $g\in\mathbb Z[t_1,\dots,t_r]$ is a finite linear combination of
quasi-symmetric monomials if and only if
$$
g(0,t_1,t_2,\dots,t_{r-1})=g(t_1,0,t_2,\dots,t_{r-1})=\dots=g(t_1,\dots,t_{r-1},0)
$$
\end{prop}
\begin{proof}
For the quasi-symmetric monomial $M_{\omega}$ we have
$$
M_{\omega}(t_1,\dots,t_i,0,t_{i+1},\dots,t_{r-1})=
\begin{cases}
0,&\omega=(j_1,\dots,j_r);\\
M_{\omega}(t_1,\dots,t_i,t_{i+1},\dots,t_{r-1}),&\omega=(j_1,\dots,j_k),\;k<r.
\end{cases}
$$
So this property is true for all quasi-symmetric functions.

On the other hand, let the condition of the proposition be true. Let us prove
that for a fixed composition $\omega$ any two monomials $t_{l_1}^{j_1}\dots
t_{l_k}^{j_k}$ and $t_{l_1'}^{j_1}\dots t_{l_k'}^{j_k}$ have the equal
coefficients $g_{l_1,\,\dots,\,l_k}^{j_1,\,\dots,\,j_k}$ and
$g_{l_1',\,\dots,\,l_l'}^{j_1,\,\dots,\,j_k}$.

Let $l_i+1<l_{i+1}$ or $i=k$ and $l_i<r$. Consider the corresponding
coefficients in the polynomial equation:
$$
g(t_1,\dots,t_{l_i-1},0,t_{l_i},t_{l_i+1},\dots,t_{r-1})=g(t_1,\dots,t_{l_i-1},t_{l_i},0,t_{l_i+1},\dots,t_{r-1}).
$$
On the left the monomial $t_{l_1}^{j_1}\dots
t_{l_i}^{j_i}t_{l_{i+1}-1}^{j_{i+1}}\dots t_{l_k-1}^{j_k}$ has the
coefficient
$g_{l_1,\,\dots,\,l_i+1,\,l_{i+1},\dots,\,l_k}^{j_1,\,\dots,\,j_k}$, and on
the right $g_{l_1,\,\dots,\,l_i,\,l_{i+1},\,\dots,\,l_k}^{j_1,\dots,j_k}$, so
they are equal. Now we can move the index $l_i$ to the right to $l_i+1$ and
in the same manner we can move $l_i$ to the left, if $l_{i-1}<l_i-1$, or
$i=1$ and $l_i>1$.

Now let us move step by step the index $l_1$ to $1$, then the index $l_2$ to
$2$, and so on. At last we obtain that
$g_{l_1,\,\dots,\,l_k}^{j_1,\,\dots,\,j_k}=g_{1,\,\dots,\,k}^{j_1,\,\dots,\,j_k}$.
\end{proof}
\begin{rem}\label{cr}
Similar argument shows that the same is true in the case of an infinite
number of variables, if we consider all the expressions
$$
a+\sum\limits_{k=1}^{\infty}\sum\limits_{1\leqslant
l_1<\dots<l_k}\sum\limits_{j_1,\,\dots,\,j_k\geqslant 1}a_{
l_1,\,\dots,\,l_k}^{j_1,\,\dots,\,j_k}t_{l_1}^{j_1}\dots
t_{l_k}^{j_k}.
$$
of bounded degree: $2|\omega|=2(j_1+\dots+j_k)<N$ for all $\omega$. Then the
infinite series $g$ of bounded degree belongs to $\Qs$ if and only if
$$
g(t_1,\dots,t_{i-1},0,t_{i+1},\dots)=g(t_1,\dots,t_{i-1},t_{i+1},\dots)
$$
for all $i\geqslant 1$. As a corollary we obtain another proof of the fact
that quasi-symmetric functions form a ring.
\end{rem}

In \cite{H} M.~Hazewinkel proved the Ditters conjecture that
$\Qsym[t_1,t_2,\dots]$ is a free commutative algebra of polynomials over the
integers.

The $(2n)$-th graded component $\Qsym^{2n}[t_1,t_2,\dots]$ has rank
$2^{n-1}$.

The numbers $\beta_i$ of the multiplicative generators of degree $2i$ can be
found by a recursive relation:
$$
\frac{1-t}{1-2t}=\prod\limits_{i=1}^{\infty}\frac{1}{(1-t^i)^{\beta_i}}
$$

There are group homomorphisms:

\begin{gather*}
V_{r+1}:\Qsym[t_1,\dots,t_r]\to\Qsym[t_1,\dots,t_{r+1}]:\quad
V_{r+1}M_{\omega}(t_1,\dots,t_r)=M_{\omega}(t_1,\dots,t_{r+1})\\
E_r:\Qsym[t_1,\dots,t_{r+1}]\to\Qsym[t_1,\dots,t_r]:\quad t_{r+1}\to 0;
\end{gather*}
The mapping $E_r$ sends all the monomials $M_{\omega}$ corresponding to the
compositions $(j_1,\dots,j_{r+1})$ of length $r+1$ to zero, and for the
compositions $\omega$ of smaller length
$E_rM_{\omega}(t_1,\dots,t_{r+1})=M_{\omega}(t_1,\dots,t_r)$. It is a ring
homomorphism.

It is easy to see that $E_rV_{r+1}$ is the identity map of the ring
$\Qsym[t_1,\dots,t_r]$.

Given $r>0$ there is a projection $\Pi_{\Qsym}:\mathbb
Z[t_1,\dots,t_r]\to\Qsym[t_1,\dots,t_r]\otimes\mathbb Q$:\\
for $s_1<s_2<\dots<s_k$
$$
\Pi_{\Qsym} t_{s_1}^{j_1}\dots
t_{s_k}^{j_k}=\frac{1}{{r\choose k}}\left(\sum\limits_{l_1<\dots<l_k}t_{l_1}^{j_1}\dots
t_{l_k}^{j_k}\right),
$$
which gives the average value over all the monomials of each type.

Then for the quasi-symmetric monomial $M_{\omega},\;\omega=(j_1,\dots,j_k)$
we have:
$$
\Pi_{\Qsym} M_{\omega}=\sum\limits_{1\leqslant j_1<\dots<j_k\leqslant
r}\frac{1}{{r\choose k}}M_{\omega}={r\choose k}\frac{1}{{r\choose k}}M_{\omega}=M_{\omega},
$$
So $\Pi_{\Qsym}$ is indeed a projection.
\begin{rem} We see that in the theory of symmetric and quasi-symmetric
functions there is an important additional graduation -- the number of
variables in the polynomial.
\end{rem}

\begin{defin}
For a composition $\omega=(j_1,\dots,j_k)$ let us define the composition
$\omega^*=(j_k,\dots,j_1)$.
\end{defin}
The correspondence $M_{\omega}\to (M_{\omega})^*=M_{\omega^*}$ defines an
involutory ring homomorphism
$$
*:\Qsym[t_1,t_2,\dots]\to \Qsym[t_1,t_2,\dots].
$$

\begin{rem}\label{DQsym}
Let $\tau_{\Qs}\colon\Qs\otimes\Qs\to\Qs\otimes\Qs$ be the ring homomorphism
that interchanges the tensor factors $\tau_{\Qs}(x\otimes y)=y\otimes x$.
Then $\Delta(M_{\omega}^*)=*\otimes*(\tau_{\Qs}\Delta M_{\omega})$.
\end{rem}

\subsection{Leibnitz-Hopf Algebras}
Let $R$ be a commutative associative ring with unity.
\begin{defin}
A {\itshape Leibnitz-Hopf algebra} over the ring $R$ is an associative Hopf
algebra $\mathcal{H}$ over the ring $R$  with a fixed sequence of a finite or
countable number of multiplicative generators $H_i$, $i=1,2\dots$ satisfying
the comultiplication formula
$$
\Delta H_n=\sum\limits_{i+j=n}H_i\otimes H_j,\quad H_0=1.
$$

A {\itshape universal Leibnitz-Hopf algebra} $\mathcal{A}$ over the ring $R$
is a Leibnitz-Hopf algebra with the universal property: for any Leibnitz-Hopf
algebra $\mathcal{H}$ over the ring $R$ the correspondence $A_i\to H_i$
defines a Hopf algebra homomorphism.
\end{defin}

Consider the free associative Leibnitz-Hopf algebra over the integers
$\mathcal{Z}=\mathbb Z \langle Z_1,Z_2,\dots\rangle$ in countably many
generators $Z_i$.

\begin{prop}
$\mathcal{Z}\otimes R$ is a universal Leibnitz-Hopf algebra over the ring
$R$.
\end{prop}

Set $\deg Z_i=2i$. Let us denote by $\mathcal{M}$ the graded dual Hopf
algebra over the integers.

It is not difficult to see that $\mathcal M$ is precisely the algebra of
quasi-symmetric functions over the integers. Indeed, for any composition
$\omega=(j_1,\dots,j_k)$ we can define $m_{\omega}$ by the dual basis formula
$$
\langle m_{\omega},Z_{\sigma}\rangle=\delta_{\omega,\,\sigma},
$$
where $Z_{\sigma}=Z_{a_1}\dots Z_{a_l}$ for a composition
$\sigma=(a_1,\dots,a_l)$. Then the elements $m_{\omega}$ are multiplied
exactly as the quasi-symmetric monomials $M_{\omega}$.

Let us denote
$$
\Phi(t)=1+Z_1t+Z_2t^2+\dots=\sum\limits_{k=0}^{\infty}Z_kt^k
$$
Then the comultiplication formula is equivalent to
$$
\Delta\Phi(t)=\Phi(t)\otimes\Phi(t)
$$

Let $\chi:\mathcal{Z}\to\mathcal{Z}$ be the {\itshape antipode}, that is a
linear operator, satisfying the property
$$
1\star\chi=\mu\circ(1\otimes\chi)\circ\Delta=\eta\circ\varepsilon=\mu\circ(\chi\otimes
1)\circ\Delta=\chi\star 1,
$$
where $\varepsilon:\mathcal{Z}\to\mathbb Z$, $\varepsilon(1)=1$,
$\varepsilon(Z_{\sigma})=0$, $\sigma\ne\varnothing$, is a {\itshape counit},
$\eta:\mathbb Z\to\mathcal{Z}$, $\eta(a)=a\cdot 1$ is a {\itshape unite} map,
and $\mu$ is a multiplication in $\mathcal{Z}$.

Then $\Phi(t)\chi(\Phi(t))=1=\chi(\Phi(t))\Phi(t)$ and $\{\chi(Z_n)\}$
satisfy the recurrent formulas
\begin{equation}\label{X1}
\chi(Z_1)=-Z_1;\quad
\chi(Z_{n+1})+\chi(Z_n)Z_1+\dots+\chi(Z_1)Z_n+Z_{n+1}=0,\;n\geqslant
1.
\end{equation}

In fact, since $\Phi(t)=1+Z_1t+\dots$, we can use the previous formula to
obtain
\begin{gather*}
\chi(\Phi(t))=\frac{1}{\Phi(t)}=\sum\limits_{i=0}^{\infty}(-1)^i\left(\Phi(t)-1\right)^i;\\
\chi(Z_n)=\sum\limits_{k=1}^n(-1)^k\sum\limits_{j_1+\dots+j_k=n}Z_{j_1}\dots
Z_{j_k}\quad j_i\geqslant 1,\;n\geqslant 1
\end{gather*}
This formula defines $\chi$ on generators, and therefore on the whole
algebra.

\begin{defin}
Let $\mathcal{Z}^{op}$ be the Hopf algebra opposite to $\mathcal{Z}$, i.e.
$\mathcal{Z}=\mathcal{Z}^{op}$ as coalgebras, and the multiplication in
$\mathcal{Z}^{op}$ is given by the rule $a\diamond b:=b\cdot a$. Then the
antiisomorphism
$$
\varrho\colon\mathcal{Z}\to\mathcal{Z},\quad Z_{a_1}\dots Z_{a_k}\to Z_{a_k}\dots Z_{a_1}
$$
satisfies the property
$\Delta\circ\varrho=(\varrho\otimes\varrho)\circ\Delta$, and defines the Hopf
algebra isomorphism $\mathcal{Z}\to\mathcal{Z}^{op}$.
\end{defin}
\begin{prop}
We have $\varrho^*=*$
\end{prop}
\begin{proof}
Indeed,
$$
\langle\varrho^*M_{\omega},Z_{\sigma}\rangle=\langle M_{\omega},\varrho Z_{\sigma}\rangle=\langle M_{\omega}, Z_{\sigma^*}\rangle=\delta_{\omega,\,\sigma^*}.
$$
Therefore, $\varrho^*M_{\omega}=M_{\omega^*}=M_{\omega}^*$
\end{proof}

\begin{defin}[\cite{N}]
A (left) \emph{Milnor module} $M$ over the Hopf algebra $X$ is an algebra
with unit $1\in R$ which is also a (left) module over $X$ satisfying
\[ x(uv) = \sum x'_n(u)x''_n(v), \quad x\in X,\; u,v \in M, \;\, \Delta x = \sum x'_n \otimes x''_n. \]
\end{defin}

\begin{prop}
The homomorphism
$\EuScript{L}_{\R}\colon\mathcal{Z}\to\mathcal{D}(\R)\;\colon\; Z_k\to d_k$
defines on the rings $\mathcal{P}$ and $\mathcal{RP}$ the structures of left
Milnor modules over the Leibnitz-Hopf algebra $\mathcal{Z}$.

The homomorphism
$\EuScript{R}_{\R}\colon\mathcal{Z}^{op}\to\mathcal{D}(\R)\;\colon\;
\EuScript{R}_{\R}=\EuScript{L}_{\R}\circ \varrho$ defines on the rings
$\mathcal{P}$ and $\mathcal{RP}$ the structures of right Milnor modules over
the Leibnitz-Hopf algebra $\mathcal{Z}$.
\end{prop}
\begin{proof}
This follows from Proposition \ref{MM}
\end{proof}

Let us note, that the word $Z_{\omega}=Z_{j_1}\dots Z_{j_k}$ corresponds
under the homomorphism $\EuScript{L}_{\R}$ to the operator
$D_{\omega}=d_{j_1}\dots d_{j_k}$, and under the homomorphism
$\EuScript{R}_{\R}$ to the operator $D_{\omega^*}=d_{j_k}\dots d_{j_1}$.

\begin{defin}
Denote by $\mathcal{U}$ the universal Hopf algebra in the category of
Leibnitz-Hopf algebras with an antipode $\chi(H_i)=(-1)^iH_i$. It is is easy
to see that $\mathcal{U}=\mathcal{Z}/J_{\mathcal{U}}$, where the two-sided
Hopf ideal $J_{\mathcal{U}}$ is generated by the relations
\begin{equation}\label{Phi}
Z_n-Z_1Z_{n-1}+\dots+(-1)^{n-1}Z_{n-1}Z_1+(-1)^nZ_n=0,\;n\geqslant 2
\end{equation}
These relations can be written in the short form as $\Phi(-t)\Phi(t)=1$.

It can be shown that $\mathcal{U}\otimes\mathbb Q\simeq \mathbb Q\langle
Z_1,Z_3,Z_5,\dots\rangle$, and $\mathcal{U}\xrightarrow{\id\otimes
1}\mathcal{U}\otimes\mathbb Q$ is an embedding.
\end{defin}
\begin{cor}
The mappings $\EuScript{L}_{\R}$ and $\EuScript{R}_{\R}$ define on the rings
$\mathcal{P}$ and $\mathcal{RP}$ the structures of left and right Milnor
modules over the Leibnitz-Hopf algebra $\mathcal{U}$.
\end{cor}
\begin{proof}
This follows from Proposition \ref{PM}.
\end{proof}

Since $\varrho(J_{\mathcal{U}})=J_{\mathcal{U}}$, we see that $\varrho$
induces a correctly defined homomorphism
$\varrho\colon\mathcal{U}\to\mathcal{U}$, which gives the isomorphism
$\mathcal{U}\simeq \mathcal{U}^{op}$.

Since the factor map $\mathcal{Z}\to\mathcal{Z}/J_{\mathcal{U}}$ is an
epimorphism, the dual map $\mathcal{U}^*\to\mathcal{Z}^*=\Qs$ is an
embedding. Then the subring $\mathcal{U}^*\subset\Qs$ is invariant under the
involution $\varrho^*=*$.

\begin{defin}
A {\itshape universal commutative Leibnitz-Hopf algebra} $\mathcal{C}=\mathbb
Z[C_1,C_2,\dots]$ is a free commutative polynomial Leibnitz-Hopf algebra in
generators $C_i$ of degree $2i$. We have
$\mathcal{C}=\mathcal{Z}/J_{\mathcal{C}}$, where the ideal $J_{\mathcal{C}}$
is generated by the relations $Z_iZ_j-Z_jZ_i$.

It is a self-dual Hopf algebra and the graded dual Hopf algebra is naturally
isomorphic to the algebra of symmetric functions $\mathbb
Z[\sigma_1,\sigma_2,\dots]=\Sym[t_1,t_2,\dots]\subset\Qsym[t_1,t_2,\dots]$
generated by the symmetric monomials
$$
\sigma_i=M_{\omega_i}=\sum\limits_{l_1<\dots<l_i}t_{l_1}\dots
t_{l_i},
$$
where $\omega_i=\underbrace{(1,\;\dots,\;1)}_{i}$.

The isomorphism $\mathcal{C}\simeq \mathcal{C}^*$ is given by the
correspondence $C_i\to \sigma_i$.
\end{defin}

\begin{defin}
Denote by $\mathcal{US}$ the universal Hopf algebra in the category of
commutative Leibnitz-Hopf algebras with an antipode $\chi(H_i)=(-1)^iH_i$. It
is easy to see that $\mathcal{US}=\mathcal{Z}/J_{\mathcal{US}}$, where the
two-sided ideal $J_{\mathcal{US}}$ is generated by the relations (\ref{Phi}),
and the commutators $Z_iZ_j-Z_jZ_i$. Then the relations (\ref{Phi}) for odd
$n$ are trivial, while for even $n$ we have
$$
2Z_{2k}=2\sum\limits_{i=1}^{k-1}(-1)^{i-1}Z_iZ_{2k-i}+(-1)^{k-1}Z_k^2
$$
It can be shown that $\mathcal{US}\otimes\mathbb Q\simeq \mathbb
Q[Z_1,Z_3,Z_5,\dots]$.
\end{defin}
\subsection{Lie-Hopf Algebras}
\begin{defin}
A {\itshape Lie-Hopf algebra} over the ring $R$ is an associative Hopf
algebra $\mathfrak{L}$ with a fixed sequence of a finite or countable number
of multiplicative generators $L_i$, $i=1,2\dots$ satisfying the
comultiplication formula
$$
\Delta L_i=1\otimes L_i+L_i\otimes 1,\quad L_0=1.
$$
A {\itshape universal Lie-Hopf algebra} over the ring $R$ is a Lie-Hopf
algebra $\mathcal{A}$ satisfying the universal property: for any Lie-Hopf
algebra $\mathfrak{L}$ over the ring $R$ the correspondence $A_i\to L_i$
defines a Hopf algebra homomorphism $\mathcal{A}\to\mathfrak{L}$.

Consider the free associative Lie-Hopf algebra $\mathcal{W}$ over the
integers $\mathbb Z\langle W_1,W_2,\dots\rangle$ in countably many variables
$W_i$.
\begin{prop}
$\mathcal{W}\otimes R$ is a universal Lie-Hopf algebra over the ring $R$.
\end{prop}

Let us set $\deg W_i=2i$. Then the graded dual Hopf algebra is denoted by
$\mathcal{N}$. This is the so-called {\itshape shuffle algebra}.
\end{defin}

The antipode $\chi$ in $\mathcal{W}$ has a very simple form $\chi(W_i)=-W_i$.

\subsection{Lyndon Words}
A well-known theorem in the theory of free Lie algebras (see, for example,
\cite{Re}) states that the algebra $\mathcal{N}\otimes\mathbb Q$ is a
commutative free polynomial algebra in the so-called {\itshape Lyndon words}.
\begin{defin}
Let us denote by $[a_1,\dots,a_n]$ an element of $\mathbb N^*$, that is the
word over $\mathbb N$, consisting of symbols $a_1,\dots,a_n$, $a_i\in\mathbb
N$. Let us order the words from $\mathbb N^*$ lexicographically, where any
symbol is larger than nothing, that is $[a_1,\dots,a_n]>[b_1,\dots,b_m]$ if
and only if there is an $i$ such that $a_1=b_1,\dots,a_{i-1}=b_{i-1},
a_i>b_i, 1\leqslant i\leqslant\min\{m,n\}$, or $n>m$ and
$a_1=b_1,\dots,a_m=b_m$.

A {\itshape proper tail} of a word $[a_1,\dots,a_n]$ is a word of the form
$[a_i,\dots,a_n]$ with $1<i\leqslant n$. (The empty word and one-symbol word
have no proper tails.)

A word is {\itshape Lyndon} if all its proper tails are larger than the word
itself. For example, the words $[1,1,2]$, $[1,2,1,2,2]$, $[1,3,1,5]$ are
Lyndon and the words $[1,1,1,1]$, $[1,2,1,2]$, $[2,1]$ are not Lyndon. The
set of Lyndon words is denoted by $\Lyn$.

The same definitions make sense for any totally ordered set, for example, for
the set $\{1,2\}$ or for the set of all odd positive integers.
\end{defin}

The role of Lyndon words is described by the following theorem:
\begin{thm*}[Chen-Fox-Lyndon Factorization \cite{CFL}] Every word $w$ in $\mathbb N^*$ factors uniquely into a decreasing concatenation product of Lyndon words
$$
w=u_1*u_2*\dots*u_k,\quad u_i\in\Lyn,\;u_1\geqslant u_2\geqslant\dots\geqslant u_k
$$
\end{thm*}
For example, $[1,1,1,1]=[1]*[1]*[1]*[1]$, $[1,2,1,2]=[1,2]*[1,2]$,
$[2,1]=[2]*[1]$.

The algebra $\mathcal{N}$ is additively generated by the words in $\mathbb
N^*$. The word $w=[a_1,\dots,a_n]$ corresponds to the function
$$
\langle [a_1,\dots,a_n], W_{\sigma}\rangle=\delta_{w,\,\sigma}
$$
where $\sigma$ is a composition $\sigma=(b_1,\dots,b_l)$,
$W_{\sigma}=W_{b_1}\dots W_{b_l}$, and for the word $[a_1,\dots,a_n]$
$$
\delta_{w,\,\sigma}=\begin{cases}
1,&\sigma=(a_1,\dots,a_n);\\
0,&\mbox{else}
\end{cases}
$$
The multiplication in the algebra $\mathcal{N}$ is the so-called {\itshape
shuffle multiplication}:
$$
[a_1,\dots,a_n]\times_{sh}[a_{n+1},\dots,a_{m+n}]=\sum\limits_{\sigma}[a_{\sigma(1)},\dots,a_{\sigma(n)},a_{\sigma(n+1)},\dots,a_{\sigma(n+m)}],
$$
where $\sigma$ runs over all the substitutions $\sigma\in S_{m+n}$ such that
$$
\sigma^{-1}(1)<\dots<\sigma^{-1}(n)\mbox{ and }\sigma^{-1}(n+1)<\dots<\sigma^{-1}(n+m).
$$
For example,
\begin{gather*}
[1]\times_{sh}[1]=[1,1]+[1,1]=2[1,1];\\
[1]\times_{sh}[2,3]=[1,2,3]+[2,1,3]+[2,3,1];\\
[1,2]\times_{sh}[1,2]=[1,2,1,2]+[1,1,2,2]+[1,1,2,2]+[1,1,2,2]+[1,1,2,2]+[1,2,1,2]=2[1,2,1,2]+4[1,1,2,2].
\end{gather*}
There is a well-known shuffle algebra structure theorem:
\begin{thm*}
$\mathcal{N}\otimes \mathbb Q=\mathbb Q[\Lyn]$, the free commutative algebra
over $\mathbb Q$ in the symbols from $\Lyn$.
\end{thm*}

The proof follows from the following theorem concerning shuffle products in
connection with Chen-Fox-Lyndon factorization.
\begin{thm*}
Let $w$ be a word on the natural numbers and let $w=u_1*u_2*\dots *u_m$ be
its Chen-Fox-Lyndon factorization. Then all words that occur with nonzero
coefficient in the shuffle product
$u_1\times_{sh}u_2\times_{sh}\dots\times_{sh}u_m$ are lexicographically
smaller or equal to $w$, and $w$ occurs with a nonzero coefficient in this
product.
\end{thm*}

Given this result it is easy to prove the shuffle algebra structure theorem
in the case of arbitrary totally ordered subset
$M=\{m_1,m_2,\dots\}\subset\mathbb N$. Let us denote by $\mathcal{W}_{M}$ the
free associative Lie-Hopf algebra $\mathbb Z\langle
W_{m_1},W_{m_2},\dots\rangle$, let $\mathcal{N}_{M}$ be its graded dual
algebra, and $\Lyn_M$ be the corresponding set of all Lyndon words. We need
to prove that $\mathcal{N}_M\otimes\mathbb Q=\mathbb Q[\Lyn_M]$.

Let $m_1$ be the minimal number in $M$. The smallest word $[m_1]$ is Lyndon.

Given a word $w$ we can assume by induction that all the words
lexicographically smaller than $w$ have been written as polynomials in the
elements of $\Lyn_M$. Take the Chen-Fox-Lyndon factorization
$w=u_1*u_2*\dots*u_m$ of $w$ and consider, using the preceding theorem, $$
u_1\times_{sh}u_2\times_{sh}\dots\times_{sh}u_m=aw+(\mbox{reminder}).
$$
By the theorem the coefficient $a$ is nonzero and all the words in (reminder)
are lexicographically smaller than $w$, hence they belong to $\mathbb
Q[\Lyn_M]$. Therefore $w\in\mathbb Q[\Lyn_M]$. This proves generation. Since
each monomial in $(2n)$-th graded component of $\mathcal{N}\otimes \mathbb Q$
has a unique Chen-Fox-Lyndon decomposition, the number of monomials in
$\mathbb Q[\Lyn_M]$ of graduation $2n$ is equal to the number of monomials of the same graduation in
$\mathcal{N}_{M}\otimes\mathbb Q$, monomials in Lyndon words are linearly
independent. This proves that Lyndon words are algebraically independent.

\begin{cor}
The graded dual algebras to the free associative Lie-Hopf algebras
$\mathcal{W}_{12}=\mathbb Z\langle W_1,W_2\rangle$ and
$\mathcal{W}_{odd}=\mathbb Z\langle W_1,W_3,W_5,\dots\rangle$ are free
polynomial algebras in Lyndon words $\Lyn_{12}$ and $\Lyn_{odd}$
respectively.\label{Poly}
\end{cor}
The correspondence
$$
1+Z_1t+Z_2t^2+\dots+=\exp(W_1t+W_2t^2+\dots);\\
$$
defines an isomorphism of Hopf algebras $\mathcal{Z}\otimes\mathbb
Q\simeq\mathcal{W}\otimes\mathbb Q$.

However it is not true that $\mathcal{N}$ is a free polynomial commutative
algebra over the integers.

\section{Topological realization of Hopf algebras} In \cite{BR} A.~Baker and
B.~Richter showed that the ring of quasi-symmetric functions has a very nice
topological interpretation.

We will consider $CW$-complexes $X$ and their homology $\C_*(X)$ and
cohomology $\C^*(X)$ with integer coefficients.

Let us assume that homology groups $\C_*(X)$ have no torsion.

The diagonal map $X \to X \times X$ defines in $\C_*(X)$ the structure of a
graded coalgebra with the comultiplication $\Delta \colon \C_*(X) \to
\C_*(X)\otimes \C_*(X)$ and the dual structure of a graded algebra in
$\C^*(X)$ with the multiplication $\Delta^* \colon \C^*(X)\otimes \C^*(X) \to
\C^*(X)$.

In the case when $X$ is an $H$-space with the multiplication $\mu\colon
X\times X\to X$ we obtain the Pontryagin product $\mu_*\colon \C_*(X)\otimes
\C_*(X)\to\C_*(X)$ in the coalgebra $\C_*(X)$ and the corresponding structure
of a graded Hopf algebra on $\C_*(X)$. The cohomology ring $\C^*(X)$ obtains
the structure of a graded dual Hopf algebra with the diagonal mapping $\mu^*
\colon \C^*(X) \to \C^*(X)\otimes \C^*(X)$.

For any space $Y$ the loop space $X=\Omega Y$ is an $H$-space. A continuous
mapping $f\colon Y_1\to Y_2$ induces a mapping of $H$-spaces $\Omega f\colon
X_1 \to X_2$, where $X_i=\Omega Y_i$. Thus for any space $Y$ such that
$X=\Omega Y$ has no torsion in integral homology we obtain the Hopf algebra
$\C_*(X)$. This correspondence is functorial, that is any continuous mapping
$f\colon Y_1\to Y_2$ induces a Hopf algebra homomorphism $f_* \colon
\C_*(X_1) \to \C_*(X_2) $

By the Bott-Samelson theorem \cite{BS}, $\C_*(\Omega\Sigma X)$ is the free
associative algebra $T(\widetilde\C_*(X))$ generated by $\widetilde \C_*(X)$.
This construction is functorial, that is a continuous mapping $f:X_1\to X_2$
induces a ring homomorphism of the corresponding tensor algebras arising from
the mapping $f_*\colon\widetilde \C_*(X_1)\to\widetilde\C_*(X_2)$. Denote
elements of $\C_*(\Omega\Sigma X)$ by $(a_1|\ldots |a_n)$, where $a_i\in
\widetilde \C_*(X)$. Since the diagonal mapping $\Sigma X \to \Sigma X \times
\Sigma X$ gives the $H$-map $\Delta \colon \Omega\Sigma X \to \Omega\Sigma X
\times \Omega\Sigma X$, it follows from the Eilenberg-Zilber theorem that
\[ \Delta_*(a_1|\ldots |a_n) = (\Delta_* a_1|\ldots |\Delta_* a_n)  \]
where $(a_1\otimes b_1 | a_2\otimes b_2) = (a_1|a_2)\otimes(b_1|b_2)$.

There is a nice combinatorial model for any topological space of the form
$\Omega\Sigma X$ with $X$ connected, namely the James construction $JX$ on
$X$. After one suspension this gives rise to a splitting
\[ \Sigma\Omega\Sigma X \sim \Sigma JX \sim\bigvee_{n\geqslant1}\Sigma X^{(n)}, \]
where $X^{(n)}$ denotes the $n$-fold smash power of $X$.

\begin{exam} There exists a homotopy equivalence
 \[ \Sigma\Omega\Sigma S^2 \to \bigvee\limits_{n\geqslant1} \Sigma(S^2)^{(n)}\simeq \Sigma
\Big(\bigvee\limits_{n\geqslant1} S^{2n}\Big). \]  Therefore there exists an $H$-map \, $\Omega \Sigma(\Omega\Sigma S^2)
\to \Omega \Sigma \Big(\bigvee\limits_{n\geqslant1} S^{2n}\Big)$ that is a
homotopy equivalence.
\end{exam}

Using these classical topological results let us describe topological
realizations of the Hopf algebras we study in this work.

We have the following Hopf algebras:
\begin{itemize}
\item[\textbf{I.}]
\begin{enumerate}
\item $\Coh_*(\mathbb CP^{\infty})$ is a divided power algebra
    $\mathbb Z[u_1,u_2,\dots]/I$, where the ideal $I$ is generated by
    the relations $u_iu_j-{i+j\choose i}u_{i+j}$, with the
    comultiplication
\begin{equation}\label{F-1}
\Delta u_n=\sum_{k=0}^{n}u_k\otimes u_{n-k};
\end{equation}
\item $\Coh^*(\mathbb CP^{\infty})=\mathbb Z[u]$ -- a polynomial ring
    with the comultiplication
$$
\Delta u=1\otimes u+u\otimes 1;
$$
\end{enumerate}

\item[\textbf{II.}]
\begin{enumerate}
\item $\Coh_*(\Omega\Sigma\mathbb CP^{\infty})\simeq
    T(\tilde\Coh_*(\mathbb CP^{\infty}))=\mathbb Z\langle
    u_1,u_2,\dots\rangle$ with $u_i$ being non-commuting variables of
    degree $2i$. Thus there is an isomorphism of rings
$$
\Coh_*(\Omega\Sigma\mathbb CP^{\infty})\simeq\mathcal{Z}
$$
under which $u_n$ corresponds to $Z_n$.

The coproduct $\Delta$ on $H_*(\Omega\Sigma\mathbb CP^{\infty})$
induced by the diagonal in $\Omega\Sigma\mathbb CP^{\infty}$ is
compatible with the one in $\mathcal{Z}$:
$$
\Delta u_n=\sum\limits_{i+j=n}u_i\otimes u_j
$$
So there is an isomorphism of graded Hopf algebras. This gives a
geometric interpretation for the antipode $\chi$  in $\mathcal{Z}$ :
in $\Coh_*(\Omega\Sigma\mathbb CP^{\infty})$ it arises from the
time-inversion of loops.

\item $\Coh^*(\Omega\Sigma\mathbb CP^{\infty})$ is the graded dual
    Hopf algebra to $\Coh_*(\Omega\Sigma\mathbb CP^{\infty})$.
\end{enumerate}

\begin{thm*}\, [BR]
There is an isomorphism of graded Hopf algebras:
$$
\Coh_*(\Omega\Sigma \mathbb CP^{\infty})\simeq \mathcal{Z}=\mathbb Z \langle
Z_1,Z_2,\dots\rangle,\quad \Coh^*(\Omega\Sigma \mathbb CP^{\infty})\simeq
\mathcal{M}=\Qsym[t_1,t_2,\dots].
$$
\end{thm*}

\item[\textbf{III.}] $\Coh_*(BU)\simeq \Coh^*(BU)\simeq\mathbb
    Z[\sigma_1,\sigma_2,\dots]\simeq\mathcal{C}$. It is a self-dual Hopf
    algebra of symmetric functions. In the cohomology $\sigma_i$ are
    represented by Chern classes.

\item[\textbf{IV.}]
\begin{enumerate}
\item $\Coh_*(\Omega\Sigma S^2)=\Coh_*(\Omega S^3)=\mathbb Z[w]$ -- a
    polynomial ring with $\deg w=2$ and the comultiplication
$$
\Delta w=1\otimes w+w\otimes 1;
$$

\item $\Coh^*(\Omega\Sigma S^2)=\mathbb Z[u_n]/I$ is a divided power
    algebra. Thus $\Coh^*(\Omega\Sigma S^2)\simeq\Coh_*(\mathbb
    CP^{\infty})$.
\end{enumerate}

\item[\textbf{V.}] $\Coh_*(\Omega\Sigma(\Omega\Sigma S^2))\simeq\mathbb
    Z\langle w_1,w_2,\dots\rangle$. It is a free associative Hopf algebra
    with the comultiplication
\begin{equation}\label{F-2}
\Delta w_n=\sum\limits_{k=0}^n{n\choose k}w_k\otimes w_{n-k}
\end{equation}

\item[\textbf{VI.}]
    $\Coh_*\Big(\Omega\Sigma\Big(\bigvee\limits_{n\geqslant1}S^{2n}\Big)\Big)\simeq
    \mathbb{Z}\langle \xi_1,\xi_2,\ldots\rangle$. It is a free
    associative algebra and has the structure of a graded Hopf algebra
    with the comultiplication
\begin{equation}\label{F-3}
\Delta\xi_n = 1\otimes \xi_n + \xi_n\otimes 1.
\end{equation}
Therefore,
$\Coh_*\Big(\Omega\Sigma\Big(\bigvee\limits_{n\geqslant1}S^{2n}\Big)\Big)$
gives a topological realization of the universal Lie-Hopf algebra
$\mathcal{W}$.

The homotopy equivalence
\[ a \colon \Omega\Sigma\Big(\bigvee\limits_{n\geqslant1}S^{2n}\Big)
\longrightarrow \Omega\Sigma(\Omega\Sigma S^2) \] induces
an isomorphism of graded Hopf algebras
 \[ a_* \colon \mathbb{Z}\langle \xi_1,\xi_2,\ldots\rangle \longrightarrow
 \mathbb{Z}\langle w_1,w_2,\ldots\rangle, \]
and its algebraic form is determined by the conditions:
 \[ \Delta a_*\xi_n = (a_* \otimes a_*)(\Delta\xi_n). \]
For example, $a_*\xi_1 = w_1,\; a_*\xi_2 = w_2- w_1|w_1,\; a_*\xi_3 =
w_3- 3w_2|w_1 + 2w_1|w_1|w_1$.
\end{itemize}

Thus using topological results we have obtained that two Hopf algebra
structures on the free associative algebra with the comultiplications
(\ref{F-2}) and (\ref{F-3}) are isomorphic over $\mathbb{Z}$.

This result is interesting from the topological point of view, since the
elements $(w_n-a_*\xi_n)$ for $n\geqslant 2$ are obstructions to the
desuspension of the homotopy equivalence

$$
\Sigma(\Omega\Sigma S^2) \to
\Sigma\Big(\bigvee\limits_{n\geqslant1}S^{2n}\Big).
$$

We have the commutative diagram:
$$
\xymatrix{
\Omega\Sigma(\Omega\Sigma S^2)\ar[r]^{\Omega\Sigma k} &
\Omega\Sigma \mathbb{C}P^{\infty}\ar@{-->}[dr]^j&&\\
\Omega\Sigma S^2\ar[r]^{k} \ar[u] & \mathbb{C}P^{\infty}
\ar[u]\ar[r]^i&BU\ar[r]^{\det}&\mathbb{C}P^{\infty}\\
}
$$
Here
\begin{itemize}
\item $i$ is the inclusion $\mathbb{C}P^{\infty}=BU(1)\subset BU$,

\item $j$ arises from the universal property of $\Omega\Sigma
    \mathbb{C}P^{\infty}$ as a free $H$-space;

\item the mapping $\mathbb{C}P^{\infty}\to\Omega\Sigma
    \mathbb{C}P^{\infty}$ is induced by the identity map $\Sigma
    \mathbb{C}P^{\infty}\to\Sigma \mathbb CP^{\infty}$,

\item the mapping $k\colon \Omega\Sigma S^2 \to \mathbb CP^{\infty}$
    corresponds to the generator of $\Coh^2(\Omega\Sigma
    S^2)=\mathbb{Z}$.

\item $\det$ is the mapping of the classifying spaces $BU\to BU(1)$
    induced by the mappings \linebreak $\det: U(n)\to
    U(1),\;n=1,2,\dots$.
\end{itemize}

As we have mentioned, $\Coh^*(BU)$ as a Hopf algebra is isomorphic to the
algebra of symmetric functions (in the generators -- Chern classes) and is
self-dual.
$$
\Coh^*(BU)\simeq \Coh_*(BU)\simeq\Sym[t_1,t_2,\dots]=\mathbb
Z[\sigma_1,\sigma_2,\dots].
$$
Since $\mathbb CP^{\infty}$ gives rise to the algebra generators in
$\Coh_*(BU)$, $j_*$ is an epimorphism on homology, and $j^*$ is a
monomorphism on cohomology.
\begin{itemize}
\item $j_*$ corresponds to the factorization of the non-commutative
    polynomial algebra $\mathcal{Z}$ by the commutation relations
    $Z_iZ_j-Z_jZ_i$ and sends $Z_i$ to $\sigma_i$.

\item $j^*$ corresponds to the inclusion
    $\Sym[t_1,t_2,\dots]\subset\Qsym[t_1,t_2,\dots]$ and sends $\sigma_i$
    to $M_{\omega_i}$, where $\omega_i=\underbrace{(1,1,\dots,1)}_i$.

\item $k^*$ sends $u$ to $u_1$ and defines an embedding of the polynomial
    ring $\mathbb Z[u]$ into the divided power algebra $\mathbb
    Z[u_n]/I$.

\item the composition $\det_*\circ j_*$ maps the algebra $\mathcal{Z}$ to
    the divided power algebra $H_*(\mathbb CP^{\infty})$. This
    corresponds to the mapping $Z_i\to \left.d_k\right|_{\mathcal{P}_s}$
    of $\mathcal{Z}$ to $\mathcal{D}(\mathcal{P}_s)$.
\end{itemize}

\section{Structure of $\mathcal{D}$}
\begin{thm}
The homomorphism $\EuScript{L}_{\R}:\mathcal{Z}\to\mathcal{D}(\R)\colon
Z_k\to d_k$ induces a ring isomorphism:
$$
\mathcal{D}(\R)\simeq \mathcal{U}=\mathcal{Z}/J_{\mathcal{U}},
$$
where $\mathcal{U}$ is a universal Leibnitz-Hopf algebra with the antipode
$\chi(U_i)=(-1)^iU_i$.
 \label{D}
\end{thm}

\begin{proof}
Let us prove the following lemma:
\begin{lemma}
Let $D\in\mathcal{D}(\R)$ be an operator of graduation $2k$. Then for each
space $\R^{2n},\;n\geqslant k$ there is a unique representation
\begin{equation}\label{X5}
D=u(d_2,d_3,\dots,d_k)+dw(d_2,d_3,\dots,d_{k-1}),
\end{equation}
and there is a unique representation
\begin{equation}\label{X6}
D=u'(d_2,d_3,\dots,d_k)+w'(d_2,\dots,d_{k-1})d,
\end{equation}
where $u,u',w,w'$ are polynomials.\label{basis}
\end{lemma}
\begin{proof}
For $k=0$ and $1$ it is evident.

Since $2d_2=d^2$, this is true for $k=2$. Let $k\geqslant 3$.

Using the relations (\ref{PM}) we obtain
$$
dd_i=(-1)^id_id+\left(d_2d_{i-1}-d_3d_{i-2}+\dots+(-1)^{i-1}d_{i-1}d_2\right)+(1+(-1)^{i+1})d_{i+1}
$$
So the expressions (\ref{X5}) and (\ref{X6}) exist.

Let $u(d_2,d_3,\dots,d_k)+dw(d_2,\dots,d_{k-1})=0$, where
$u=\sum\limits_{|\omega|=k}a_{\omega}D_{\omega}$, and
$w=\sum\limits_{|\omega|=k-1}b_{\omega}D_{\omega}$.

Then for any $n$-dimensional polytope $P^n,\;n\geqslant k$ we have
$$
\xi_1(u+dw)P^n=0, \quad \varepsilon_1(u+dw)P^n=0
$$
Since $u+dw$ has degree $2k$, this equality can be written as
\begin{equation}
\sum\limits_{|\omega|=k}a_{\omega}f_{n-k,\,n-k+j_1,\,n-k+j_1+j_2,\,\dots,\,n-j_l}(P^n)+\sum\limits_{|\omega|=k-1}b_{\omega}f_{n-k,\,n-k+1,\,n-k+1+j_1,\,\dots,\,n-j_l}(P^n).\label{FVECT}
\end{equation}
Using the generalized Dehn-Sommerville equations we obtain
\begin{multline*}
f_{n-k,\,n-k+1,\,n-k+1+j_1,\,\dots,\,n-j_l}=(-1)^{n-k-1}\left(\sum\limits_{j=0}^{n-k-1}(-1)^jf_{j,\,n-k+1,\,n-k+1+j_1,\,\dots,\,n-j_l}\right)+\\
+(1+(-1)^{n-k})f_{n-k+1,\,n-k+1+j_1,\,\dots,\,n-j_l}
\end{multline*}
All the sets
$\{n-k,n-k+j_1,n-k+j_1+j_2,\dots,n-j_l\},\{j,n-k+1,n-k+j_1,\dots,n-j_l\}$ and
$\{n-k+1,n-k+1+j_1,\dots,n-j_l\}$ for all $\omega$ are different and belong
to $\Psi^n$.

In the case of $\mathcal{RP}$ and $P^{k-1}$ we obtain that $n=k-1$, therefore
$n-k=-1$ in the expression (\ref{FVECT}) vanishes and we obtain
$$
\sum\limits_{|\omega|=k}a_{\omega}f_{j_1-1,\,j_1+j_2-1,\,\dots,\,n-j_l}(P^n)+\sum\limits_{|\omega|=k-1}b_{\omega}f_{0,\,j_1,\,\dots,\,n-j_l}(P^n).\label{fv}
$$
Again all he sets $\{j_1-1,j_1+j_2-1,\dots,n-j_l\}, \{0,j_1,\dots,n-j_l\}$
for all $\omega$ are different and belong to $\Psi^n$.

Since the vectors $\{f_{S}(Q),\,S\in\Psi^n\}, Q\in\Omega^n$ are linearly
independent, we obtain that all $a_{\omega}$ and $b_{\omega}$ are equal to
$0$, so the representation (\ref{X5}) is unique.

We obtain that the monomials $D_{\omega}=d_{j_1}\dots
d_{j_l},\;|\omega|=k,\;j_i\geqslant 2$ and $dD_{\omega}=dd_{j_1}\dots
d_{j_l}\;|\omega|=k-1,\;j_i\geqslant2$ form a basis of the abelian group
$\mathcal{D}(\R^{2n},\R^{2(n-k)})$.

But each monomial $dD_{\omega},\;|\omega|=k,\;j_i\geqslant 2$ can be
expressed as an integer combination of the monomials $D_{\omega'}$ and
$D_{\omega'}d$. So the monomials $D_{\omega}$ and $D_{\omega}d$ also form a
basis. This proves the second part of the lemma.
\end{proof}
Now let us prove the theorem. The mapping
$\EuScript{L}_{\R}:\mathcal{Z}\to\mathcal{D}(\R):\; Z_i\to d_i$ is an
epimorphism. Let $z\in\mathcal{Z}$ such that $\deg z=2k$ and
$\EuScript{L}_{\R}z=0$.
$$
z=\sum\limits_{|\omega|=k}a_{\omega}Z_{\omega}.
$$
We know that
$$
Z_1Z_{k-1}=(-1)^{k-1}Z_{k-1}Z_1+\sum\limits_{i=2}^{k-2}(-1)^iZ_iZ_{k-i}+(1+(-1)^k)Z_k-\sum\limits_{i=0}^{k}(-1)^iZ_iZ_{k-i}.
$$
So
$$
z=\sum\limits_{|\omega|=k,\;j_i\geqslant 2}a'_{\omega}Z_{\omega}+\sum\limits_{|\omega|=k-1,\;j_i\geqslant 2}b_{\omega}Z_{\omega}Z_1+z',
$$
where $z'\in J_{\mathcal{U}}$. Since
$\EuScript{L}_{\R}z=\EuScript{L}_{\R}z'=0$, we obtain
$$
\sum\limits_{|\omega|=k,\;j_i\geqslant 2}a'_{\omega}D_{\omega}+\sum\limits_{|\omega|=k-1,\;j_i\geqslant 2}b_{\omega}D_{\omega}d=0,
$$
Lemma \ref{basis} implies that all the coefficients $a_{\omega}'$ and
$b_{\omega}$ are equal to $0$, so $z=z'\in J_{\mathcal{U}}$. This proves that
$\mathcal{D}(\R)\simeq \mathcal{Z}/J_{\mathcal{U}}$.
\end{proof}

\begin{rem}
Now we see that the correspondence $d_k\to d_k$ defines a natural isomorphism
between $\mathcal{D}(\mathcal{RP})$ and $\mathcal{D}(\mathcal{P})$. Therefore
from this moment we will often use the notation $\mathcal{D}$ for both these
algebras, and the notations $\EuScript{L}$, $\EuScript{R}$ for the
homomorphisms $\EuScript{L}_{\R}$ and $\EuScript{R}_{\R}$.
\end{rem}

\begin{cor}
The comultiplication $d_k\to \sum\limits_{i+j=k}d_i\otimes d_j$ defines on
$\mathcal{D}$ the structure of a Leibnitz-Hopf algebra over the ring $\mathbb
Z$. This Hopf algebra $\mathcal{D}$ is isomorphic to the Hopf algebra
$\mathcal{U}$, universal in the category of Leibnitz-Hopf algebras  with the
antipode $\chi(U_i)=(-1)^iU_i$.

Both $\mathcal{P}$ and $\mathcal{RP}$ have natural left Milnor module
structures over $\mathcal{D}$
\end{cor}

It is evident that $\left.D_{\omega}\right|_{\R^{[2n]}}=0$ for $|\omega|>n$.
The following corollary says that the other relations between operators in
$\mathcal{D}$ on the abelian group $\R^{[2n]}$ are the same as on the whole
ring $\R$.
\begin{cor}
We have
$$
\mathcal{D}(\R^{[2n]})=\mathcal{D}(\R)/J_n
$$
where the ideal $J_n$ is generated by all the operators
$D_{\omega},\;|\omega|>n$.
\end{cor}
\begin{proof}
As in the proof of Theorem \ref{D} using Lemma \ref{basis} it is easy to show
that $\mathcal{D}(\R^{[2n]})\simeq\mathcal{Z}/J_{\mathcal{D}(\R^{[2n]})}$,
where the ideal $J_{\mathcal{D}(\R^{[2n]})}$ is generated by the the
relations $\Phi(-t)\Phi(t)=1$ and $Z_{\omega}=0$ for $|\omega|>n$.
\end{proof}

\begin{cor}
The operators $d_2,d_3,d_4,\dots$ are algebraically independent.
\end{cor}

\begin{cor}
Rank of the $(2n)$-th graded component of the ring $\mathcal{D}$ is equal to
the $(n-1)$-th Fibonacci number $c_{n-1}$.\label{Fib}
\end{cor}
\begin{proof}
At first, let us calculate rank $\rk_n$ of the $(2n)$-th graded component of
the ring $\mathbb Z\langle d_2,d_3,d_4,\dots\rangle$.
$$
\rk_0=1,\;\rk_1=0,\;\rk_2=1,\;\rk_3=1,\;\rk_4=2,\;\rk_5=3,\dots
$$
It is easy to see that there is a recursive relation
$$
\rk_{n+1}=\rk_{n-1}+\rk_{n-2}+\dots+\rk_2+1,\;n\geqslant 3
$$
Then $\rk_{n+1}=\rk_{n-1}+\rk_n$. Since $\rk_2=\rk_3=1$, we obtain
$\rk_n=c_{n-2},\; n\geqslant 2$. At last, rank of the $(2n)$-th  graded
component of the ring $\mathcal{D}$ is equal to
$\rk_{n-1}+\rk_n=c_{n-3}+c_{n-2}=c_{n-1}$ for $n\geqslant 3$. It is easy to
see that for $n=1$ and $2$ it is also true.
\end{proof}

Let us remind that on the ring of simple polytopes
$\left.d_k\right|_{\mathcal{P}_s}=\left.\frac{d^k}{k!}\right|_{\mathcal{P}_s}$,
so the ring $\mathcal{D}(\mathcal{P}_s)$ is isomorphic to the divided power
ring $d_kd_l={k+l\choose k}d_{k+l}$.
\begin{prop}
$$
\mathcal{D}\otimes\mathbb Q=\mathbb Q\langle
d_1,d_3,d_5,\dots\rangle,
$$
The inclusion $\mathcal{D}\subset\mathcal{D}\otimes\mathbb Q$ is an
embedding, and the operators $d_{2k}$ are expressed in terms of the operators
$d_1,d_3,\dots,d_{2k-1}$ by the formulas
\begin{equation}\label{X7}
d_{2k}=\sum\limits_{i=1}^{k}(-1)^{i-1}\frac{{2i-2\choose i-1}}{i2^{2i-1}}\left(\sum\limits_{j_1+j_2+\dots+j_{2i}=i+k,\,j_l\geqslant 1}d_{2j_1-1}\dots
d_{2j_{2i}-1}\right).
\end{equation}
\end{prop}
\begin{proof}
Let us denote
$$
a(t)=\sum\limits_{k\geqslant 0}d_{2k}t^{2k},\quad b(t)=\sum\limits_{k\geqslant 0}d_{2k+1}t^{2k+1}.
$$
Then
\begin{gather*}
\Phi(t)=a(t)+b(t)\;\quad\Phi(-t)=a(t)-b(t);\\
a(t)b(t)=\frac{\Phi(t)+\Phi(-t)}{2}\cdot\frac{\Phi(t)-\Phi(-t)}{2}=\frac{\Phi(t)^2-\Phi(-t)^2}{4}=b(t)a(t).
\end{gather*}
Then the relation $\Phi(-t)\Phi(t)=1$ is equivalent to the relations
\begin{gather*}
a(t)^2-b(t)^2=1;\\
a(t)b(t)=b(t)a(t).
\end{gather*}
Therefore $a(t)=\sqrt{1+b(t)^2}$ and the formula (\ref{X7}) is valid.
Consequently all the operators $d_{2k}$ are expressed as polynomials in
$d_1,d_3,\dots$ with rational coefficients. For example,
$$
d_2=\frac{d^2}{2},\quad d_4=\frac{dd_3+d_3d}{2}-\frac{d^4}{8}.
$$
This means that the algebra $\mathcal{D}\otimes\mathbb Q$ is generated by
$d_1,d_3,d_5,\dots$. On the other hand, let us calculate the number of the
monomials $d_{2j_1-1}\dots d_{2j_k-1}$ with $(2j_1-1)+\dots+(2j_k-1)=n$.
Denote this number by $l_n$. Then
$$
l_0=1,\;l_1=1,\;l_2=1,\;l_3=2,\;l_4=3,\dots,
$$
and there is a recursive formula $l_{n+1}=l_{n}+l_{n-2}+l_{n-4}+\dots$ for
$n\geqslant 2$. Hence $l_{n+1}=l_n+l_{n-1}$ for $n\geqslant 2$, and
$l_n=c_{n-1}$ for $n\geqslant 1$.

We see that the number of monomials of degree $2n$ is equal to dimension of
the $(2n)$-th graded component according to Corollary \ref{Fib}. This implies
that they are linearly independent over the rationals. Therefore
$d_1,d_3,d_5,\dots$ are algebraically independent.

At last, the inclusion $\mathcal{D}\subset \mathcal{D}\otimes\mathbb Q$ is an
embedding, since $\mathcal{D}$ is torsion-free.
\end{proof}
\begin{defin}
Let us define operators $s_k$ by the formula
$$
s(t)=s_1t+s_2t^2+s_3t^3+\dots=\log\Phi(t).
$$
\end{defin}
Then $s_1=d_1,\;s_2=0,\;s_3=d_3-\frac{d^3}{6}$, and so on.

The relation $\Phi(-t)\Phi(t)=1$ turns into the relation $s(-t)+s(t)=0$. Thus
$s_{2k}=0$ for all $k$.

Also we have $\Delta s(t)=1\otimes s(t)+s(t)\otimes 1$, so each operator
$s_{2k-1}$ is a derivation.
\begin{prop}
There is an isomorphism of Hopf algebras
$$
\mathcal{D}\otimes\mathbb Q=\mathbb Q\langle
s_1,s_3,s_5,\dots\rangle,
$$
where $\mathbb Q\langle s_1,s_3,s_5,\dots\rangle$ is a free Lie-Hopf algebra
in the generators of degree $\deg s_{2k-1}=2(2k-1)$, $k\geqslant 1$, the
comultiplication $\Delta s_{2k-1}=1\otimes s_{2k-1}+s_{2k-1}\otimes 1$, and
the antipode $\chi(s_{2k-1})=-s_{2k-1}$\label{S}
\end{prop}

\section{Generalized $f$-polynomial}

\subsection{Definition}

The mapping $\Phi:\mathcal{P}\to\mathcal{P}[t]$ is a ring homomorphism. If we
set $\Phi(t_2)t_1=t_1$, then the operator $\Phi(t_2)$ defines a ring
homomorphism $\mathcal{P}\to\mathcal{P}[t_1,t_2]$: $P\to\Phi(t_2)\Phi(t_1)P$,

Thus for each $n\geqslant 0$ we can define the ring homomorphism
$\Phi_n:\mathcal{P}\to\mathcal{P}[[t_1,\dots,t_n]]$ as a composition:
$$
\Phi_n(t_1,\dots,t_n)P=\Phi(t_n)\dots\Phi(t_1)P.
$$

\begin{defin}
Let $P^n$ be an $n$-dimensional polytope. Let us define a quasi-symmetric
function
$$
f_r(\alpha,t_1,\dots,t_r)(P^n)=\xi_{\alpha}\Phi(t_r)\dots\Phi(t_1)P^n=\alpha^n+\sum\limits_{k=1}^{\min\{r,n\}}\sum\limits_{0\leqslant
a_1<\dots<a_k\leqslant
n-1}f_{a_1,\,\dots,\,a_k}\alpha^{a_1}M_{(n-a_k,\,\dots,\,a_2-a_1)}.
$$
\end{defin}
For $n=1$ we obtain
$f_1(\alpha,t_1)(P^n)=\alpha^n+f_{n-1}\alpha^{n-1}t_1+\dots+f_0t_1^n$ is a
homogeneous $f$-polynomial in two variables (\cite{Buch}). So the polynomial
$f_n$ is a generalization of the $f$-polynomial.

Consider the increasing sequence of rings
$$
\mathbb Z[\alpha]\subset \mathbb Z[\alpha,t_1]\subset\mathbb
Z[\alpha,t_1,t_2]\subset\dots
$$
with the restriction maps $E_r:\mathbb Z[\alpha,t_1,\dots,t_{r+1}]\to\mathbb
Z[\alpha,t_1,\dots,t_r]$
$$
(E_rg)(\alpha,t_1,\dots,t_r)=g(\alpha,t_1,\dots,t_r,0).
$$
Since $f_{r+1}(\alpha,t_1,\dots,t_r,0)(P^n)=f_r(\alpha,t_1,\dots,t_r)(P^n)$,
we obtain the ring homomorphism
\begin{gather*}
f:\mathcal{P}\to\Qsym[t_1,t_2,\dots][\alpha]\subset\varprojlim\limits_r\mathbb
Z[\alpha,t_1,\dots,t_r]:\\
f(\alpha,t_1,t_2,\dots)(P^n)=\alpha^n+\sum\limits_{k=1}^{n}\sum\limits_{0\leqslant
a_1<\dots<a_k\leqslant
n-1}f_{a_1,\,\dots,\,a_k}\alpha^{a_1}M_{(n-a_k,\,\dots,\,a_2-a_1)}.
\end{gather*}

\begin{defin}
For a set $S=\{a_1,\dots,a_k\}\subset \{0,1,\dots,n-1\}$ let us denote by $\omega(S)$
the composition $(n-a_k,a_k-a_{k-1},\dots,a_2-a_1)$ of the number $n-a_1$.
\end{defin}
Then we can write
$$
f(P^n)=\sum\limits_{k\geqslant 0}\sum\limits_{S:\;l(S)=k}f_{S}\;\alpha^{a_1}M_{\omega(S)}
$$

\subsection{Image}

It follows from the formula, that the restriction
$$
f(\alpha,t_1,\dots,t_r,t_{r+1},\dots)\to
f(\alpha,t_1,\dots,t_r,0,0,\dots)=f_r(\alpha,t_1,\dots,t_r)
$$
is injective on the space of all $n$-dimensional polytopes, $n\leqslant r$.
\begin{thm} The image of the space $\mathcal{P}^{2n}$ generated by all
$n$-dimensional polytopes in the ring\linebreak
$\Qsym[t_1,\dots,t_r][\alpha],\;r\geqslant n$ under the mapping $f_r$
consists of all the homogeneous polynomials $g$ of degree $2n$  satisfying
the equations
\begin{enumerate}
\item
\begin{align*}
g(\alpha,t_1,-t_1,t_3,\dots,t_r)&=g(\alpha,0,0,t_3,\dots,t_r);\\
g(\alpha,t_1,t_2,-t_2,t_4,\dots,t_r)&=g(\alpha,t_1,0,0,t_4\dots,t_r);\\
\dots\\
g(\alpha,t_1,\dots,t_{r-2},t_{r-1},-t_{r-1})&=g(\alpha,t_1,\dots,t_{r-2},0,0);
\end{align*}
\item
$$
g(-\alpha,t_1,\dots,t_{r-1},\alpha)=g(\alpha,t_1,\dots,t_{r-1},0);
$$
\end{enumerate}
These equations are equivalent to the Bayer-Billera (generalized
Dehn-Sommerville) relations.\label{f}
\end{thm}
\begin{proof}
$\Phi(-t)\Phi(t)=1=\Phi(0)\Phi(0)$, therefore
\begin{align*}
\xi_{\alpha}\Phi(t_r)\dots\Phi(t_3)\Phi(-t_1)\Phi(t_1)P^n&=\xi_{\alpha}\Phi(t_r)\dots\Phi(t_3)\Phi(0)\Phi(0)P^n;\\
\xi_{\alpha}\Phi(t_r)\dots\Phi(t_4)\Phi(-t_2)\Phi(t_2)\Phi(t_1)P^n&=\xi_{\alpha}\Phi(t_r)\dots\Phi(t_4)\Phi(0)\Phi(0)\Phi(t_1)P^n;\\
\dots\\
\xi_{\alpha}\Phi(-t_{r-1})\Phi(t_{r-1})\Phi(t_{n-2})\dots\Phi(t_1)P^n&=\xi_{\alpha}\Phi(0)\Phi(0)\Phi(t_{r-2})\dots\Phi(t_1)P^n.
\end{align*}
On the other hand, Proposition \ref{Euler} gives the last relation
$$
\xi_{-\alpha}\Phi(\alpha)\Phi(t_{r-1})\dots\Phi(t_1)P^n=\xi_{\alpha}\Phi(0)\Phi(t_{r-1})\dots\Phi(t_1)P^n.
$$

Now let us proof the opposite inclusion, that is if the homogeneous
polynomial $g$ of degree $2n$ satisfies the conditions of the theorem, then
$g=f_r(p^n)$ for some $p^n\in\mathcal{P}^{2n}$.

\begin{lemma}\label{equations}
Let $P^n$ be an $n$-dimensional polytope. Then
\begin{enumerate}
\item The equation
$$
f_r(\alpha,t_1,\dots,t_q,-t_q,\dots,t_r)=f_r(\alpha,t_1,\dots,0,0,\dots,t_r)
$$
is equivalent to the generalized Dehn-Sommerville relations
$$
\sum\limits_{j=a_t+1}^{a_{t+1}-1}(-1)^{j-a_t-1}f_{a_1,\,\dots,\,a_t,\,j,\,a_{t+1},\,\dots,\,a_k}=\left(1+(-1)^{a_{t+1}-a_t}\right)f_{a_1,\,\dots,\,a_t,\,a_{t+1},\,\dots,\,a_k}
$$
for $1\leqslant k\leqslant\min\{r-1,n-1\},\;k+1-q\leqslant t\leqslant
r-q$.

\item The equation
$$
f_r(-\alpha,t_1,\dots,t_{r-1},\alpha)=f_r(\alpha,t_1,\dots,t_{r-1},0)
$$
is equivalent to the generalized Dehn-Sommerville relations
$$
\sum\limits_{j=0}^{a_1-1}(-1)^jf_{j,\,a_1,\,\dots,\,a_k}=(1+(-1)^{a_1-1})f_{a_1,\,\dots,\,a_k}
$$
for $0\leqslant k\leqslant\min\{r-1,n-1\}$.
\end{enumerate}
\end{lemma}
\begin{proof}
\begin{multline*}
f_r(\alpha,t_1,\dots,t_q,-t_q,\dots,t_r)=\\
\alpha^n+\left.\sum\limits_{k=1}^{\min\{r,\,n\}}\sum\limits_{0\leqslant
a_1<\dots<a_k\leqslant
n-1}f_{a_1,\,\dots,\,a_k}\alpha^{a_1}\left(\sum\limits_{1\leqslant
l_1<\dots<l_k\leqslant r}t_{l_1}^{n-a_k}\dots
t_{l_k}^{a_2-a_1}\right)\right|_{t_{q+1}=-t_q}=\\
=\alpha^n+\sum\limits_{k=1}^{\min\{r-2,\,n\}}\sum\limits_{0\leqslant
a_1<\dots<a_k\leqslant
n-1}f_{a_1,\,\dots,\,a_k}\alpha^{a_1}\left(\sum\limits_{1\leqslant
l_1<\dots<l_k\leqslant r,\,l_j\ne q,\,q+1}t_{l_1}^{n-a_k}\dots
t_{l_k}^{a_2-a_1}\right)+\\
+\sum\limits_{k=1}^{\min\{r-1,\,n\}}\sum\limits_{0\leqslant
a_1<\dots<a_k\leqslant
n-1}f_{a_1,\,\dots,\,a_k}\alpha^{a_1}\left(\sum\limits_{1\leqslant
l_1<\dots<l_j=q<\dots<l_k\leqslant r,\,l_{j+1}\ne q+1}t_{l_1}^{n-a_k}\dots
t_q^{a_{k+2-j}-a_{k+1-j}}\dots
t_{l_k}^{a_2-a_1}\right)+\\
+\sum\limits_{k=1}^{\min\{r-1,\,n\}}\sum\limits_{0\leqslant
a_1<\dots<a_k\leqslant
n-1}f_{a_1,\,\dots,\,a_k}\alpha^{a_1}\left(\sum\limits_{1\leqslant
l_1<\dots<l_{j+1}=q+1<\dots<l_k\leqslant r,\,l_j\ne q}t_{l_1}^{n-a_k}\dots
(-t_q)^{a_{k+1-j}-a_{k-j}}\dots
t_{l_k}^{a_2-a_1}\right)+\\
\sum\limits_{k=1}^{\min\{r,\,n\}}\!\!\!\!\!\sum\limits_{0\leqslant
a_1<\dots<a_k\leqslant
n-1}f_{a_1,\,\dots,\,a_k}\alpha^{a_1}\left(\sum\limits_{1\leqslant
l_1<\dots<l_j=q<l_{j+1}=q+1<\dots< l_k\leqslant r}t_{l_1}^{n-a_k}\dots
(-1)^{a_{k+1-j}-a_{k-j}}t_q^{a_{k+2-j}-a_{k-j}}\dots t_{l_k}^{a_2-a_1}\right)
\end{multline*}
The first and the second summands form exactly
$f(\alpha,t_1,\dots,0,0,\dots,t_r)$. Therefore all the coefficients of the
polynomial consisting of the last three summands should be equal to $0$.

Consider the monomial $\alpha^{a_1}t_{l_1}^{n-a_k}\dots
t_q^{a_{t+1}-a_t}\dots t_{l_k}^{a_2-a_1}$. Here $q=l_{k+1-t}$. The existence
of a monomial of this form in the sum is equivalent to the conditions
$$
k\leqslant\min\{r-1,n-1\},\;k+1-t\leqslant q,\;t-1\leqslant r-q-1.
$$
The coefficient of the monomial should be equal to $0$. This is equivalent to
the relation:
$$
\left(1+(-1)^{a_{t+1}-a_t}\right)f_{a_1,\,\dots,\,a_t,\,a_{t+1},\,\dots,\,a_k}+\sum\limits_{j=a_t+1}^{a_{t+1}-1}(-1)^{j-a_t}f_{a_1,\,\dots,\,a_t,\,j,\,a_{t+1},\,a_k}=0
$$
Now let us consider the remaining relation
    $f_r(-\alpha,t_1,\dots,t_{r-1},\alpha)=f_r(\alpha,t_1,\dots,t_{r-1},0)$:
\begin{multline*}
=(-\alpha)^n+\sum\limits_{k=1}^{\min\{r-1,\,n\}}\sum\limits_{0\leqslant
a_1<\dots<a_k\leqslant
n-1}f_{a_1,\,\dots,\,a_k}(-\alpha)^{a_1}\left(\sum\limits_{1\leqslant
l_1<\dots<l_k\leqslant r-1}t_{l_1}^{n-a_k}\dots
t_{l_k}^{a_2-a_1}\right)+\\
+\sum\limits_{k=1}^{\min\{r,\,n\}}\sum\limits_{0\leqslant
a_1<\dots<a_k\leqslant
n-1}f_{a_1,\,\dots,\,a_k}(-\alpha)^{a_1}\left(\sum\limits_{1\leqslant
l_1<\dots<l_k=r}t_{l_1}^{n-a_k}\dots
t_{l_{k-1}}^{a_3-a_2}\alpha^{a_2-a_1}\right)=\\
\alpha^n+\sum\limits_{k=1}^{\min\{r-1,\,n\}}\sum\limits_{0\leqslant
a_1<\dots<a_k\leqslant
n-1}f_{a_1,\,\dots,\,a_k}\alpha^{a_1}\left(\sum\limits_{1\leqslant
l_1<\dots<l_k\leqslant r-1}t_{l_1}^{n-a_k}\dots t_{l_k}^{a_2-a_1}\right)
\end{multline*}
This is equivalent to the relations
$$
(-1)^{a_1}f_{a_1,\,\dots,\,a_k}+\sum\limits_{j=0}^{a_1-1}(-1)^jf_{j,\,a_1,\,\dots,\,a_k}=f_{a_1,\,\dots,\,a_k}
$$
for $0\leqslant k\leqslant\min\{r-1,n-1\}$. If $k=0$, then the corresponding
relation is exactly the Euler formula
$$
(-1)^n+(-1)^{n-1}f_{n-1}+\dots+f_2-f_1+f_0=1.
$$

\end{proof}
\begin{cor}
For $r\geqslant n$ relations 1. and 2. of the theorem are equivalent to the
generalized Dehn-Sommerville relations for the polytope $P^n$: For
$S\subset\{0,\dots,n-1\}$, and $\{i,k\}\subseteq S\cup\{-1,n\}$ such that
$i<k-1$ and $S\cap\{i+1,\dots,k-1\}=\varnothing$:
$$
\sum\limits_{j=i+1}^{k-1}(-1)^{j-i-1}f_{S\cup\{j\}}=(1-(-1)^{k-i-1})f_S.
$$
\end{cor}
\begin{proof}
We see that relations 1. and 2. follow from the generalized Dehn-Sommerville
relations.

On the other hand, if $i=-1$ then the corresponding relation follows from
equation 2.

If $S=\{a_1,\dots,a_t,a_{t+1},\dots,a_s\}$, $i=a_t\geqslant 0$, and
$k=a_{t+1}$; or $S=\{a_1,\dots,a_t\}$, $i=a_t$, and $k=n$, then we can take
$q$ such that $s+1-t\leqslant q\leqslant r-t$.
\end{proof}
\begin{rem}
For  $n$-dimensional polytopes for different $r\geqslant n$ not all the
equations are independent.

For $i=-1$ relation (\ref{BBR}) follows from equation 2. and all the
equations of type 1. do not contain the case $i=-1$.

On the other hand, let $S=\{a_1,\dots,a_t,a_{t+1},\dots,a_s\}$,
$i=a_t\geqslant 0$, and $k=a_{t+1}$; or $S=\{a_1,\dots,a_t\}$, $i=a_t$, and
$k=n$. Lemma \ref{equations} implies that relation (\ref{BBR}) follows from
the equation with $s+1-t\leqslant q\leqslant r-t$. Let us denote $a=s+1-t$,
$b=r-t$. There are two conditions for $s,t$, namely
$$
1\leqslant s\leqslant n-1 \mbox{ and }1\leqslant t\leqslant s.
$$
Let us rewrite these conditions in terms of $a$ and $b$:
$t=r-b;\;s=a+t-1=a-b+r-1$, so
$$
r-n\leqslant b-a\leqslant r-2;\;b\leqslant r-1;\;1\leqslant a
$$
This gives us a triangle on the plane $(a,b)$. Each point $(a,b)$ of this
triangle corresponds to the conditions $a\leqslant q\leqslant b$, that is
there should be the equation for $q$ in the segment $[a,b]$. We can imagine
that this segment is the segment $[(a,a),(a,b)]$ on the plane.

Thus for $r=n$ all the equations for $t_1,\dots,t_{n-1}$ are necessary, for
$r=n+1$ it is enough to take the equations for
$q=2,4,6,\dots,2[\frac{n}{2}]$. If $r\geqslant 2n-2$ one equation
$$
f(\alpha,t_1,\dots,t_{n-1},-t_{n-1},t_{n+1},\dots,t_r)=f(\alpha,t_1,\dots,0,0,t_{n+1},\dots,t_r)
$$
gives all the relations of type 1.
\end{rem}
Now let us finish the proof of the theorem. If the homogeneous polynomial
$g\in\Qsym[t_1,\dots,t_r][\alpha]$ of degree $2n$ satisfies all the relations
of the theorem, then its coefficients satisfy the generalized
Dehn-Sommerville relations. Therefore all the coefficients are linear
combinations of the coefficients $g_{a_1,\,\dots,\,a_k}$, where
$S=\{a_1,\dots,a_k\}\in \Psi^n$. As we know, the vectors
$\{f_{S}(Q),\;S\in\Psi^n\},\;Q\in\Omega^n$ form a basis of the abelian group
of all the vectors $\{f_{S},\;S\in\Psi^n,\; f_{S}\in\mathbb Z\}=\mathbb
Z^{c_n}$. So the vector $\{g_{S},\;S\in\Psi^n\}$ is an integer combination of
the vectors $\{f_{S}(Q),\;S\in\Psi^n\},\; Q\in\Omega^n$. This implies that
the polynomial $g$ is an integer combination of the polynomials
$f_r(Q),\;Q\in\Omega^n$ with the same coefficients.
\end{proof}

Let us remind that rank of the space $f_r(\mathcal{P}^{2n}),\;r\geqslant n$
is equal to $c_n$.

\begin{prop}
Let $r\geqslant 2$, and let $P^n$ be an $n$-dimensional polytope. Then
$$
f_r(\alpha,t_1,\dots,t_r)(P^n)=f_1(\alpha,t_1+\dots+t_r)(P^n)
$$
if and only if $P^n$ is simple. Here $f_1(\alpha,t)$ is a usual homogeneous
$f$-polynomial in two variables.
\end{prop}
\begin{proof}
On the ring of simple polytopes $
\left.d_k\right|_{\mathcal{P}_s}=\left.\frac{d^k}{k!}\right|_{\mathcal{P}_s},
$ so $\Phi(t)=e^{td}$.

Then $\Phi(t_r)\Phi(t_{r-1})\dots\Phi(t_1)=\Phi(t_1+\dots+t_r)$, therefore
$f_r(\alpha,t_1,\dots,t_r)(P^n)=f_1(\alpha,t_1+\dots+t_r)(P^n)$.

On the other hand, let
$f_r(\alpha,t_1,\dots,t_r)(P^n)=f_1(\alpha,t_1+\dots+t_r)(P^n)$.

Then $f_2(\alpha,t_1,t_2)(P^n)=f_1(\alpha,t_1+t_2)(P^n)$. So
$$
\alpha^n+\sum\limits_{i=0}^{n-1}f_i\alpha^i(t_1^{n-i}+t_2^{n-i})+\sum\limits_{0\leqslant i<j\leqslant n-1}f_{ij}\,\alpha^it_1^{n-j}t_2^{j-i}=\alpha^n+\sum\limits_{i=0}^{n-1}f_i\alpha^i(t_1+t_2)^{n-i}
$$
In particular, $f_{01}=nf_0$, since the coefficients of the monomial
$t_1^{n-1}t_2$ on the left and on the right are equal. Hence
$2f_1=f_{01}=nf_0$. This implies that the polytope $P^n$ is simple.
\end{proof}
\begin{rem}
Letting $r$ tend to infinity we obtain that
$f(\alpha,t_1,t_2,\dots)(P^n)=f_1(\alpha,t_1+t_2+\dots)(P^n)$ if and only if
$P^n$ is simple.
\end{rem}
In the case of simple polytopes the equations of the first type are trivial,
but the equation of the second type has the form:
$$
f_1(-\alpha,t_1+\dots+t_{r-1}+\alpha)=f_1(\alpha,t_1+\dots+t_{r-1})
$$
If we denote $t=t_1+\dots+t_{r-1}$, then
$f_1(-\alpha,\alpha+t)=f_1(\alpha,t)$. This equation is equivalent to the
Dehn-Sommerville relations (after the change of variables
$h(\alpha,t)=f_1(\alpha-t,t)$ it is equivalent to the fact that
$h(\alpha,t)=h(t,\alpha)$)

\subsection{Characterization of the generalized $f$-polynomial.}
In this part we find the condition that uniquely determines the generalized
$f$-polynomial.

At first let us find the relation between $f(P^n)$ and $f(d_kP^n)$.
\begin{prop}
For any polytope $P^n\in\mathcal{P}$ we have
$$
f(\alpha,t_1,t_2,\dots)(d_kP^n)=\frac{1}{k!}\left.\frac{\partial^k}{\partial t^k}\right|_{t=0}f(\alpha,t,t_1,t_2,\dots)(P^n)
$$\label{dk}
\end{prop}
\begin{proof}
Indeed, both expressions belong to $\Qs[\alpha]$, and they are uniquely
defined by their restrictions to $n$ variables. Then we have
\begin{multline*}
\left.\left(\left.\frac{1}{k!}\frac{\partial^k}{\partial t^k}\right|_{t=0}f_{n+1}(\alpha,t,t_1,t_2,\dots)(P^n)\right)\right|_{t_{n+1}=t_{n+2}=\dots=0}=\\
\frac{1}{k!}\left.\frac{\partial^k}{\partial t^k}\right|_{t=0}f_{n+1}(\alpha,t,t_1,t_2,\dots,t_n)(P^n)=\frac{1}{k!}\left.\frac{\partial^k}{\partial t^k}\right|_{t=0}\xi_{\alpha}\Phi(t_n)\dots\Phi(t_1)\Phi(t)P^n=\xi_{\alpha}\Phi(t_n)\dots\Phi(t_1)d_kP^n=\\
f_n(\alpha,t_1,t_2,\dots,t_n)(d_kP^n)=\left.f(\alpha,t_1,t_2,\dots)(d_kP^n)\right|_{t_{n+1}=t_{n+2}=\dots=0}
\end{multline*}
\end{proof}
\begin{cor}
\begin{gather*}
f(\alpha,t,t_1,t_2,\dots)(P^n)=f(\alpha,t_1,t_2,\dots)(P^n)+f(\alpha,t_1,t_2,\dots)(dP^n)t+\\
+f(\alpha,t_1,t_2,\dots)(d_2P^n)t^2+\dots+f(\alpha,t_1,t_2,\dots)(d_nP^n)t^n\\
f_{r+1}(\alpha,t,t_1,\dots,t_r)(P^n)=f_r(\alpha,t_1,\dots,t_r)(P^n)+f_r(\alpha,t_1,\dots,t_r)(dP^n)t+\\
+f_r(\alpha,t_1,\dots,t_r)(d_2P^n)t^2+\dots+f_r(\alpha,t_1,\dots,t_r)(d_nP^n)t^n,\;r\geqslant
0;
\end{gather*}\label{eqn}
\end{cor}
\begin{proof}
The first equality follows from Proposition \ref{dk}.

In fact, both equalities can be proved directly:
\begin{multline*}
f_{r+1}(\alpha,t,t_1,\dots,t_r)(P^n)=\xi_{\alpha}\Phi(t_r)\dots\Phi(t_1)\Phi(t)P^n=\sum\limits_{k=0}^n\left(\xi_{\alpha}\Phi(t_r)\dots\Phi(t_1)d_kP^n\right)t^k=\\
=\sum\limits_{k=0}^{n}f_r(\alpha,t_1,\dots,t_r)(d_kP^n)t^k.
\end{multline*}

Letting $r$ tend to infinity we obtain the first equality.
\end{proof}

The first equality of Corollary \ref{eqn} is equivalent to the condition
$f(\alpha,t_1,t_2,\dots)(\Phi(t)P^n)=f(\alpha,t,t_1,t_2,\dots)(P^n)$.

Thus we see that the following diagram commutes
$$
\begin{CD}
\mathcal{P} @>f>>\Qsym[t_1,t_2,\dots][\alpha]\\
@V{\Phi}VV @VTVV\\
\mathcal{P}[t]@>f>>\Qsym[t,t_1,t_2,\dots][\alpha]
\end{CD}
$$
where $f(t)=t$, and
$T:\Qsym[t_1,t_2\dots][\alpha]\to\Qsym[t,t_1,t_2,\dots][\alpha]$ is a ring
homomorphism:
$$
Tg(\alpha,t_1,t_2\dots)=g(\alpha,t,t_1,t_2,\dots),\quad g\in\Qsym[t_1,t_2,\dots][\alpha].
$$
\begin{rem}
Note that
$f_r(\alpha,t_1,\dots,t_r)(\Phi(t)P^n)=f_{r+1}(\alpha,t,t_1,\dots,t_r)(P^n)$
for all $r\geqslant 0$.

Consider the ring homomorphism
$T_{r+1}:\Qsym[t_1,\dots,t_r][\alpha]\to\Qsym[t,t_1,\dots,t_r][\alpha]$
defined as
$$
T_{r+1}(\alpha)=\alpha,\quad T_{r+1}M_{\omega}(t_1,\dots,t_r)=M_{\omega}(t,t_1,\dots,t_r)
$$
Then the corresponding diagram
$f_r(\alpha,t_1,\dots,t_r)(\Phi(t)P^n)=(T_{r+1}f_r)(\alpha,t,t_1,\dots,t_r)$
commutes only for $r\geqslant n$.
\end{rem}

\begin{thm}
Let $\psi:\mathcal{P}\to\Qsym[t_1,t_2,\dots][\alpha]$ be a linear map such
that
\begin{enumerate}
\item $\psi(\alpha,0,0,\dots)(P^n)=\alpha^n$;
\item The following diagram commutes:

$$
\begin{CD}
\mathcal{P} @>\psi>>\Qsym[t_1,t_2,\dots][\alpha]\\
@V{\Phi}VV @VTVV\\
\mathcal{P}[t]@>\psi>>\Qsym[t,t_1,t_2,\dots][\alpha]
\end{CD}
$$
\end{enumerate}
Then $\psi=f$.
\end{thm}
\begin{proof}
The first condition implies that $\psi(\alpha,0,0,\dots)(p)=\xi_{\alpha}p$
for any $p\in\mathcal{P}$.

Let $\psi(P^n)=\alpha^n+\sum\limits_{\omega}\psi_{\omega}(\alpha)M_{\omega}$,
and let $\omega=(j_1,\dots,j_k)$.

Since
$\psi(\alpha,t,t_1,t_2,\dots)(P^n)=\psi(\alpha,t_1,t_2,\dots)(\Phi(t)P^n)$,
we have
$$
\psi(\alpha,y_1,y_2,\dots,y_k,t_1,t_2,\dots)(P^n)=\psi(\alpha,t_1,t_2,\dots)(\Phi(y_k)\dots\Phi(y_1)P^n)
$$
Therefore,
$$
\psi(\alpha,y_1,\dots,y_k,0,0,\dots)=\psi(\alpha,0,0,\dots)(\Phi(y_k)\dots\Phi(y_1)P^n)=\xi_{\alpha}\Phi(y_k)\dots\Phi(y_1)P^n=f(\alpha,y_1,\dots,y_k,0,0,\dots)
$$
Hence $\psi_{\omega}(\alpha)=f_{S(\omega)}\alpha^{n-|\omega|}$, where
$S(\omega)=\{n-|\omega|,n-|\omega|+j_k,\dots,n-j_1\}$. This is valid for all
$\omega$, so $\psi=f$.
\end{proof}
\begin{rem}
Let us mention that the mapping $T$ is an isomorphism
$\Qsym[t_1,t_2,\dots][\alpha]\to\Qsym[t,t_1,t_2,\dots][\alpha]$, while
$\Phi$ is an injection and its image is described by the condition:
$p(t)\in\Phi(\mathcal{P})$ is and only if $\Phi(-t)p(t)\in\mathcal{P}$.
\end{rem}
On the ring of simple polytopes we have
$f(\alpha,t_1,t_2,\dots)(p)=f_1(\alpha,t_1+t_2+\dots)(p)=f_1(\alpha,\sigma_1)(p)$,
so the image of $\mathcal P_s$ belongs to $\mathbb Z[\alpha,\sigma_1]$. On
the other hand, $\Phi(t)=e^{td}$, so we have the condition
$f_1(\alpha,t_1)(e^{td}p)=f_1(\alpha,t+t_1)(p)$.
\begin{prop}\label{Simple}
Let $\psi:\mathcal{P}_s\to\mathbb Z[\alpha,t]$ be a linear mapping such that
\begin{enumerate}
\item $\psi(\alpha,0)(P^n)=\alpha^n$;
\item One of the following equivalent conditions holds:
   \begin{enumerate}
     \item $\psi(\alpha,t_1)(e^{td}p)=\psi(\alpha,t+t_1)(p)$;
     \item $\psi(\alpha,t)(dp)=\frac{\partial}{\partial
         t}\psi(\alpha,t)(p)$.
    \end{enumerate}
\end{enumerate}
Then $\psi=f_1$.
\end{prop}
\begin{proof}
Let conditions $1$ and $2a$ hold. Then
$$
\psi(\alpha,t+t_1)(p)=\psi(\alpha,t_1)(e^{td}p)=\psi(\alpha,0)(e^{(t+t_1)d}p)=\xi_{\alpha}\Phi(t+t_1)p=f_1(\alpha,t_1+t)(p).
$$

It remains to prove that conditions $2a$ and $2b$ are equivalent.

Indeed, if $\psi(\alpha,t_1)(e^{td}p)=\psi(\alpha,t+t_1)(p)$, then
$\psi(\alpha,t)(p)=\psi(\alpha,0)(e^{td}p)$, therefore
$$
\frac{\partial}{\partial t}\psi(\alpha,t)(p)=\frac{\partial}{\partial t}\psi(\alpha,0)(e^{td}p)=\psi(\alpha,0)(e^{td}dp)=\psi(\alpha,t)(dp).
$$

Now let $\psi(\alpha,t)(dp)=\frac{\partial}{\partial t}\psi(\alpha,t)(p)$.
Then
$$
\frac{\partial}{\partial t}\psi(\alpha,t_1)(e^{td}p)=\psi(\alpha,t_1)(d(e^{td}p))=\frac{\partial}{\partial t_1}\psi(\alpha,t_1)(e^{td}p)
$$
Hence $\left(\frac{\partial}{\partial t}-\frac{\partial}{\partial
t_1}\right)\psi(\alpha,t_1)(e^{td}p)=0$, so $\psi(\alpha,t_1)(e^{td}p)$
depends on $\alpha$ and $t+t_1$.

Thus
$\psi(\alpha,t_1)(e^{td}p)=\psi(\alpha,t+t_1)(e^{0d}p)=\psi(\alpha,t+t_1)(p)$.
\end{proof}
Proposition \ref{Simple} in the form $1, 2b$ was first proved in \cite{Buch}.

\section{Structure of $\mathcal{D}^*$}
\subsection{Main theorem}
\begin{defin}
Let us denote by $\mathcal{D}^*$ the graded dual Hopf algebra to
$\mathcal{D}$.
\end{defin}

Let us identify the Hopf algebra $\mathcal{M}$, which is the graded dual
$\mathcal{Z}^*$ to $\mathcal{Z}$, with $\Qsym[t_1,t_2,\dots]$.

The mappings $\EuScript{L}\colon\mathcal{Z}\to\mathcal{D}$,
$\EuScript{R}\colon\mathcal{Z}^{op}\to\mathcal{D}$, $Z_i\to d_i$ are
surjections, so the dual mappings are injections. So the ring homomorphisms
$\EuScript{L}^*$, and $\EuScript{R}^*=\varrho^*\circ\EuScript{L}^*$ define
two embeddings $\mathcal{D}^*\subset\Qs$, which are related by the involution
$\varrho^*=*$.

In fact, $\EuScript{L}^*$ defines an embedding $\mathcal{D}^*\subset\Qs$ as a
Hopf subalgebra, while for $\EuScript{R}^*$ we have
$$
\Delta\circ\EuScript{R}^*=\tau_{\Qs}\circ(\EuScript{R}^*\otimes\EuScript{R}^*)\circ\Delta,
$$
where $\tau_{\Qs}$ is a homomorphism $\Qs\otimes\Qs\to\Qs\otimes\Qs$ that
interchanges the tensor factors: $\tau_{\Qs}(x\otimes y)=y\otimes x$.
\begin{defin}
For each $r$ let us define the ring homomorphism
$\EuScript{R}_r^*:\mathcal{D}^*\to\Qsym[t_1,\dots,t_r]$ by the formula
$$
\EuScript{R}^*_r(\psi)=\EuScript{R}^*(\psi)(t_1,\dots,t_r,0,0,\dots)
$$
\end{defin}

Since the restriction $\Qs\to\Qsym[t_1,\dots,t_r]$ is injective for any
graduation $(2n)$, $n\leqslant r$, we see that $\EuScript{R}_r^*$ is
injective for graduations $(2n)$, $n\leqslant r$, of $\mathcal{D}^*$.

\begin{thm}
Let $r\geqslant n$. Then the image of the $(2n)$-th, $n\geqslant 1$, graded
component of the ring $\mathcal{D}^*$ in the ring $\Qsym[t_1,\dots,t_r]$
under the map $\EuScript{R}_r^*$ consists of all the homogeneous polynomials
of degree $2n$ satisfying the relations:
\begin{gather*}
g(t_1,-t_1,t_3,\dots,t_r)=g(0,0,t_3,\dots,t_r);\\
g(t_1,t_2,-t_2,t_4,\dots,t_r)=g(t_1,0,0,t_4,\dots,t_r);\\
\dots\\
g(t_1,\dots,t_{r-2},t_{r-1},-t_{r-1})=g(t_1,\dots,t_{r-2},0,0).
\end{gather*}\label{D*}
\end{thm}
\begin{proof}
The ring $\EuScript{R}^*(\mathcal{D}^*)\subset\mathcal Z^*$ consists of all
the linear functions $\psi\in\mathcal{Z}^*$ satisfying the property:
\begin{equation}
\langle \psi, z_1\Phi(t)\Phi(-t)z_2\rangle=\langle\psi, z_1z_2\rangle\label{DRel}
\end{equation}
for all $z_1,z_2\in\mathcal{Z}$, that is the coefficients of all
$t^k,\;k\geqslant 1$, on the left are equal to $0$.

We have
$$
\EuScript{R}^*(\psi)=\sum\limits_{\omega}\langle\EuScript{R}^*(\psi),Z_{\omega}\rangle=\sum\limits_{\omega}\langle\psi,\EuScript{R}Z_{\omega}\rangle M_{\omega}
$$
Therefore,
\begin{multline*}
\EuScript{R}^*_r(\psi)=\sum\limits_{k=0}^r\sum\limits_{(j_1,\,\dots,\,j_k)}\langle\psi,d_{j_k}\dots
d_{j_1}\rangle M_{(j_1,\,\dots,\,j_k)}(t_1,\dots,t_r)=\\
=\sum\limits_{k=0}^r\sum\limits_{(j_1,\,\dots,\,j_k)}\langle\psi,d_{j_k}\dots
d_{j_1}\rangle\sum\limits_{1\leqslant l_1<\dots<
l_k\leqslant r}t_{l_1}^{j_1}\dots t_{l_k}^{j_k}=\\
=\sum\limits_{k=0}^r\sum\limits_{(j_1,\,\dots,\,j_k)}\sum\limits_{1\leqslant
l_1<\dots<l_k\leqslant r}\langle\psi,d_{j_k}t_{l_k}^{j_k}\dots d_{j_1}
t_{l_1}^{j_1}\rangle=\langle\psi, \Phi(t_r)\dots\Phi(t_1)\rangle
\end{multline*}
Then
\begin{multline*}
\EuScript{R}^*_r(\psi)(t_1,\dots,t_i,-t_i,\dots,t_r)=\langle\psi,
\Phi(t_r)\dots\Phi(-t_i)\Phi(t_i)\dots\Phi(t_1)\rangle=\\
\langle\psi,
\Phi(t_r)\dots\Phi(0)\Phi(0)\dots\Phi(t_1)\rangle=\EuScript{R}^*_r(\psi)(t_1,\dots,0,0,\dots,t_r)
\end{multline*}

On the other hand, let $g\in\Qsym[t_1,\dots,t_r]$ be a homogeneous polynomial
of degree $2n$ satisfying all the relations of the theorem. Consider a unique
homogeneous function $\psi\in\mathcal{M}$ such that $\deg\psi=2n$, and
$\psi(t_1,\dots,t_r,0,0,\dots)=g$. Let us prove that $\psi$ belongs to the
image of $\mathcal{D}^*$ under the embedding $\EuScript{R}^*$.

It is sufficient to prove the relation (\ref{DRel}) in the case when
$z_1,z_2$ are monomials. Since
$$
\Phi(t)\Phi(-t)=1+(Z_1-Z_1)t+(2Z_2-Z_1^2)t^2+\dots=1+(2Z_2-Z_1^2)t^2+\dots,\mbox{ and }\deg\psi=2n,
$$
the cases $\deg z_1+\deg z_2=2n$ and $2(n-1)$ are trivial. Let
$$
z_1=Z_{\omega},\;z_2=Z_{\omega'},\;\omega=(j_1,\dots,j_l),\;\omega'=(j_1',\dots,j_{l'}'),\;|\omega|+|\omega'|=n-k,\;k\geqslant 2.$$
Then the only equality we need to prove is
$$
\langle \psi, Z_{\omega}\left(\sum\limits_{i=0}^{k}(-1)^iZ_{k-i}Z_i\right)Z_{\omega'}\rangle=0
$$
Let us consider the equality
$$
g(t_1,\dots,t_l,t_{l+1},-t_{l+1},t_{l+3},\dots,t_{l+2+l'},\dots,t_r)=g(t_1,\dots,t_l,0,0,t_{l+3},\dots,t_{l+2+l'},\dots,t_r).
$$
The coefficient of the monomial $t_1^{j_1}\dots
t_l^{j_l}t_{l+1}^{k}t_{l+3}^{j_1'}\dots t_{l+2+l'}^{j_{l'}'}$ on the left is
exactly
$$
\sum\limits_{i=0}^k(-1)^i\langle\psi, Z_{\omega}Z_{k-i}Z_iZ_{\omega'}\rangle=\langle\psi, Z_{\omega}\left(\sum\limits_{i=0}^{k}(-1)^iZ_{k-i}Z_i\right)Z_{\omega'}\rangle,
$$
and on the right it is equal to $0$. So $\langle\psi,
Z_{\omega}\Phi(t)\Phi(-t)Z_{\omega'}\rangle=0$ for all $\omega,\;\omega'$,
therefore, $\psi\in\EuScript{R}^*(\mathcal{D}^*)$.
\end{proof}

\begin{defin}
Let us define the operations $\Theta_k:\mathbb Z[t_1,t_2,\dots]\to\mathbb
Z[t,t_1,t_2,\dots]$ as
$$
\Theta_kg(t_1,t_2,\dots)=g(t_1,\dots,t_{k-1},t,-t,t_k,t_{k+1},\dots)
$$
\end{defin}

\begin{cor}\label{tD*}
The ring $\EuScript{R}^*(\mathcal{D}^*)\subset\Qs$ is defined by the
equations
$$
g\in\EuScript{R}^*(\mathcal{D}^*)\mbox{ if and only if }\;\Theta_kg=g \mbox{ for all } k\geqslant 1.
$$
\end{cor}
\begin{proof}
Let $g$ be a homogenous quasi-symmetric function of degree $2n$ representing
the function $\psi\in\mathcal{D}^*$. Since $g$ has bounded degree, each
monomial of $g(t_1,\dots,t_{k-1},t,-t,t_k,t_{k+1},\dots)$ appears with the
same coefficient in the polynomial
$g_{r+2}(t_1,\dots,t_{k-1},t,-t,t_k,\dots,t_r)=g(t_1,\dots,t_{k-1},t,-t,t_k,\dots,t_r,0,0,\dots)$,
when $r$ is large enough. Then Theorem \ref{D*} implies that
$$
g_{r+2}(t_1,\dots,t_{k-1},t,-t,t_k,\dots,t_r)=g_{r+2}(t_1,\dots,t_{k-1},0,0,t_k,\dots,t_r).
$$
Therefore each monomial containing $t$ has coefficient $0$, so
$g(t_1,\dots,t_{k-1},t,-t,t_k,\dots)$ does not depend on $t$. So we can set
$t=0$ to obtain
$$
g(t_1,\dots,t_{k-1},t,-t,t_k,\dots)=g(t_1,\dots,t_{k-1},0,0,t_k,\dots)=g(t_1,\dots,t_{k-1},t_k,\dots)
$$
On the other hand, let $g\in\Qs^{2n}$ be a quasi-symmetric function such that
$\Theta_kg=g$ for all $k\geqslant 1$. Then for
$g_n(t_1,\dots,t_n)=g(t_1,\dots,t_n,0,0,\dots)$ we have
\begin{multline*}
g_n(t_1,\dots,t_{i-1},t_i,-t_i,t_{i+2},\dots,t_n)=g(t_1,\dots,t_{i-1},t_i,-t_i,t_{i+2},\dots,t_n,0,0,\dots)=\\
=g(t_1,\dots,t_{i-1},t_{i+2},\dots,t_n,0,0,\dots)=g_n(t_1,\dots,t_{i-1},t_{i+2},\dots,t_n,0,0)=g_n(t_1,\dots,t_{i-1},0,0,t_{i+2},\dots,t_n)
\end{multline*}
for $1\leqslant i\leqslant n-1$. Then Theorem \ref{D*} implies that
$g_n=\EuScript{R}_n^*(\psi)$ for some function $\psi\in\mathcal{D}^*$ of
degree $2n$. Since the restriction $g(t_1,t_2,\dots)\to
g(t_1,\dots,t_n,0,0,\dots)$ is injective for graduation $2n$, we have
$g=\EuScript{R}^*\psi$.
\end{proof}

\begin{rem}
Since $\EuScript{R}^*=\varrho^*\circ\EuScript{L}^*$, and the subring
$\EuScript{R}^*(\mathcal{D}^*)\subset\Qs$ is invariant under the involution
$\varrho^*=*$, Theorem \ref{D*} and Corollary \ref{tD*} remain valid for the
mapping $\EuScript{L}^*$.
\end{rem}

\begin{prop}\label{DFM}
$\mathcal{D}^*\otimes\mathbb Q$ is a free polynomial algebra
$$
\mathcal{D}^*\otimes\mathbb Q\simeq\mathbb Q[\Lyn_{odd}],
$$
where $\Lyn_{odd}$ are Lyndon words consisting of odd positive integers. Rank
of the $(2n)$-th graded component of $\mathcal{D}^*\otimes\mathbb Q$ is equal
to the $(n-1)$-th Fibonacci number $c_{n-1}$.
\end{prop}

\begin{proof}
This follows from Proposition \ref{S}, Corollary \ref{Poly} to the shuffle
algebra structure theorem, and Corollary \ref{Fib}.
\end{proof}

\subsection{Applications}
Let $\mathcal{D}^*[\alpha]$ be a graded ring of polynomials in variable
$\alpha$, where $\deg \alpha=2$, and for the element $\psi_i$ of graduation
$2i$  the element $\psi_i\alpha^j$ has graduation $2(i+j)$.

Let us consider the mapping
$\varphi_{\alpha}\colon\mathcal{P}\to\mathcal{D}^*[\alpha]$, defined by the
formula
$$
\langle\varphi_{\alpha}(P),D\rangle=\xi_{\alpha}(DP)
$$
for $D\in\mathcal{D}$. Then
$$
\langle\varphi_{\alpha}(P\times Q),D\rangle=\xi_{\alpha}\left(D(P\times Q)\right)=\xi_{\alpha}\left(\sum\limits_{i}D'_i
P\times D''_i
Q\right)=\sum\limits_{i}\left(\xi_{\alpha}D_i'P\right)\left(\xi_{\alpha}D_i''Q\right)=
\langle\varphi_{\alpha}(P)\,\cdot\,\varphi_{\alpha}(Q),D\rangle,
$$
where $\Delta(D)=\sum\limits_{i}D_i'\otimes D_i''$. So $\varphi_{\alpha}$ is
a ring homomorphism.

\begin{prop}
We have $\EuScript{R}^*\circ\varphi_{\alpha}=f$, where $f$ is the generalized
$f$-polynomial.
\end{prop}
\begin{proof}
Indeed,
\begin{multline*}
\EuScript{R}^*\circ\varphi_{\alpha}(P)=\sum\limits_{\omega}\langle
\EuScript{R}^*\circ\varphi_{\alpha}(P), Z_{\omega}\rangle\,
M_{\omega}=\sum\limits_{\omega}\langle\varphi_{\alpha}(P),\EuScript{R}Z_{\omega}\rangle\,
M_{\omega}=\\
=\sum\limits_{\omega}\langle\varphi_{\alpha}(P), D_{\omega^*}\rangle\,
M_{\omega}=\sum\limits_{\omega}\xi_{\alpha}(D_{\omega}P)\,M_{\omega^*}=f(P)
\end{multline*}
\end{proof}

\begin{prop}
The image of the space $\mathcal{P}^{2n}$ under the mapping
$\varphi_{\alpha}:\mathcal{P}\to\mathcal{D}^*[\alpha]$ consists of all
homogeneous functions $\psi(\alpha)$ of degree $2n$ such that
\begin{equation}\label{X8}
\langle\psi(-\alpha),\Phi(\alpha)D\rangle=\langle\psi(\alpha),D\rangle
\end{equation}
for all $D\in\mathcal{D}$.
\end{prop}
\begin{proof}
According to Proposition \ref{Euler} we have
$\xi_{-\alpha}\Phi(\alpha)=\xi_{\alpha}$, so if
$\psi(\alpha)=\varphi_{\alpha}(p)$, then
$$
\langle\psi(-\alpha),\Phi(\alpha)D\rangle=\xi_{-\alpha}\Phi(\alpha)Dp=\xi_{\alpha}Dp=\langle\psi(\alpha),D\rangle.
$$
On the other hand, let condition (\ref{X8}) hold. Then the polynomial
$g(\alpha,t_1,\dots,t_n)=\EuScript{R}^*_n\psi(\alpha)$ satisfies the relation
\begin{multline*}
g(-\alpha,t_1,\dots,t_{n-1},\alpha)=\langle\psi(-\alpha),\Phi(\alpha)\Phi(t_{n-1})\dots\Phi(t_1)\rangle=\\
=\langle\psi(\alpha),\Phi(0)\Phi(t_{n-1})\dots\Phi(t_1)\rangle=g(\alpha,t_1,\dots,t_{n-1},0).
\end{multline*}
Theorem \ref{D*} implies that the polynomial $g$ satisfies all the relations
of Theorem \ref{f}, so $g=f_n(p^n)$ for some $p^n\in\mathcal{P}^{2n}$. We
have
$\EuScript{R}^*_n\psi(\alpha)=g=f_n(p^n)=\EuScript{R}^*_n\varphi_{\alpha}(p^n)$.
Since the restriction $\EuScript{R}^*\to \EuScript{R}^*_n$ is injective on
the $(2j)$-th graded component of the ring\; $\mathcal{D}^*$ for $j\leqslant
n$, and $\EuScript{R}^*$ is an embedding, we obtain that
$\psi(\alpha)=\varphi_{\alpha}(p^n)$.
\end{proof}
There is a right action of the ring $\mathcal{D}$ on its graded dual
$\mathcal{D}^*$:
$$
(\psi,D)\to\psi D\colon\quad \langle \psi D, D'\rangle=\langle \psi, DD'\rangle.
$$
If $\deg\psi=2n$ and $\deg D=2k$, then $\deg\psi D=2(n-k)$. Similarly, there
is an action of $\mathcal{D}[[\alpha]]$ on $\mathcal{D}^*[[\alpha]]$. Then
condition (\ref{X8}) means that $\psi(-\alpha)\Phi(\alpha)=\psi(\alpha)$.
Since for any homogeneous function $\psi(\alpha)$ of degree $2n$ the function
$\psi(\alpha)\Phi(\alpha)$ still has degree $2n$, we obtain the corollary.
\begin{cor} Let $\psi(\alpha)\in\mathcal{D}^*[\alpha]$. Then
\begin{equation}\label{X9}
\psi(\alpha)\in \varphi_{\alpha}(\mathcal{P})\Leftrightarrow \psi(-\alpha)\Phi(\alpha)=\psi(\alpha).
\end{equation}\label{Im}
\end{cor}
Let $\psi(\alpha)=\psi_0+\psi_1\alpha+\dots+\psi_n\alpha^n$,
$u(\alpha)=\psi_0+\psi_2\alpha^2+\psi_4\alpha^4+\dots$,
$w(\alpha)=\psi_1\alpha+\psi_3\alpha^3+\dots$. Then
$u(\alpha)\Phi(\alpha)-w(\alpha)\Phi(\alpha)=u(\alpha)+w(\alpha)$. Therefore,
\begin{equation}\label{X10}
u(\alpha)\left(\Phi(\alpha)-1\right)=w(\alpha)\left(\Phi(\alpha)+1\right).
\end{equation}
For example,
$$
\psi_0d=2\psi_1;\quad \psi_0d_3+\psi_2d=\psi_1d_2+2\psi_3.
$$

Consider the Hopf algebra $\mathcal{D}\otimes\mathbb Z[\frac{1}{2}]$. Then
for the graded dual over $\mathbb Z[\frac{1}{2}]$ Hopf algebra we have:
$(\mathcal{D}\otimes \mathbb
Z[\frac{1}{2}])^*\simeq\mathcal{D}^*\otimes\mathbb Z[\frac{1}{2}]$. The
condition that describes the image of the ring $\mathcal{P}\otimes\mathbb
Z[\frac{1}{2}]$ in $\mathcal{D}^*\otimes\mathbb Z[\frac{1}{2}][\alpha]$ is
the same, namely relation (\ref{X10}).

\begin{prop}
In the ring $\mathcal{D}^*\otimes\mathbb Z[\frac{1}{2}][\alpha]$ relation
(\ref{X10}) is equivalent to the relation:
\begin{equation}\label{X11}
w(\alpha)=u(\alpha)\frac{\Phi(\alpha)-1}{\Phi(\alpha)+1}=u(\alpha)\frac{\frac{\Phi(\alpha)-1}{2}}{1+\frac{\Phi(\alpha)-1}{2}}=u(\alpha)\sum\limits_{k=1}^{\infty}(-1)^{k-1}\left(\frac{\Phi(\alpha)-1}{2}\right)^k.
\end{equation}
\end{prop}
\begin{rem} In fact, equation (\ref{X11}) implies that the function $w(\alpha)$ is odd,
if the function $u(\alpha)$ is even. Indeed,
\begin{gather*}
w(-\alpha)=u(-\alpha)\frac{\Phi(-\alpha)-1}{\Phi(-\alpha)+1}=u(\alpha)\frac{\frac{1}{\Phi(\alpha)}-1}{\frac{1}{\Phi(\alpha)}+1}=u(\alpha)(1-\Phi(\alpha))\Phi(\alpha)^{-1}(1+\Phi(\alpha))^{-1}\Phi(\alpha)=\\
=-u(\alpha)\frac{\Phi(\alpha)-1}{\Phi(\alpha)+1}=-w(\alpha).
\end{gather*}
Therefore for any even function $u(\alpha)$ the function
$\psi(\alpha)=u(\alpha)+u(\alpha)\frac{\Phi(\alpha)-1}{\Phi(\alpha)+1}=u(\alpha)\frac{2\Phi(\alpha)}{\Phi(\alpha)+1}$
belongs to the image of $\varphi_{\alpha}$
\end{rem}
\begin{rem}\label{uvQ}
The same relation is valid in $\mathcal{D}^*\otimes\mathbb Q[\alpha]$.
\end{rem}

Let us take $\varphi_{\alpha}(p)$ with $\alpha=0$. Then we obtain the
classical map
$$
\varphi_0:\mathcal{P}\to\mathcal{D}^*:\quad p\to\varphi_{0}(p)\quad\langle\varphi_0(p), D\rangle=\xi_{0}Dp,\;\forall p\in\mathcal{P},\;D\in\mathcal{D}
$$
\begin{prop}
The mapping $\varphi_0\otimes 1:\mathcal{P}\otimes\mathbb
Z[\frac{1}{2}]\to\mathcal{D}^*\otimes\mathbb Z[\frac{1}{2}]$ is a surjection.
\end{prop}
\begin{proof}
Let $\psi_0\in\mathcal{D}^*\otimes\mathbb Z[\frac{1}{2}]$. Consider the
element $\psi(\alpha)=u(\alpha)+w(\alpha)$, $u(\alpha)=\psi_0$,
$w(\alpha)=\psi_0\frac{\Phi(\alpha)-1}{\Phi(\alpha)+1}$. According to
$\mathbb Z[\frac{1}{2}]$-version of Corollary \ref{Im} we obtain that
$\psi(\alpha)=\varphi_{\alpha}(p)\in\varphi_{\alpha}\otimes
1(\mathcal{P}\otimes\mathbb Z[\frac{1}{2}])$. Then
$\psi_0=\psi(0)=\varphi_0(p)$.
\end{proof}
\begin{quest}
What is the image of the map $\varphi_0$ over the integers?
\end{quest}

\begin{rem}
Let us note, that the space $\mathcal{P}/\Ker{\varphi_{\alpha}}$ consists of
the equivalence classes of integer combinations of polytopes under the
equivalence relation: $p\sim q$ if and only if $p$ and $q$ have equal flag
$f$-vectors. Rank of the $(2n)$-th graded component of the group
$\mathcal{P}/\Ker{\varphi_{\alpha}}$ is equal to $c_n$ (\cite{BB}, see also
Section 2.5).

On the other hand, by Corollary \ref{Fib} rank of the $(2n)$-th graded
component of the ring $\mathcal{D}^*$ is equal to $c_{n-1}$ , so the mapping
$\varphi_0:\mathcal{P}/\Ker\varphi_{\alpha}\to\mathcal{D}^*$ is not
injective.
\end{rem}
\begin{exam}
Let us consider small dimensions.
\begin{itemize}
\item $n=1$.  $\mathcal{P}/\Ker\varphi_{\alpha}$ is generated by $CC=I$.
    The ring $\mathcal{D}$ in this graduation is generated by $d$. Then
    $\varphi_0(I)=2d^*$.
\item $n=2$. $\mathcal{P}/\Ker\varphi_{\alpha}$ is generated by
    $CCC=\Delta^2$ and $BCC=I^2$. $\mathcal{D}$ has one generator $d_2$.
    Then $\varphi_0(\Delta^2)=3d_2^*$, $\varphi_0(I^2)=4d_2^*$.
    Therefore, $d_2^*=\varphi_0(I^2-\Delta^2)$, $\Ker\varphi_0$ is
    generated by $3I^2-4\Delta^2$
\item $n=3$. $\mathcal{P}/\Ker\varphi_{\alpha}$ is generated by $BCCC,
    CBCC=CI^2, CCCC=\Delta^3$, while $\mathcal{D}$ is generated by
    $d_3,d_2d$. Then
$$
\varphi_0(BC^3)=5d_3^*+18(d_2d)^*,\quad\varphi_0(CBC^2)=5d_3^*+16(d_2d)^*,\quad \varphi_0(C^4)=4d_3^*+12(d_2d)^*
$$
So $\im \varphi_0$ has the basis $d_3^*,\,2(d_2d)^*$ and $\Ker\varphi_0$
is generated by $2BC^3-6CBC^2+5C^4$.
\end{itemize}
Here by $D_{\omega}^*$ we denote the element of the dual basis: $\langle
D_{\omega}^*,D_{\sigma}\rangle=\delta_{\omega,\,\sigma}$
\end{exam}

\section{Multiplicative Structure of $f(\mathcal{P})\otimes\mathbb Q$}
\begin{thm}
The ring $f(\mathcal{P})\otimes\mathbb Q$ is a free polynomial algebra.
\end{thm}
\begin{proof}
According to Proposition \ref{DFM} we have:
$$
\mathcal{D}^*\otimes\mathbb Q\simeq\mathbb Q[\Lyn_{odd}],
$$
where $\Lyn_{odd}$ are Lyndon words consisting of odd positive integers. This
is a free polynomial algebra. Therefore $\mathcal{D}^*\otimes\mathbb
Q[\alpha]$ is a free polynomial algebra in the generators $\Lyn_{odd}$ and
$\alpha$.

Let $\{f_{\lambda}\}$ be the functions in $\mathcal{D}^*\otimes\mathbb Q$
corresponding to the Lyndon words. Consider the polynomials
$f_{\lambda}(\alpha)$ defined as
$$
f_{\lambda}(\alpha)=u_{\lambda}(\alpha)+w_{\lambda}(\alpha);\quad u_{\lambda}(\alpha)=f_{\lambda},\quad w_{\lambda}(\alpha)=f_{\lambda}\frac{\Phi(\alpha)-1}{\Phi(\alpha)+1}=f_{\lambda}\frac{e^{s(\alpha)}-1}{e^{s(\alpha)}+1}
$$
Each polynomial $f_{\lambda}(\alpha)$ is a homogeneous element of degree
$\deg f_{\lambda}$. Let us identify $\mathcal{D}^*$ with its image in
$\Qsym[t_1,t_2,\dots]$ under the embedding $\EuScript{R}^*$.

Then according to Remark \ref{uvQ} we have: $f_{\lambda}(\alpha)\in
f(\mathcal{P}\otimes\mathbb Q)=f(\mathcal{P})\otimes\mathbb Q$.
\begin{lemma}
$f(\mathcal{P})\otimes\mathbb Q=\mathbb Q[f_{\lambda}(\alpha),\alpha^2]$.
\end{lemma}
\begin{proof}
At first let us proof that the polynomials
$\{f_{\lambda}(\alpha)\},\,\alpha^2$ are algebraically independent.

Let $g\in\mathbb Q[x,y_1,y_2,\dots,y_s]$ be a polynomial such that
$g(\alpha^2,f_1(\alpha),\dots,f_s(\alpha))=0$. We can write this equality as
$$
g_0(f_1(\alpha),\dots,f_s(\alpha))+g_2(f_1(\alpha),\dots,f_s(\alpha))\alpha^2+\dots+g_{2t}(f_1(\alpha),\dots,f_s(\alpha))\alpha^{2t}=0
$$
for some integers $s,t$. Let us set $\alpha=0$. Then we obtain
$$
g_0(f_1,\dots,f_s)=0
$$
Since $\{f_{\lambda}\}$ are algebraically independent, $g_0\equiv 0$. Since
$\mathcal{D}^*[\alpha]$ is a polynomial ring, we can divide the equality by
$\alpha^2$ to obtain
$$
g_2(f_1(\alpha),\dots,f_s(\alpha))+g_4(f_1(\alpha),\dots,f_s(\alpha))\alpha^2+\dots+g_{2t}(f_1(\alpha),\dots,f_s(\alpha))\alpha^{2t-2}=0
$$
Iterating this step in the end we obtain that all the polynomials
$g_{2i}\equiv 0$, $i=0,1,\dots,2t$, so $g\equiv0$.

Now let us prove that the polynomials $\{f_{\lambda}(\alpha)\},\alpha^2$
generate $f(\mathcal{P})\otimes \mathbb Q$.

Let $\psi(\alpha)=\psi_0+\psi_1\alpha+\dots+\psi_n\alpha^n\in
f(\mathcal{P})\otimes\mathbb Q$ be a homogeneous element of degree $2n$. The
functions $\{f_{\lambda}\}$ generate $\mathcal{D}^*\otimes\mathbb Q$, so we
can find a homogeneous polynomial $g_0\in\mathbb Q[y_1,\dots,y_s]$ such that
$$
g_0(f_1,\dots,f_s)=\psi_0,\quad \deg f_i\leqslant 2n
$$
Let us consider the element
$\psi(\alpha)-g_0(f_1(\alpha),\dots,f_s(\alpha))\in f(\mathcal{P})\otimes
\mathbb Q$. Since $\theta_1=\theta_0\frac{d}{2}$ for any
$$
\theta(\alpha)=\theta_0+\theta_1\alpha+\theta_2\alpha^2+\dots+\theta_r\alpha^r\in f(\mathcal{P})\otimes \mathbb Q,
$$
we obtain that
$$
\psi(\alpha)-g_0(f_1(\alpha),\dots,f_s(\alpha))=\hat\psi(\alpha)\alpha^2.
$$
is a homogeneous element of degree $2n$. $\hat\psi(\alpha)\alpha^2\in
f(\mathcal{P})\otimes\mathbb Q$, so
$\hat\psi(-\alpha)\alpha^2\Phi(\alpha)=\hat\psi(\alpha)\alpha^2$. Since
$\mathcal{D}^*\otimes\mathbb Q[\alpha]$ is a polynomial ring,
$\hat\psi(-\alpha)\Phi(\alpha)=\hat\psi(\alpha)$, so $\hat\psi(\alpha)\in
f(\mathcal{P})\otimes \mathbb Q$ is a homogeneous element of degree $2(n-2)$.

Iterating this argument in the end we obtain the expression of $\psi(\alpha)$
as a polynomial in $\{f_{\lambda}(\alpha)\}$ and $\alpha^2$.
\end{proof}
This lemma proves the theorem.
\end{proof}

Let us note that dimension of the $(2n)$-th graded component of the ring
$f(\mathcal{P})\otimes\mathbb Q$ is equal to the $n$-th Fibonacci number
$c_n$.

\begin{cor} There is an isomorphism of free polynomial algebras
$$
f(\mathcal{P})\otimes\mathbb Q\simeq \mathbb Q[\Lyn_{odd},\alpha^2]=\mathcal{N}_{odd}\otimes\mathbb Q[\alpha^2]\simeq \mathbb Q[\Lyn_{12}]=\mathcal{N}_{12}\otimes\mathbb Q.
$$\label{three}
\end{cor}
\begin{proof}
Indeed, these three algebras are free polynomial algebras with the same
dimensions of the corresponding graded components: dimension of the $(2n)$-th
graded component of each algebra is equal to $c_n$
\end{proof}
Let us denote by $k_n$ the number of multiplicative generators of degree
$2n$. According to Corollary \ref{three} for $n\geqslant 3$ the number of
generators $k_n$ is equal to the number of Lyndon words of degree $2n$ in
$\Lyn_{odd}$ or in $\Lyn_{12}$.

For example, for small $n$ we have
\begin{center}
\begin{tabular}{|c|c|c|}
\hline
n&$\Lyn^{2n}_{12}$&$\Lyn^{2n}_{odd}$\\
\hline
3&[1,2]&[3]\\
4&[1,1,2]&[1,3]\\
5&[1,2,2],\;[1,1,1,2]&[5],\;[1,1,3]\\
6&[1,1,2,2],\;[1,1,1,1,2]&[1,5],\;[1,1,1,3]\\
7&[1,2,2,2],\;[1,1,2,1,2],\;[1,1,1,2,2],\;[1,1,1,1,1,2]&[7],\;[1,1,5],\;[1,3,3],\;[1,1,1,1,3],\;\\
\hline
\end{tabular}
\end{center}

\begin{cor}
\begin{enumerate}
\item $k_{n+1}\geqslant k_{n}$
\item $k_n\geqslant N_n-2$, where $N_n$ is the number of different
    decompositions of $n$ into the sum of odd numbers.
\end{enumerate}
\end{cor}
\begin{rem}
It is a well-known fact, that the number of the decompositions of $n$ into
the sum of odd numbers is equal to the number of the decompositions of $n$
into the sum of different numbers. For example,
$$
\begin{matrix}
2&=&1+1&=&2\\
3&=&1+1+1,\;3&=&1+2,\;3\\
4&=&1+1+1+1,\;1+3&=&1+3,\;4\\
5&=&1+1+1+1+1,\;1+1+3,\;5&=&1+4,\;2+3,\;5
\end{matrix}
$$
\end{rem}
\begin{proof}
\begin{enumerate}
\item $k_1=k_2=k_3=1$.

Let $n\geqslant 3$ and let $w=[a_1,\dots,a_k]\in\Lyn^{2n}_{12}$. Then
$1w=[1,a_1,\dots,a_k]\in\Lyn^{2(n+1)}_{12}$. Indeed, for any proper tail
$[a_i,\dots,a_k],\; i\geqslant 1$ if $a_i>1$, then $[a_i,\dots,a_k]>1w$.
If $a_i=1$ then $i\ne k$ and $[a_{i+1},\dots,a_k]>[a_1,\dots,a_k]$, since
$w$ is Lyndon. Thus $k_{n+1}\geqslant k_n$
\item Any decomposition $n=d_1+\dots+d_k$ into the sum of odd numbers
    $d_1\leqslant d_2\leqslant\dots\leqslant d_k$ gives a Lyndon word
    $[d_1,\dots,d_k]\in\Lyn_{odd}$, except for the case, when $k>1$,
    $d_1=\dots=d_k=d\ne n$.

If $d\geqslant 5$, then the word $[1,d-2,1,d,\dots,d]$ is Lyndon. When
$d=1$ or $d=3$ there can be no Lyndon word corresponding to the
decomposition $n=1+\dots+1$ or $n=3+\dots+3$, as it happens for $n=6$:
$$
6=1+5=3+3=1+1+1+3=1+1+1+1+1+1
$$
We have $k_6=2$ and $N_6=4$. Thus $k_n\geqslant N_n-2$.
\end{enumerate}
\end{proof}
Then we obtain the corollary:
\begin{cor}
For any $t$, $|t|<\frac{\sqrt{5}-1}{2}$, there is the following decomposition
of the generating series of Fibonacci numbers into the absolutely convergent
infinite product:
$$
\frac{1}{1-t-t^2}=1+t+t^2+2t^3+3t^4+\dots=\sum\limits_{n=0}^{\infty}c_nt^n=\prod\limits_{i=1}^{\infty}\frac{1}{(1-t^i)^{k_i}},
$$
The numbers $k_n$ satisfy the properties $k_{n+1}\geqslant k_n\geqslant
N_n-2$, where $N_n$ is the number of decompositions of $n$ into the sum of
odd summands.
\end{cor}
\begin{proof}
Let $|t|<\frac{\sqrt{5}-1}{2}$. Then the series
$\sum\limits_{n=0}^{\infty}c_nt^n$ converges absolutely to
$\frac{1}{1-t-t^2}$. Then
$$
R_N=\sum\limits_{n=N}^{\infty}c_n|t^n|\xrightarrow[N\to\infty]{} 0.
$$
But we have:
$$
S_n=\sum\limits_{k=0}^{\infty}c_kt^k-\prod\limits_{i=1}^{n-1}\frac{1}{(1-t^i)^{k_i}}=a_nt^n+a_{n+1}t^{n+1}+a_{n+2}t^{n+2}+\dots,
$$
where $0\leqslant a_i\leqslant c_i$. Then $|S_n|\leqslant R_n$, therefore
$S_n\xrightarrow[n\to\infty]{}0$, and the formula is valid.

The infinite product
$\lim\limits_{n\to\infty}\prod\limits_{i=1}^n\frac{1}{(1-t^i)^{k_i}}$
converges absolutely if and only if the series
$\sum\limits_{n=1}^{\infty}|-k_n\ln(1-t^n)|$ converges. For $0\leqslant
t<\frac{\sqrt{5}-1}{2}$ this is true. In this case $0\leqslant
k_nt^n\leqslant-k_n\ln(1-t^n)$, therefore the power series
$\sum\limits_{n=1}^{\infty}k_nt^n$ converges and its radius of convergence it
at least $\frac{\sqrt{5}-1}{2}$. Thus for $|t|<\frac{\sqrt{5}-1}{2}$ we have:
$\sum\limits_{n=1}^{\infty}k_n|t^n|<\infty$,  and the equivalence
$|-k_n\ln(1-t^n)|\sim k_n|t^n|$ implies the absolute convergence of the
infinite product.
\end{proof}

Here are the examples in small dimensions:
\begin{center}
\begin{tabular}{|c|c|c|c|c|}
\hline n&generators&$k_n$&additive
basis&$c_n$\\
\hline
1&[1]&1&[1]&1\\

2&[2]&1&[2],\;[1]*[1]&2\\

3&[1,2]&1&[1,2],\;[2]*[1],\;[1]*[1]*[1]&3\\

4&[1,1,2]&1&[1,1,2],\;[1,2]*[1],\;[2]*[2],\;[2]*[1]*[1],\;[1]*[1]*[1]*[1]&5\\
\hline
5&[1,2,2],\;[1,1,1,2]&2&[1,2,2],\;[1,1,1,2],\;[1,1,2]*[1],\;[1,2]*[2],\;[1,2]*[1]*[1]&8\\
&&&[2]*[2]*[1],\;[2]*[1]*[1]*[1],\;[1]*[1]*[1]*[1]*[1]&\\
\hline
6&[1,1,2,2],\;[1,1,1,1,2]&2&&13\\
7&[1,2,2,2],\;[1,1,2,1,2],\;[1,1,1,2,2],\;[1,1,1,1,1,2]&4&&21\\
\hline
\end{tabular}
\end{center}
The corresponding products are:
\begin{gather*}
(1-t)(1-t^2)=1-t-t^2+t^3;\\
(1-t)(1-t^2)(1-t^3)=1-t-t^2+t^4+t^5-t^6;\\
(1-t)(1-t^2)(1-t^3)(1-t^4)=1-t-t^2+2t^5-t^8-t^9+t^{10};\\
(1-t)(1-t^2)(1-t^3)(1-t^4)(1-t^5)^2=1-t-t^2+2t^6+2t^7-t^8-t^9-2t^{10}-\\
-t^{11}-t^{12}+2t^{13}+2t^{14}-t^{18}-t^{19}+t^{20}.
\end{gather*}
The numbers $k_i$ can be found in the following way:
\begin{gather*}
-\log(1-t-t^2)=-\sum\limits_{i=1}^{\infty}k_i\log(1-t^i);\\
\sum\limits_{n=1}^{\infty}\frac{1}{n}\sum\limits_{j=0}^n{n\choose
j}t^{n+j}=\sum\limits_{i=1}^{\infty}k_i\sum\limits_{r=1}^{\infty}\frac{t^{ri}}{r}
\end{gather*}
Thus for any $N$
$$
\sum\limits_{j=0}^{[\frac{N}{2}]}\frac{{N-j\choose
j}}{N-j}=\frac{1}{N}\sum\limits_{i|N}ik_i.
$$
According to the M\"obius inversion formula we obtain
$$
Nk_N=\sum\limits_{d|N}\left(d\sum\limits_{j=0}^{[\frac{d}{2}]}\frac{{d-j\choose
j}}{d-j}\right)\cdot\mu\left(\frac{N}{d}\right),
$$
where $\mu(n)$ is the M\"obius function, that is
$$
\mu(n)=
\begin{cases}
1,&n=1;\\
(-1)^r,&n=p_1\dots p_r,\;\{p_i\}\mbox{ -- distinct prime
numbers};\\
0,&n\mbox{ is not square-free}.
\end{cases}
$$
For example, if $N$ is a prime number $p$, then
$$
pk_p=-1+p\sum\limits_{j=0}^{[\frac{p}{2}]}\frac{{p-j\choose
j}}{p-j};\quad
k_p=\sum\limits_{j=1}^{[\frac{p}{2}]}\frac{{p-j\choose j}}{p-j}.
$$
Each summand is an integer, since ${p-j\choose j}=\frac{{p-j-1\choose
j-1}(p-j)}{j}$ is an integer, and $(p-j)$ and $j$ are relatively prime
numbers.

Thus $k_5=2,\; k_7=4,\;k_{11}=18$, and so on.

\section{Bayer-Billera Ring} Let us consider the free graded abelian group
$BB\subset\mathcal{P}$, generated by $1$ and all the polytopes
$Q\in\Omega^n,\;n\geqslant 1$. Rank of the $(2n)$-th graded component of this
group is equal to $c_n$.

Since the determinant of the matrix $K^n$ is equal to $1$, the generalized
$f$-polynomials $\{f(Q),\;Q\in\Omega^n\}$ form a basis of the $(2n)$-th
graded component of the ring $f(\mathcal{P})$.

So the composition of the inclusion $i:BB\subset\mathcal{P}$ and the mapping
$f:\mathcal{P}\to\Qsym[t_1,t_2,\dots][\alpha]$ is an isomorphism of the
abelian groups $BB$ and $f(\mathcal{P})$. This gives a projection
$\pi:\mathcal{P}\to BB$
$$
\pi(p)=x\in BB:\quad f(p)=f(x).
$$
It follows from the definition, that $\pi\circ i=1$ on the space $BB$.
\begin{thm}
The projection $\pi\colon\mathcal{P}\to BB$ defined by the relation $f(\pi
p)=f(p)$ gives the graded group $BB$ the structure of a graded commutative
associative ring with the multiplication $x\times_{BB}y=\pi(x\times y)$ such
that $f(x\times_{BB}y)=f(x)f(y)$, and $BB\otimes\mathbb Q$ is a free
polynomial algebra in a countable set of variables.
\end{thm}

\begin{quest}
To describe the projection $\pi\colon\mathcal{P}\to BB$ in an efficient way.
\end{quest}
\begin{quest}
To find the multiplicative generators of $BB\otimes\mathbb Q$ in an efficient
way.
\end{quest}
\begin{quest}
To describe the multiplication in $BB$ in an efficient way.
\end{quest}

\section{Hopf comodule structures}
Let $R$ be a field or the ring $\mathbb Z$.
\begin{defin}
By a (left) {\it Hopf comodule} (or {\it Milnor comodule}) over a Hopf
algebra $X$ we mean an $R$-algebra $M$ with a unit provided $M$ is a comodule
over $X$ with a coaction $b\colon M\to X\otimes M$ such that
$b(uv)=b(u)b(v)$, i.e. such that $b$ is a homomorphism of rings.
\end{defin}

\subsection{Hopf comodule structure arising from the Hopf module structure}
\begin{lemma}\label{CM}
Let $\mathcal{H}$ be a graded connected torsionfree Hopf algebra over
$\mathbb Z$ with finite rank of each graded component. If a ring
$\mathcal{A}$ has the right graded Milnor module structure over
$\mathcal{H}$, then $\mathcal{A}$ has a left graded Hopf comodule structure
over the graded dual Hopf algebra $\mathcal{H}^*$ defined by the formula
$$
\Delta_{\mathcal{H}^*}(A)=\sum\limits_{H_{\omega}\in\mathcal{H}}H^*_{\omega}\otimes AH_{\omega}
$$
where $\{H_{\omega}\}$ is some basis in the graded group $\mathcal{H}$, and
$\{H_{\omega}^*\}$ is the graded dual basis: $\langle
H_{\omega}^*,H_{\sigma}\rangle=\delta_{\omega,\,\sigma}$.
\end{lemma}
\begin{proof}
\begin{multline*}
(1\otimes\Delta_{\mathcal{H}^*})\circ\Delta_{\mathcal{H}^*}A=\sum\limits_{\omega}H^*_{\omega}\otimes
\left(\sum\limits_{\sigma}H^*_{\sigma}\otimes (AH_{\omega})H_{\sigma}\right)=
\sum\limits_{\omega,\,\sigma}H^*_{\omega}\otimes H^*_{\sigma}\otimes
A\left(\sum\limits_{\tau}\,\langle H^*_{\tau}, H_{\omega}H_{\sigma}\rangle
H_{\tau}\right)=\\
=\sum\limits_{\tau}\left(\sum\limits_{\omega,\,\sigma}H_{\omega}^*\otimes
H_{\sigma}^*\;\langle H^*_{\tau}, H_{\omega} H_{\sigma} \rangle \right)
\otimes AH_{\tau}
=\sum\limits_{\tau}\left(\sum\limits_{\omega,\,\sigma}H_{\omega}^*\otimes
H_{\sigma}^*\;\langle \Delta H^*_{\tau}, H_{\omega}\otimes
H_{\sigma}\rangle\right)\otimes AH_{\tau}=\\
=\sum\limits_{\tau}\Delta H^*_{\tau}\otimes AH_{\tau}=(\Delta\otimes
1)\circ\Delta_{\mathcal{H}^*}A
\end{multline*}

Since the augmentation in the ring $\mathcal{H}^*$ has the form
$\varepsilon(\xi)=\xi(1)$, we see that
\begin{multline*}
(\varepsilon\otimes 1)\Delta_{\mathcal{H}^*}A=(\varepsilon \otimes
1)\sum\limits_{\omega} H^*_{\omega}\otimes
AH_{\omega}=\sum\limits_{\omega}H_{\omega}^*(1)\otimes
AH_{\omega}=H_{\omega_0}^*(1)\otimes AH_{\omega_0}=1\otimes A,
\end{multline*}
where $H_{\omega_0}$ is a basis in $\mathcal{H}^0$.

Therefore $\Delta_{\mathcal{H}^*}$ is a comodule structure.

Now
\begin{multline*}
\Delta_{\mathcal{H}^*}(AB)=\sum\limits_{\omega}H^*_{\omega}\otimes(AB)H_{\omega}=\sum\limits_{\omega}H^*_{\omega}
\otimes\left(\sum\limits_{i}(A H_{\omega,\,i}')(B H_{\omega,\,i}'')\right)=\\
=\sum\limits_{\omega}H^*_{\omega}\otimes\left(\sum\limits_{i}\left(A
\sum\limits_{\sigma}\langle H_{\sigma}^*, H_{\omega,\,i}'\rangle
H_{\sigma}\right)\left(B\sum\limits_{\tau}\langle H_{\tau}^*,
H_{\omega,\,i}''\rangle  H_{\tau}\right)\right)=\\
=\sum\limits_{\sigma,\,\tau}\left(\sum\limits_{\omega}\left(\sum\limits_{i}
\langle H_{\sigma}^*, H_{\omega,\,i}'\rangle \langle H_{\tau}^*,
H_{\omega,\,i}''\rangle\right)H^*_{\omega}\right)\otimes (AH_{\sigma})(B H_{\tau})=\\
\sum\limits_{\sigma,\,\tau}\left(\sum\limits_{\omega}\langle
H_{\sigma}^*\cdot H_{\tau}^*,\,H_{\omega}\rangle\,H^*_{\omega}\right)\otimes
(AH_{\sigma})(B H_{\tau})=\sum\limits_{\sigma,\,\tau} H^*_{\sigma}\cdot
H^*_{\tau}\otimes (AH_{\sigma})(B
H_{\tau})=\left(\Delta_{\mathcal{H}^*}A\right)\cdot
\left(\Delta_{\mathcal{H}^*}B\right),
\end{multline*}
where $\Delta H_{\omega}=\sum\limits_{i}H_{\omega,\,i}'\otimes
H_{\omega,\,i}''$. This equality finishes the proof.
\end{proof}
\begin{rem}
This Hopf comodule structure does not depend on the choice of a graded basis
in the graded group $\mathcal{H}$, since we can write it as
$$
\Delta_{\mathcal{H}^*}A=(1\otimes A)\left(\sum\limits_{\omega}H^*_{\omega}\otimes H_{\omega}\right),
$$
and $\sum\limits_{\omega}H^*_{\omega}\otimes H_{\omega}$ defines in each
graded component $\mathcal{H}^n$ the identity operator in
$(\mathcal{H}^{n})^*\otimes\mathcal{H}^n\simeq\Hom_{\mathbb
Z}(\mathcal{H}^n,\mathcal{H}^n)$.
\end{rem}
\begin{rem}
The similar argument proves that if a ring $\mathcal{A}$ has the left graded
Milnor module structure over $\mathcal{H}$, then $A$ has a right graded Hopf
comodule structure over the graded dual Hopf algebra $\mathcal{H}^*$ defined
by the formula
$$
\Delta_{\mathcal{H}^*}(A)=\sum\limits_{H_{\omega}\in\mathcal{H}}H_{\omega}A\otimes H_{\omega}^*.
$$
\end{rem}

\begin{cor}\label{ZD}
There ia a natural left Hopf comodule structure on the ring $\R$ over the
Hopf algebra $\Qs$ induced by the right Hopf module structure over the Hopf
algebra $\mathcal{Z}$:
$$
\Delta_{\EuScript{R}}(P)=\sum\limits_{\omega}M_{\omega}\otimes
PZ_{\omega}=\sum\limits_{k\geqslant
0}\sum\limits_{(j_1,\,\dots,\,j_k)}M_{(j_1,\,\dots,\,j_k)}\otimes
\left(PZ_{j_1}\dots Z_{j_k}\right)=\sum\limits_{k\geqslant
0}\sum\limits_{(j_1,\,\dots,\,j_k)}M_{(j_k,\,\dots,\,j_1)}\otimes
\left(d_{j_1}\dots d_{j_k}P\right)
$$
There are natural right Hopf comodule structures $\Delta_{\EuScript{L}}$ and
$\Delta_{\mathcal{D}^*}$ on the ring $\R$ over the Hopf algebras $\Qs$ and
$\mathcal{D}^*$ induced by the left Hopf module structures over the Hopf
algebras $\mathcal{Z}$ and $\mathcal{D}$.

These structures are compatible in the following sense:
\begin{gather}
\Delta_{\EuScript{L}}=\tau\circ(*\otimes1)\circ\Delta_{\EuScript{R}}\\
(1\otimes\EuScript{L}^*)\circ\Delta_{\mathcal{D}^*}=\Delta_{\EuScript{L}}\label{LD}
\end{gather}
where $\tau\colon\Qs\otimes \R\to\R\otimes\Qs$ is a homomorphism that
interchanges the tensor factors
$$
\tau(M_{\omega}\otimes P)=P\otimes M_{\omega}.
$$
We have
$$
\Delta_{\EuScript{L}}P=\sum\limits_{k\geqslant 0}\sum\limits_{(j_1,\,\dots,\,j_k)}\left(d_{j_1}\dots d_{j_k}P\right)\otimes M_{(j_1,\,\dots,\,j_k)}
$$
\end{cor}
\begin{proof}
The existence of these structures and the explicit formulas follow
immediately from Lemma \ref{CM} and Corollary \ref{ZD}. The first relation
follows easily from the formulas for $\Delta_{\EuScript{L}}$ and
$\Delta_{\EuScript{R}}$.

Since $\mathcal{D}\simeq\mathcal{Z}/J_{\mathcal{U}}$ is a free abelian group,
we have $\mathcal{Z}=\mathcal{Z}'\oplus J_{\mathcal{U}}$. Let us choose
graded bases in both subgroups: $\{H_{\beta}\}$ in $\mathcal{Z}'$, and
$\{H_{\gamma}\}$ in $J_{\mathcal{U}}$. Then $\{H_{\beta}, H_{\gamma}\}$ form
a graded basis in $\mathcal{Z}$.

Let us note that $\EuScript{L}^*(\mathcal{D}^*)=\{\psi\in\Qs:
\left.\psi\right|_{J_{\mathcal{U}}}=0\}$.

Then the elements of the dual basis $\{H_{\beta}^*\},\{H_{\gamma}^*\}$ form a
basis in $\Qs$, while the functions $\{H_{\beta}^*\}$ form a basis in
$\EuScript{L}^*(\mathcal{D}^*)$. Let $H_{\beta}^*=\EuScript{L}^*D^*_{\beta}$.
Then the operators $\{D_{\beta}=\EuScript{L}H_{\beta}\}$ form a graded basis
in $\mathcal{D}$, dual to the basis $\{D_{\beta}^*\}$.

Now for any $P\in \R$ we have
$$
\Delta_{\EuScript{L}}P=\sum\limits_{\beta}\left(H_{\beta}P\right)\otimes H_{\beta}^*
+\sum\limits_{\gamma}\left(H_{\gamma}P\right)\otimes H_{\gamma}^*
$$
But $H_{\gamma}P=0$ for all $\gamma$, therefore,
$$
\Delta_{\EuScript{L}}P=\sum\limits_{\beta}H_{\beta}P\otimes H_{\beta}^*=\sum\limits_{\beta}\left(\EuScript{L}H_{\beta}\right)P\otimes H_{\beta}^*=\sum\limits_{\beta}D_{\beta}P\otimes\EuScript{L}^*D^*_{\beta}=(1\otimes \EuScript{L}^*)\circ\Delta_{\mathcal{D}^*}P
$$
\end{proof}

We have the series of mappings
\begin{gather*}
\Phi(t_1)\colon\R\to\R[t_1],
\Phi_2(t_1,t_2)=\Phi(t_2)\Phi(t_1)\colon\R\to\Qsym[t_1,t_2]\otimes\R,\\
\dots\\
\Phi_n(t_1,\dots,t_n)=\Phi(t_n)\Phi(t_{n-1})\dots\Phi(t_1)\colon
\R\to\Qsym[t_1,\dots,t_n]\otimes\R
\end{gather*}

Taking the inverse limit we obtain the mapping
$$
\Phi_{\infty}:\R\to\Qs\otimes\R,
$$
such that for any element $P\in\R$ of graduation $2n$
$$
\Phi_{\infty}P=\sum\limits_{k=0}^{n}\sum\limits_{(j_1,\,\dots,\,j_k)}M_{(j_k,\,\dots,\,j_1)}\otimes\left(d_{j_1}\dots d_{j_k}P\right)
$$
Thus we have proved the following fact.
\begin{prop}
$$
\Phi_{\infty}=\Delta_{\EuScript{R}}
$$
\end{prop}

\subsection{Ring homomorphisms}

\begin{defin}
Any ring homomorphism $h\colon\R\to\mathcal{A}$ induces the ring homomorphism
$$
\Phi_{h}\colon\R\to\Qs\otimes \mathcal{A},\; \Phi_{h}=(1\otimes h)\circ\Delta_{\EuScript{R}}
$$
Let us denote by $\Phi_{h,\,r}$ the homomorphism
$\R\to\Qsym[t_1,\dots,t_r]\otimes\mathcal{A}$.
$$
\Phi_{h,\,r}(P)=\Phi_{h}(P)(t_1,\dots,t_r,0,0,\dots)
$$
\end{defin}

\begin{exam}
Let us consider several examples.
\begin{enumerate}
\item $\xi_{\alpha}\colon\mathcal{P}\to\mathbb Z[\alpha]$ induces the
    homomorphism $\Phi_{\xi_{\alpha}}\colon\mathcal{P}\to\Qs[\alpha]$.
    This homomorphism coincides with the generalized $f$-polynomial.
\item $\varepsilon_{\alpha}\colon\mathcal{RP}\to \mathbb Z[\alpha]$
    induces the ring homomorphism $\Phi_{\varepsilon_{\alpha}}\colon
    \mathcal{RP}\to\Qs[\alpha]$. Let us denote this ring homomorphism by
    $f_{\mathcal{RP}}$.
\item $\varepsilon_0\colon\mathcal{RP}\to\mathbb Z$ induces the
    homomorphism $\Phi_{\varepsilon_0}\colon\mathcal{RP}\to\Qs$. Then
    $\Phi_{\varepsilon_0}=\F^*$, where $\F$ is the Ehrenborg
    $\F$-quasi-symmetric function introduced in \cite{Ehr}:
    $$
      \F(P)=\!\!\!\!\sum\limits_{\hat 0=x_0<x_1<\dots<x_{k+1}=\hat 1}M_{(\rho(x_0,\,x_1),\,\rho(x_1,\,x_2),\,\dots,\,\rho(x_k,\,x_{k+1}))}=\!\!\!\!\!\!\!\!\!\sum\limits_{0\leqslant a_1<\dots<a_k\leqslant n-1}f_{a_1,\,\dots,\,a_k}(P)M_{(a_1+1,\,a_2-a_1,\,\dots,\,n-a_k)}
    $$
    where $P$ is an $n$-dimensional polytope, the sum ranges over all
    chains from $\hat 0=\varnothing$ to $\hat 1=P$ in the face lattice
    $L(P)$, and $f_{a_1,\,\dots,\,a_k}$ are flag numbers.

    In fact, $F$ is defined on the Rota-Hopf algebra $\mathcal{R}$ of
    graded posets, and it was proved in \cite{Ehr} that it is a Hopf
    algebra homomorphism $\mathcal{R}\to\Qs$ such that
    $\F(\Ps^*)=\F(\Ps)^*$.
\end{enumerate}
\end{exam}

\begin{prop}
For any polytope $P\in\mathcal{P}$ we have
$$
f_{\mathcal{RP}}(P)=\F^*(P)+\alpha f( P)
$$
\end{prop}
\begin{proof}
Indeed, let $P$ be an $n$-dimensional polytope. Then
$$
f_{\mathcal{RP}}(P)=\sum\limits_{k=0}^n\sum\limits_{0\leqslant a_1<\dots<a_k\leqslant n-1}f_{a_1,\,\dots,\,a_k}\left(M_{(n-a_k,\,\dots,\,a_1+1)}+\alpha^{a_1+1}M_{(n-a_k,\,\dots,\,a_2-a_1)}\right)
$$
\end{proof}

\begin{cor}
For any two polytopes $P,Q\in \mathcal{P}$ we have
$$
f(P\divideontimes Q)=f(P)\cdot\F^*(Q)+\F^*(P)\cdot f(Q)+\alpha f(P)\cdot f(Q)
$$
\end{cor}
\begin{proof}
Since
$$
f_{\mathcal{RP}}(P\divideontimes Q)=f_{\mathcal{RP}}(P)\cdot f_{\mathcal{RP}}(Q),\quad\mbox{and }\F^*(P \divideontimes Q)=\F^*(P)\cdot\F^*(Q),
$$
we obtain
$$
\F^*(P)\cdot\F^*(Q)+\alpha f(P\divideontimes Q)=\left(\F^*(P)+\alpha f(P)\right)\cdot\left(\F^*(Q)+\alpha f(Q)\right)
$$
Removing the brackets, canceling the equal summands, and  dividing by
$\alpha$, we obtain the required formula.
\end{proof}

\begin{cor}
$$
f(CP)=\F^*(P)+(\alpha+\sigma_1)f(P)
$$
\end{cor}

\begin{thm}
\begin{itemize}
\item For $r\geqslant n$ the image
    $f_r(\mathcal{P}^{2n})\subset\Qsym[t_1,\dots,t_r][\alpha]$ consists
    of homogeneous polynomials $g$ of degree $2n$ satisfying the
    equations
\begin{enumerate}
\item
\begin{equation}\label{I}
\begin{aligned}
g(\alpha,t_1,-t_1,t_3,\dots,t_r)&=g(\alpha,0,0,t_3,\dots,t_r)\\
g(\alpha,t_1,t_2,-t_2,\dots,t_r)&=g(\alpha,t_1,0,0,\dots,t_r)\\
\dots \\
g(\alpha,t_1,\dots,t_{r-1},-t_{r-1})&=g(\alpha,t_1,\dots,0,0)
\end{aligned}
\end{equation}
\item
\begin{equation}\label{I2}
g(-\alpha,t_1,t_2,\dots,t_{r-1},\alpha)=g(\alpha,t_1,t_2,\dots,t_{r-1},0)
\end{equation}
\end{enumerate}
\item For $r\geqslant n$ the image
    $f_{\mathcal{RP},\,r}(\mathcal{RP}^{2n})\subset\Qsym[t_1,\dots,t_r][\alpha]$
    consists of homogeneous polynomials $g$ of degree $2n$ satisfying the
    following conditions:
\begin{enumerate}
\item The equations of type (\ref{I})
\item The equation
$$
g(-\alpha,t_1,t_2,\dots,t_{r-1},\alpha)=g(0,t_1,t_2,\dots,t_{r-1},0)
$$
\end{enumerate}
\item For $r\geqslant n$ the image
    $\F_r(\mathcal{RP}^{2n})\subset\Qsym[t_1,\dots,t_r]$ consists of
    homogeneous polynomials $g$ of degree $2n$ satisfying the equations
    of type (\ref{I}).
\end{itemize}
In all the cases the equations are equivalent to the Bayer-Billera
(generalized Dehn-Sommerville) relations \cite{BB}
\end{thm}
\begin{proof}
The first statement of the theorem is exactly Theorem \ref{f}.

For the second and the third statements we need a lemma corresponding to
Lemma \ref{equations}
\begin{lemma}\label{RPequations}
Let $P$ be an $n$-dimensional polytope. Then
\begin{enumerate}
\item The equation
$$
f_{\mathcal{RP},\,r}(\alpha,t_1,\dots,t_q,-t_q,\dots,t_r)=f_{\mathcal{RP},\,r}(\alpha,t_1,\dots,0,0,\dots,t_r)
$$
is equivalent to the generalized Dehn-Sommerville relations
$$
\sum\limits_{j=a_t+1}^{a_{t+1}-1}(-1)^{j-a_t-1}f_{a_1,\,\dots,\,a_t,\,j,\,a_{t+1},\,\dots,\,a_k}=\left(1+(-1)^{a_{t+1}-a_t}\right)f_{a_1,\,\dots,\,a_t,\,a_{t+1},\,\dots,\,a_k}
$$
for $-1\leqslant a_1<a_2<\dots<a_k<n$, and $1\leqslant
k\leqslant\min\{r-1,n\},\;k+1-q\leqslant t\leqslant r-q$.

\item The equation
$$
f_{\mathcal{RP},\,r}(-\alpha,t_1,\dots,t_{r-1},\alpha)=f_{\mathcal{RP},\,r}(0,t_1,\dots,t_{r-1},0)
$$
is equivalent to the generalized Dehn-Sommerville relations
$$
\sum\limits_{j=0}^{a_1-1}(-1)^jf_{j,\,a_1,\,\dots,\,a_k}=(1+(-1)^{a_1-1})f_{a_1,\,\dots,\,a_k}
$$
for $0\leqslant k\leqslant\min\{r-1,n-1\}$.

\item The equation
$$
\F^*_r(t_1,\dots,t_q,-t_q,\dots,t_r)=\F^*_r(t_1,\dots,0,0,\dots,t_r)
$$
is equivalent to the generalized Dehn-Sommerville relations
$$
\sum\limits_{j=a_t+1}^{a_{t+1}-1}(-1)^{j-a_t-1}f_{a_1,\,\dots,\,a_t,\,j,\,a_{t+1},\,\dots,\,a_k}=\left(1+(-1)^{a_{t+1}-a_t}\right)f_{a_1,\,\dots,\,a_t,\,a_{t+1},\,\dots,\,a_k}
$$
for $0\leqslant a_1<\dots<a_k$, and $0\leqslant
k\leqslant\min\{r-2,n-1\},\;k+2-q\leqslant t\leqslant r-q$.

\end{enumerate}
\end{lemma}
\begin{proof}
Items $1.$ and $3.$ are proved in the same manner as Lemma \ref{equations}.
Let us consider the relation of item $2.$:
$f_{\mathcal{RP},\,r}(-\alpha,t_1,\dots,t_{r-1},\alpha)=f_{\mathcal{RP},\,r}(0,t_1,\dots,t_{r-1},0)$:
\begin{multline*}
=(-\alpha)^{n+1}+\sum\limits_{k=1}^{\min\{r-1,\,n+1\}}\sum\limits_{-1\leqslant
a_1<\dots<a_k\leqslant
n-1}f_{a_1,\,\dots,\,a_k}(-\alpha)^{a_1+1}\left(\sum\limits_{1\leqslant
l_1<\dots<l_k\leqslant r-1}t_{l_1}^{n-a_k}\dots
t_{l_k}^{a_2-a_1}\right)+\\
+\sum\limits_{k=1}^{\min\{r,\,n+1\}}\sum\limits_{-1\leqslant
a_1<\dots<a_k\leqslant
n-1}f_{a_1,\,\dots,\,a_k}(-\alpha)^{a_1+1}\left(\sum\limits_{1\leqslant
l_1<\dots<l_k=r}t_{l_1}^{n-a_k}\dots
t_{l_{k-1}}^{a_3-a_2}\alpha^{a_2-a_1}\right)=\\
\sum\limits_{k=1}^{\min\{r-1,\,n+1\}}\sum\limits_{-1=a_1<\dots<a_k\leqslant
n-1}f_{a_1,\,\dots,\,a_k}\left(\sum\limits_{1\leqslant l_1<\dots<l_k\leqslant
r-1}t_{l_1}^{n-a_k}\dots t_{l_k}^{a_2-a_1}\right)
\end{multline*}
This is equivalent to the relations
$$
f_{a_1,\,\dots,\,a_k}+\sum\limits_{j=0}^{a_1-1}(-1)^{j+1}f_{j,\,a_1,\,\dots,\,a_k}+(-1)^{a_1+1}f_{a_1,\,\dots,\,a_k}=0
$$
for $0\leqslant a_1<\dots<a_k$, and $0\leqslant k\leqslant\min\{r-1,n-1\}$.
The case $k=0$ again corresponds to the Euler formula
$$
1-f_0+f_1+\dots+(-1)^{n}f_{n-1}+(-1)^{n+1}=0.
$$
\end{proof}
\begin{rem}
Let us mention that in the case of $f_{\mathcal{RP},\,r}$ and $\F^*_r$ the
equations of type (\ref{I}) contain all the generalized Dehn-Sommerville
relations, and these equations for $f_{\mathcal{RP},\,r}$ imply the equation
(\ref{I2}). This follows from the fact that $a_1$ can be equal to $-1$, which
is not possible in the case of the generalized $f$-polynomial.
\end{rem}
Now the end of the proof of the theorem is the same as for Theorem \ref{f}.
\end{proof}

\begin{prop}
The correspondence
$f_{\mathcal{RP}}\to\left.f_{\mathcal{RP}}\right|_{\alpha=0}=\F^*$ is an
isomorphism of the images $f_{\mathcal{RP}}(\mathcal{RP})$ and
$\F^*(\mathcal{RP})$.
\end{prop}

\begin{proof}
Indeed, it is clear that this mapping is an epimorphism. On the other hand,
let $\F^*(P)=0$. Then all the flag numbers of the element $P\in\mathcal{RP}$
are equal to $0$. Therefore, $f_{\mathcal{RP}}(P)=0$, so the mapping is a
monomorphism.
\end{proof}

\begin{cor}
\begin{enumerate}
\item The ring homomorphism
    $\F^*:\mathcal{RP}\to\EuScript{R}^*(\mathcal{D}^*)\subset\Qs$ is an
    epimorphism,
\item $\F^*(\mathcal{RP})\otimes\mathbb Q$ is a free polynomial algebra
    with dimension of the $(2n)$-th graded component equal to the
    $(n-1)$-th Fibonacci number $c_{n-1}$ ($c_0=c_1=1$,
    $c_{n+1}=c_{n}+c_{n-1}$, $n\geqslant 1$).
\end{enumerate}
\end{cor}
\begin{proof}
\begin{enumerate}
\item Let $g\in\EuScript{R}^*(\mathcal{D}^*)$ be a quasi-symmetric
    function of degree $2n$. Then according to Theorem \ref{D*}
    $g_n(t_1,\dots,t_n)=g(t_1,\dots,t_n,0,0,\dots)$ satisfies the
    relations of type (\ref{I}). Therefore $g_n(t_1,\dots,t_n)=\F^*_n(P)$
    for some $P\in\mathcal{RP}^{2n}$. Since the mapping
    $\Qs\to\Qsym[t_1,\dots,t_n], t_j \to 0, j>n$ is injective on the
    graduation $(2n)$, we see that $g=\F^*(P)$.
\item We see that $\F^*(\mathcal{RP})=\EuScript{R}^*(\mathcal{D}^*)$.
    Therefore,
    $$
    \F^*(\mathcal{RP})\otimes\mathbb Q\simeq\mathcal{D}^*\otimes\mathbb Q\simeq\mathbb Q[\Lyn_{odd}],
    $$
    and $\dim \F^*(\mathcal{RP}^{2n})\otimes \mathbb Q=c_{n-1}$ according
    to Proposition \ref{DFM}.
\end{enumerate}
\end{proof}

\begin{rem}
Let us mention that $\deg L(P^n)=2(n+1)$, therefore, $\rank
\F^*(\mathcal{RP}^{2(n+1)})=c_n=\rank f(\mathcal{P}^{2n})$
\end{rem}

\subsection{Natural Hopf comodule structure}

$\mathcal{P}$ is not a Hopf algebra, but it turns out that there is a natural
Hopf comodule structure on the ring $\mathcal{P}$ over the Hopf algebra
$\mathcal{RP}$.

\begin{prop}
Formula
$$
\Delta_{\mathcal{RP}} P=\sum\limits_{F\subseteq P}F\otimes(P/F)
$$
defines a natural graded right Hopf comodule structure
$\Delta_{\mathcal{RP}}\colon\mathcal{P}\to\mathcal{P}\otimes\mathcal{RP}$ on
the ring $\mathcal{P}$ over the Hopf algebra $\mathcal{RP}$.
\end{prop}
\begin{proof}
We should proof that the following two diagrams commute
$$
\xymatrix{
\mathcal{P}\ar[d]^{\Delta_{\mathcal{RP}}}\ar[rr]^{\Delta_{\mathcal{RP}}}&&\mathcal{P}\otimes\mathcal{RP}\ar[d]^{\id\otimes\Delta}\\
\mathcal{P}\otimes\mathcal{RP}\ar[rr]^{\Delta_{\mathcal{RP}}\otimes\id}&&\mathcal{P}\otimes\mathcal{RP}\otimes\mathcal{RP}\\
}
\hspace{2cm}
\xymatrix{
\mathcal{P}\ar[d]_{\Delta_{\mathcal{RP}}}\ar[dr]^{\id\otimes 1}&\\
\mathcal{P}\otimes \mathcal{RP}\ar[r]^{\id\otimes \varepsilon}&\mathcal{P}\otimes\mathbb Z\\
}
$$
and that $\Delta_{\mathcal{RP}}$ is a ring homomorphism. Indeed, we have
\begin{multline*}
(\Delta_{\mathcal{RP}}\otimes\id)\circ\Delta_{\mathcal{RP}} P=(\Delta_{\mathcal{RP}}\otimes\id)\left(\sum\limits_{F\subseteq P}F\otimes P/F\right)=\sum\limits_{F\subseteq P}\left(\sum\limits_{G\subseteq F}G\otimes F/G\right)\otimes P/F=\\
=\sum\limits_{G\subseteq P}G\otimes\left(\sum\limits_{G\subseteq F\subseteq
P}F/G\otimes P/F\right)=\sum\limits_{G\subseteq P}G\otimes\Delta (P/G)
=(\id\otimes\Delta)\circ\Delta_{\mathcal{RP}} P,
\end{multline*}
and
$$ (\id\otimes\varepsilon)\circ\Delta_{\mathcal{RP}}(P)=\sum\limits_{F\subseteq
P}F\otimes\varepsilon(P/F)=P\otimes 1,
$$

Let us check that $\Delta_{\mathcal{RP}}$ is a homomorphism of rings
$\mathcal{P}\to\mathcal{P}\otimes\mathcal{RP}$:
\begin{multline*}
\Delta_{\mathcal{RP}}(P\times Q)=\sum\limits_{F\times G\subseteq P\times
Q}(F\times G)\otimes \left(P\times Q/ F\times G\right)=
\sum\limits_{F\subseteq P,\,G\subseteq Q}(F\times G)\otimes \left(P/F\divideontimes Q/G\right)=\\
=\left(\sum\limits_{F\subseteq P}F\otimes
P/F\right)\cdot\left(\sum\limits_{G\subseteq Q}G\otimes
Q/G\right)=\Delta_{\mathcal{RP}}(P)\cdot\Delta_{\mathcal{RP}}(Q)
\end{multline*}
$$
\Delta_{\mathcal{RP}}(\pt)=\pt\otimes (\pt/\pt)=\pt\otimes \varnothing=1\otimes 1.
$$
\end{proof}

\begin{cor}
\begin{itemize}
\item Any ring homomorphism $\mathcal{P}\to\mathcal{A}$, where
    $\mathcal{A}$ is a ring, induces the ring homomorphism
    $\mathcal{P}\to\mathcal{A}\otimes\mathcal{RP}$
\item Any linear homomorphism $\psi\colon\mathcal{RP}\to\mathbb Z$
    induces the linear operator $\mathcal{P}\to\mathcal{P}$, which is a
    ring homomorphism, if $\psi$ is a ring homomorphism.
\end{itemize}
\end{cor}

\begin{exam}
The homomorphism $\xi_{\alpha}\colon\mathcal{P}\to\mathbb Z[\alpha]$ defines
the ring homomorphism $l_{\alpha}\colon\mathcal{P}\to\mathcal{RP}[\alpha]$:
$$
l_{\alpha}(P^n)=\sum\limits_{F\subseteq P}\alpha^{\dim F}P/F.
$$
Let us recall that in the case of a simple $n$-polytope a face polytope of an
$i$-dimensional face is a simplex $\Delta^{n-i-1}=x^{n-i}$. Therefore,
$$
l_{\alpha}(P^n)=\sum\limits_{i=0}^nf_i\alpha^ix^{n-i}=f_1(\alpha,x).
$$
This is a homogeneous $f$-polynomial in two variables defined in \cite{Buch}.
\end{exam}

\begin{prop}\label{DPQ}
\begin{equation}
\Delta_{\mathcal{RP}}(P\divideontimes Q)=(1\otimes P)\cdot\Delta_{\mathcal{RP}}(Q)+\Delta_{\mathcal{RP}}(P)\cdot (1\otimes Q)+\Delta_{\mathcal{RP}}(P)(\divideontimes\otimes\divideontimes) \Delta_{\mathcal{RP}}(Q)\label{FDPQ}
\end{equation}
\end{prop}
\begin{proof}
\begin{multline*}
\Delta_{\mathcal{RP}}(P\divideontimes Q)=\sum\limits_{G\subseteq Q}
\left(\varnothing\divideontimes G\right)\otimes \left(P\divideontimes
Q/\varnothing \divideontimes G\right)+\sum\limits_{F\subseteq P}
\left(F\divideontimes\varnothing\right)\otimes \left(P\divideontimes
Q/F \divideontimes \varnothing\right)+\\
+\sum\limits_{F\subseteq P,\,G\subseteq Q}(F\divideontimes  G)\otimes
\left(P\divideontimes Q/F\divideontimes G\right)=\sum\limits_{G\subseteq Q}
G\otimes \left(P\divideontimes (Q/G)\right)+
\sum\limits_{F\subseteq P} F\otimes \left((P/F)\divideontimes Q\right)+\\
+\sum\limits_{F\subseteq P,\,G\subseteq Q}\left(F\divideontimes
G\right)\otimes\left(P/F\divideontimes Q/G\right)=(1\otimes
P)\Delta_{\mathcal{RP}}(Q)+\Delta_{\mathcal{RP}}(P)(1\otimes
Q)+\Delta_{\mathcal{RP}}(P)(\divideontimes\otimes\divideontimes)
\Delta_{\mathcal{RP}}(Q)
\end{multline*}
\end{proof}

Now if we apply $\xi_{\alpha}\otimes 1$ to the formula (\ref{FDPQ}), then we
obtain the corollary:
\begin{cor}\label{lPQ}
$$
l_{\alpha}(P\divideontimes Q)=P\divideontimes l_{\alpha}(Q)+l_{\alpha}(P)\divideontimes Q+\alpha\; l_{\alpha}(P)\cdot l_{\alpha}(Q)
$$
\end{cor}
\subsection{Interrelation}

Let us remind that the left action of the Leibnitz-Hopf algebra $\mathcal{Z}$
in the ring $\R$ induces the right Hopf comodule structure
$$
\Delta_{\EuScript{L}}P^n=\sum\limits_{k=0}^n\sum\limits_{(j_1,\,\dots,\,j_k)}\left(d_{j_1}\dots
d_{j_k}P\right)\otimes M_{(j_1,\,\dots,\,j_k)},
$$

\begin{prop}
The following diagram commutes:
$$
\xymatrix{
\mathcal{P}\ar[dr]_-{\Delta_{\EuScript{L}}}\ar[r]^-{\Delta_{\mathcal{RP}}}&\mathcal{P}\otimes\mathcal{RP}\ar[d]^{1\otimes \F}\\
&\mathcal{P}\otimes\Qs
}
$$
\end{prop}
\begin{proof}
Let $P$ be an $n$-dimensional polytope. Then
\begin{multline*}
\Delta_{\EuScript{L}}(P)=\sum\limits_{k=0}^n\sum\limits_{(j_1,\,\dots,\,j_k)}(d_{j_1}\dots
d_{j_k}P)\otimes
M_{(j_1,\,\dots,\,j_k)}=\\
\sum\limits_{k=0}^n\sum\limits_{(j_1,\,\dots,\,j_k)}\sum\limits_{F^{n-(j_1+\dots+j_k)}}\left(F^{n-(j_1+\dots+j_k)}\sum\limits_{F^{n-(j_1+\dots+j_k)}\subset
F^{n-(j_2+\dots+j_k)}\subset\dots\subset F^{n-j_k}\subset
P}1\right)\otimes M_{(j_1,\,\dots,\,j_k)}=\\
=\sum\limits_{r=0}^n\sum\limits_{F^r\subseteq P}F^r\otimes
\left(\sum\limits_{k=0}^{n-r}\sum\limits_{(j_1,\,\dots,\,
j_k):j_1+\dots+j_k=n-r}\sum\limits_{F^r\subset F^{r+j_1}\subset\dots\subset
F^{n-j_k}\subset P}M_{(j_1,\,\dots,\,j_k)}\right)=\\
=\sum\limits_{F\subseteq P}F\otimes \F(P/F)=(1\otimes\F)\circ\Delta_{RP}(P)
\end{multline*}
\end{proof}

\begin{cor}\label{Fl} For any polytope $P\in\mathcal{P}$ we have
$$
\F^*(l_{\alpha}P)=f(P)
$$
\end{cor}
\begin{proof}
We have $(\xi_{\alpha}\otimes1)\circ\Delta_{\EuScript{L}}P=f(P)^*$. So if we
apply the homomorphism $\xi_{\alpha}\otimes 1$ to the diagram above, we
obtain $\F\circ\, l_{\alpha}=(\xi_{\alpha}\otimes
1)\circ\Delta_{\EuScript{L}}=f^*$. Therefore, $\F^*\circ\, l_{\alpha}=f$.
\end{proof}
\begin{cor}
The following diagram commutes:
$$
\begin{CD}
\mathcal{P}@>{f^*}>>\Qs[\alpha]\\
@V{\Delta_{\mathcal{RP}}}VV @V{\Delta}VV\\
\mathcal{P}\otimes\mathcal{RP}@>{f^*\otimes F}>>\Qs[\alpha]\otimes\Qs
\end{CD}
$$
\end{cor}

\begin{proof}
\begin{enumerate}
\item Consider the commutative diagram
$$
\xymatrix{
\mathcal{P}\ar[dr]_-{\Delta_{\EuScript{L}}}\ar[r]^-{\Delta_{\mathcal{RP}}}&\mathcal{P}\otimes\mathcal{RP}\ar[d]^{1\otimes \F}\ar[r]^-{\Delta_{\EuScript{L}}\otimes 1}& P\otimes\Qs\otimes \mathcal{RP}\ar[d]^{1\otimes 1\otimes F}\\
&\mathcal{P}\otimes\Qs\ar[dr]\ar[r]^-{\Delta_{\EuScript{L}}\otimes 1}& P\otimes \Qs\otimes\Qs\ar[d]^{\xi_{\alpha}\otimes 1\otimes 1}\\
&&\mathbb Z[\alpha]\otimes\Qs\otimes\Qs
}
$$
Then
$$
\left(\xi_{\alpha}\otimes 1\otimes 1\right) \circ \left(\Delta_{\EuScript{L}}\otimes 1\right)\circ\Delta_{\EuScript{L}}=
\left(\xi_{\alpha}\otimes 1\otimes 1\right)\circ\left(1\otimes\Delta\right)\circ\Delta_{\EuScript{L}}=\Delta\circ f^*
$$
while
$$
\left(\xi_{\alpha}\otimes 1\otimes 1\right)\circ\left(1\otimes 1\otimes \F\right)\circ\left(\Delta_{\EuScript{L}}\otimes 1\right)\circ \Delta_{\mathcal{RP}}= \left(f^*\otimes\F\right)\circ\Delta_{\mathcal{RP}}
$$
Since these two compositions correspond to two pathes in the diagram,
they are equal.
\item Let us consider also a direct proof. We know that
$$
\Delta(M_{(b_1,\,\dots,\,b_k)})=1\otimes M_{(b_1,\,\dots,\,b_k)}+M_{(b_1)}\otimes M_{(b_2,\,\dots,\,b_k)}+\dots+M_{(b_1,\,\dots,\,b_{k-1})}\otimes M_{(b_k)}+M_{(b_1,\,\dots,\,b_k)}\otimes 1
$$
and
$$
f(\alpha,t_1,t_2,\dots)(P^n)^*=\alpha^n+\sum\limits_{F^{a_1}\subset\dots\subset F^{a_k}}\alpha^{a_1}M_{(a_2-a_1,\,\dots,\,n-a_k)}=\sum\limits_{F^{a_1}\subset\dots\subset F^{a_k}\subset P^n}\alpha^{a_1}M_{(a_2-a_1,\,\dots,\,n-a_k)}
$$
Then
\begin{multline*}
\Delta(f(P^n)^*)=\sum\limits_{F^{a_1}\subset\dots\subset G^{a_s}\subset\dots\subset F^{a_k}\subset P^n}\alpha^{a_1}M_{(a_2-a_1,\,\dots,\,a_s-a_{s-1})}\otimes M_{(a_{s+1}-a_s,\,\dots,\,n-a_k)}=\\
=\sum\limits_{G^{a_s}\subseteq P^n}\left(\sum\limits_{F^{a_1}\subset\dots\subset F^{a_{s-1}}\subset G^{a_s}}\alpha^{a_1}M_{(a_2-a_1,\,\dots,\,a_s-a_{s-1})}\right)\otimes\left(\sum\limits_{G^{a_s}\subset F^{a_{s+1}}\subset\dots\subset F^{a_k}\subset P^n}M_{(a_{s+1}-a_s,\,\dots,\,n-a_k)}\right)=\\
=\sum\limits_{G\subseteq P^n}f^*(G)\otimes \F(P^n/G)=(f^*\otimes
\F)\circ\Delta_{\mathcal{RP}} (P^n).
\end{multline*}
\end{enumerate}
\end{proof}

\section{Operators $B$ and $C$}
We know that the {\itshape cone} and the {\itshape bipyramid} operations can
be considered as linear operators on the rings $\mathcal{P}$ and
$\mathcal{RP}$. It turns out that there are corresponding operations on the
rings $\Qs[\alpha]$ and $\Qs$, which commute with the homomorphisms $f,
f_{\mathcal{RP}}$, and $\F^*$.

\subsection{The case of $\mathcal{P}$} At first let us consider the ring $\mathcal{P}$. We have
$$
[C,d]=1+\xi_0
$$
For $k\geqslant 2$ we have the similar situation.
$$
d_kCP^n=\begin{cases}
Cd_kP^n+d_{k-1}P^n,&k<n+1,\\
1+d_{k-1}P^n=C\varnothing+d_{k-1}P^n,&k=n+1
\end{cases}
$$
So we have $[d_{k+1},C]=d_k+\frac{1}{k!}\left.\frac{\partial^k}{\partial
\alpha^k}\right|_{\alpha=0}\xi_{\alpha}$, $k\geqslant 0$. This can be
reformulated in terms of $\Phi(t)$.
\begin{prop}\label{PC}
We have
$$
\Phi(t)C=(C+t)\Phi(t)+t\xi_t
$$
\end{prop}
\begin{proof}
Let $P$ be an $n$-polytope. Then
$$
\Phi(t)CP=CP+(CdP+P)t+\dots+(Cd_nP+d_{n-1}P)t^n+(1+d_nP)t^{n+1}=C\Phi(t)P+t\Phi(t)P+t\xi_{t}P
$$
\end{proof}
Now let us consider the bipyramid. $dBP=2CdP$ for any polytope of positive
dimension, while $dB\pt=dI=2\pt$, and $2Cd\pt=0$. So $dBP=2CdP+2\xi_0$. For
$k\geqslant 2$ we have
$$
d_kBP^n=\begin{cases}2Cd_kP^n+d_{k-1}P^n,&k<n+1;\\
2+d_{k-1}P^n=2C\varnothing+d_{k-1}P^n,&k=n+1;
\end{cases}
$$
So $d_{k+1}B=2Cd_{k+1}+d_k+\frac{2}{k!}\left.\frac{\partial^k}{\partial
\alpha^k}\right|_{\alpha=0}\xi_{\alpha}$, $k\geqslant 1$.
\begin{prop}\label{PB}
$$
\Phi(t)B=(B-2C-t)+(2C+t)\Phi(t)+2t\xi_{t}
$$
\end{prop}
\begin{proof}
For a point we have: $\Phi(t)B\pt=\Phi(t)I=I+(2\pt)t$, and
$$
\left((B-2C-t)+(2C+t)\Phi(t)+2t\xi_t\right)\pt=I-2I-(\pt)t+2I+(\pt) t+2(\pt)t=I+(2\pt) t
$$
Let $P$ be an $n$-polytope, $n>0$. Then
\begin{multline*}
\Phi(t)BP=BP+(2CdP)t+(2Cd_2P+dP)t^2+\dots+(2Cd_nP+d_{n-1}P)t^n+(2+d_n
P)t^{n+1}=\\
=(B-2C-t)P+2C\Phi(t)P+t\Phi(t)P+2t\xi_tP
\end{multline*}
\end{proof}

Let us denote $A=2C-B$. Then from Propositions \ref{PC} and \ref{PB} we have
$$
\Phi(t)A=A+t+t\Phi(t)
$$
At last, Propositions \ref{PC} and \ref{PB} imply the following formula.
\begin{prop}
$$
\Phi(t)[B,C]=[B,C]+At+t^2
$$
\end{prop}
\begin{proof}
\begin{multline*}
\Phi(t)[B,C]=\Phi(t)(BC-CB)=\left(B-2C-t+(2C+t)\Phi(t)+2t\xi_t\right)C-\left((C+t)\Phi(t)+t\xi_t\right)B=\\
=(B-2C-t)C+(2C+t)\left((C+t)\Phi(t)+t\xi_t\right)+2t^2\xi_t-(C+t)\left(B-2C-t+(2C+t)\Phi(t)+2t\xi_t\right)-t^2\xi_t=\\
=(BC-CB)-t(B-2C-t)+2Ct\xi_t+t^2\xi_t+2t^2\xi_t-2Ct\xi_t-2t^2\xi_t-t^2\xi_t=[B,C]+At+t^2
\end{multline*}
\end{proof}

Now let us apply this formulas to the generalized $f$-polynomial.
\begin{prop}\label{AC}
Mappings $C,A\colon\Qs[\alpha]\to\Qs[\alpha]$ defined by the formulas
\begin{gather*}
Cg=(\alpha+\sigma_1)g+\sum\limits_{n=1}^{\infty}t_n g(t_n,t_1,\dots,t_{n-1},0,0,\dots)\\
Ag=\alpha
g(\alpha,0,0,\dots)+t_1g(\alpha,t_1,t_2,\dots)+\sum\limits_{n=2}^{\infty}(t_n+t_{n-1})g(\alpha,t_n,t_{n+1},\dots)
\end{gather*}
satisfy the relations
$$
f(CP)=Cf(P),\quad f(AP)=Af(P)
$$
for all $P\in\mathcal{P}$.
\end{prop}
\begin{proof}
At first let us prove that this mappings are correctly defined, that is for
any $g\in\Qs[\alpha]$ its images $Cg$, and $Ag$ belong to $\Qs[\alpha]$. We
will use the criterion obtained in Proposition \ref{Crn} and Remark \ref{cr}.
It is easy to see that if $g$ has bounded degree, then $Cg$ and $Ag$ have
bounded degree.

Let $i\geqslant 1$. Then
\begin{multline*}
(Cg)(\alpha,t_1,\dots,t_{i-1},0,t_{i+1},\dots)=(\alpha+\sum\limits_{j\ne
i}t_j)g(\alpha,t_1,\dots,t_{i-1},0,t_{i+1},\dots)+\sum\limits_{n=1}^{i-1}t_n g(t_n,t_1,\dots,t_{n-1},0,0,\dots)+\\
+\sum\limits_{n=i+1}^{\infty} t_n g(t_n,t_1,\dots,t_{i-1},0,t_{i+1},\dots,t_{n-1},0,0,\dots)=\\
=(\alpha+\sigma_1(t_1,\dots,t_{i-1},t_{i+1},\dots))g(\alpha,t_1,\dots,t_{i-1},t_{i+1},\dots)+\sum\limits_{n=1}^{i-1}t_n g(t_n,t_1,\dots,t_{n-1},0,0,\dots)+\\
+t_{i+1}g(t_{i+1},t_1,\dots,t_{i-1},0,0,\dots)+\sum\limits_{n=i+2}^{\infty}
t_n
g(t_n,t_1,\dots,t_{i-1},t_{i+1},\dots,t_{n-1},0,0,\dots)=(Cg)(\alpha,t_1,\dots,t_{i-1},t_{i+1},\dots)
\end{multline*}
Now consider $A$. For $i=1$ the relation is clear. Let $i>1$. Then
\begin{multline*}
(Ag)(\alpha,t_1,\dots,t_{i-1},0,t_{i+1},\dots)=\alpha g(\alpha,0,0,\dots)+\sum\limits_{n=1}^{i-1}(t_n+t_{n-1})g(\alpha,t_n,\dots,t_{i-1},0,t_{i+1},\dots)+\\
+(0+t_{i-1})g(\alpha,0,t_{i+1},\dots)+(t_{i+1}+0)g(\alpha,t_{i+1},t_{i+2},\dots)+\sum\limits_{n=i+2}^{\infty}(t_n+t_{n-1})g(\alpha,t_n,t_{n+1},\dots)=\\
=\alpha g(\alpha,0,0,\dots)+\sum\limits_{n=1}^{i-1}(t_n+t_{n-1})g(\alpha,t_n,\dots,t_{i-1},t_{i+1},\dots)+(t_{i+1}+t_{i-1})g(\alpha,t_{i+1},t_{i+2},\dots)+\\
+\sum\limits_{n=i+2}^{\infty}(t_n+t_{n-1})g(\alpha,t_n,t_{n+1},\dots)=
(Ag)(\alpha,t_1,\dots,t_{i-1},t_{i+1},\dots)
\end{multline*}
On the other hand, we have
\begin{gather*}
\Phi(t_1)C=(C+t_1)\Phi(t_1)+t_1\xi_{t_1}\\
\Phi(t_2)\Phi(t_1)C=(C+t_2+t_1)\Phi(t_2)\Phi(t_1)+t_2\xi_{t_2}\Phi(t_1)+t_1\xi_{t_1}\\
\dots\\
\Phi(t_n)\Phi(t_{n-1})\dots\Phi(t_1)C=(C+t_n+\dots+t_1)\Phi(t_n)\Phi(t_{n-1})\dots\Phi(t_1)+t_n\xi_{t_n}\Phi(t_{n-1})\dots\Phi(t_1)+\dots+t_1\xi_{t_1}\\
\dots\\
\Phi_{\infty}C=(C+\sigma_1)\Phi_{\infty}+\sum\limits_{n=1}^{\infty}t_nf(t_n,t_1,\dots,t_{n-1},0,\dots)\\
f(CP)=\xi_{\alpha}\Phi_{\infty}CP=(\alpha+\sigma_1)\xi_{\alpha}\Phi_{\infty}P+\sum\limits_{n=1}^{\infty}t_nf(t_n,t_1,\dots,t_{n-1},0,\dots)=Cf(P)
\end{gather*}

\begin{gather*}
\Phi(t_1)A=A+t_1+t_1\Phi(t_1)\\
\Phi(t_2)\Phi(t_1)A=A+t_2+(t_2+t_1)\Phi(t_2)+t_1\Phi(t_2)\Phi(t_1)\\
\dots\\
\Phi(t_n)\Phi(t_{n-1})\dots\Phi(t_1)A=A+t_n+(t_n+t_{n-1})\Phi(t_n)+(t_{n-1}+t_{n-2})\Phi(t_n)\Phi(t_{n-1})+\dots+t_1\Phi(t_n)\dots\Phi(t_1)\\
\dots\\
\Phi_{\infty}A=A+t_1\Phi_{\infty}(t_1,t_2,\dots)+\sum\limits_{n=2}^{\infty}(t_n+t_{n-1})\Phi_{\infty}(t_n,t_{n+1},\dots)\\
f(AP)=\xi_{\alpha}\Phi_{\infty}AP=\alpha
f(\alpha,0,0,\dots)+t_1f(\alpha,t_1,t_2,\dots)+\sum\limits_{n=2}^{\infty}(t_n+t_{n-1})f(\alpha,t_n,t_{n+1},\dots)=Af(P)
\end{gather*}
\end{proof}

Now the mapping $B:\Qs[\alpha]\to\Qs[\alpha]$ can be defined as $B=2C-A$.
Then $f(BP)=Bf(P)$.

\subsection{The case of $\mathcal{RP}$}
Now let us consider the ring $\mathcal{RP}$. In this ring $CP$ is just the
multiplication by the element $x=\pt$. So
\begin{gather*}
\Phi(t)CP=\Phi(t)(xP)=(\Phi(t)x)\divideontimes(\Phi(t)P)=(x+t)\Phi(t)P=(C+t)\Phi(t)P\\
\Phi_{\infty}(CP)=\Phi_{\infty}(xP)=\Phi_{\infty}(x)\divideontimes\Phi_{\infty}(P)=(x+\sigma_1)\Phi_{\infty}(P)=(C+\sigma_1)\Phi_{\infty}(P)\\
f_{\mathcal{RP}}(CP)=f_{\mathcal{RP}}(xP)=f_{\mathcal{RP}}(x)f_{\mathcal{RP}}(P)=(\alpha+\sigma_1)f_{\mathcal{RP}}(P)\\
\F^*(CP)=\F^*(xP)=\F^*(x)\F^*(P)=\sigma_1\F^*(P)\\
\end{gather*}
In fact, the existence of the empty polytope $\varnothing$ such that
$d_{n+1}P^n=\varnothing$ for any $n$-polytope $P^n$ makes the formulas
connecting $C,B$, and $d_k$ simpler. We have:
\begin{align*}
[C,d]=1,& \quad [C,d_{k+1}]=d_k,\, k\geqslant 1\\
dB=2Cd+\varepsilon_0,&\quad d_{k+1}B=2Cd_{k+1}+d_k,\, k\geqslant 1
\end{align*}
\begin{prop}
$$
\Phi(t)B=(B-2C-t)+(2C+t)\Phi(t)+\varepsilon_0t
$$
\end{prop}
\begin{proof}
Indeed,
\begin{multline*}
\Phi(t)BP=BP+(2CdP+\varepsilon_0(P))t+(2Cd_2P+dP)t^2+\sum\limits_{k\geqslant
3}(2Cd_kP+d_{k-1}P)t^k=\\
=(B-2C-t)P+2C\Phi(t)P+t\Phi(t)P+\varepsilon_0(P)t
\end{multline*}
\end{proof}
Then
$$
\Phi(t)A=A+t+t\Phi(t)-\varepsilon_0t
$$

\begin{prop}
The mappings  $A_{\mathcal{RP}}\colon\Qs[\alpha]\to\Qs[\alpha]$,
$A_{0}\colon\Qs\to\Qs$ defined by the formulas
\begin{gather*}
A_{\mathcal{RP}}g=\alpha
g(\alpha,0,0,\dots)+t_1g(\alpha,t_1,t_2,\dots)+\sum\limits_{n=2}^{\infty}(t_n+t_{n-1})g(\alpha,t_n,t_{n+1},\dots)-\sigma_1g(0,0,\dots)\\
A_{0}g=t_1g(t_1,t_2,\dots)+\sum\limits_{n=2}^{\infty}(t_n+t_{n-1})g(t_n,t_{n+1},\dots)-\sigma_1g(0,0,\dots)
\end{gather*}
satisfy the relations
$$
f_{\mathcal{RP}}(AP)=A_{\mathcal{RP}}f_{\mathcal{RP}}(P),\quad \F^*(AP)=A_{0}\F^*(P)
$$
for all $P\in\mathcal{RP}$.
\end{prop}
\begin{proof}
From the proof of Proposition \ref{AC} we see that both mappings
$A_{\mathcal{RP}}$ and $A_{0}$ are correctly defined. Then
\begin{gather*}
\Phi(t_1)A=A+t_1+t_1\Phi(t_1)-\varepsilon_0 t_1\\
\Phi(t_2)\Phi(t_1)A=A+t_2+(t_2+t_1)\Phi(t_2)+t_1\Phi(t_2)\Phi(t_1)-\varepsilon_0(t_2+t_1)\\
\dots\\
\Phi(t_n)\Phi(t_{n-1})\dots\Phi(t_1)A=A+t_n+(t_n+t_{n-1})\Phi(t_n)+(t_{n-1}+t_{n-2})\Phi(t_n)\Phi(t_{n-1})+\dots\\
+t_1\Phi(t_n)\dots\Phi(t_1)-\varepsilon_0(t_n+\dots+t_1)\\
\dots\\
\Phi_{\infty}A=A+t_1\Phi_{\infty}(t_1,t_2,\dots)+\sum\limits_{n=2}^{\infty}(t_n+t_{n-1})\Phi_{\infty}(t_n,t_{n+1},\dots)-\varepsilon_0\sigma_1\\
f_{\mathcal{RP}}(AP)=\varepsilon_{\alpha}\Phi_{\infty}A=\varepsilon_{\alpha}(AP)+t_1\varepsilon_{\alpha}\Phi_{\infty}(t_1,t_2,\dots)+\sum\limits_{n=2}^{\infty}(t_n+t_{n-1})\varepsilon_{\alpha}\Phi_{\infty}(t_n,t_{n+1},\dots)-\varepsilon_0(P)\sigma_1=\\
=\alpha
f_{\mathcal{RP}}(\alpha,0,0,\dots)+t_1f_{\mathcal{RP}}(\alpha,t_1,t_2,\dots)+\sum\limits_{n=2}^{\infty}(t_n+t_{n-1})f_{\mathcal{RP}}(\alpha,t_n,t_{n+1},\dots)-\sigma_1f_{\mathcal{RP}}(0,0,\dots)=A_{\mathcal{RP}}f_{\mathcal{RP}}(P)\\
\F^*(AP)=\varepsilon_{0}\Phi_{\infty}A=\varepsilon_{0}(AP)+t_1\varepsilon_{0}\Phi_{\infty}(t_1,t_2,\dots)+\sum\limits_{n=2}^{\infty}(t_n+t_{n-1})\varepsilon_0\Phi_{\infty}(t_n,t_{n+1},\dots)-\varepsilon_0(P)\sigma_1=\\
=t_1\F^*(t_1,t_2,\dots)+\sum\limits_{n=2}^{\infty}(t_n+t_{n-1})\F^*(t_n,t_{n+1},\dots)-\sigma_1\F^*(0,0,\dots)=A_{0}\F^*(P)
\end{gather*}
\end{proof}
As before $B_{\mathcal{RP}}$ and $B_{0}$ are defined as
$B_{\mathcal{RP}}=2C_{\mathcal{RP}}-A_{\mathcal{RP}}$, $B_0=2C_0-A_0$, where
$C_{\mathcal{RP}}g=(\alpha+\sigma_1)g$, $C_{0}g=\sigma_1g$, and
$$
f_{\mathcal{RP}}(BP)=B_{\mathcal{RP}}f_{\mathcal{RP}}(P),\quad \F^*(BP)=B_0\F^*(B)
$$

At last
\begin{multline*}
\Phi(t)(BC-CB)=\left(B-2C-t+(2C+t)\Phi(t)+\varepsilon_0t\right)C-(C+t)\Phi(t)B=\\
=BC-(2C+t)C+(2C+t)(C+t)\Phi(t)-(C+t)\left(B-2C-t+(2C+t)\Phi(t)+\varepsilon_0t\right)=\\
=BC-CB+(2C-B)t+t^2-(C+t)\varepsilon_0t=BC-CB+At+t^2-(x+t)t\varepsilon_0
\end{multline*}
So we obtain the following proposition.
\begin{prop}
In the ring $\mathcal{RP}$
$$
\Phi(t)[B,C]=[B,C]+At+t^2-(x+t)t\varepsilon_0
$$
\end{prop}

\section{Problem of the Description of Flag Vectors of
Polytopes} We have mentioned that on the space of simple polytopes
$$
f(\alpha,t_1,t_2,\dots)(P^n)=f_1(\alpha,t_1+t_2+\dots)(P^n)=f_1(\alpha,\sigma_1)(P^n).
$$
The only linear relation on the polynomial $f_1$ is
$f_1(-\alpha,\alpha+t)=f_1(\alpha,t)$, and it is equivalent to the
Dehn-Sommerville relations. In fact, for the polynomial
$g=g(\alpha,\sigma_1)\in\Qsym[t_1,t_2,\dots][\alpha]$ this condition is
necessary and sufficient to be an image of an integer combination of simple
polytopes.

One of the outstanding results in the polytope theory is the so-called
$g$-theorem, which was formulated as a conjecture by P.~McMullen \cite{Mc1}
in 1970 and proved by R.~Stanley \cite{St1} (the necessity) and L.~Billera
and C.~Lee \cite{BL} (the sufficiency) in 1980.

For an $n$-dimensional simple polytope $P^n$ let us define an {\itshape
$h$-polynomial}:
$$
h(\alpha,t)(P^n)=h_0\alpha^n+h_1\alpha^{n-1}t+\dots+h_{n-1}\alpha
t^{n-1}+h_nt^n=f_1(\alpha-t,t)(P^n)
$$

Since $f_1(-\alpha,\alpha+t)=f_1(\alpha,t)$, the $h$-polynomial is symmetric:
$h(\alpha,t)=h(t,\alpha)$. So $h_i=h_{n-i}$.

\begin{defin}
A {\itshape $g$-vector} of a simple polytope $P^n$ is the set of numbers
$g_0=1$, $g_i=h_i-h_{i-1}, 1\leqslant i\leqslant [\frac{n}{2}]$.
\end{defin}

For any positive integers $a$ and $i$ there exists a unique ''binomial
$i$-decomposition'' of $a$:
$$a=\binom{a_i}{i}+\binom{a_{i-1}}{i-1}+\dots+\binom{a_j}{j},$$
where $a_i>a_{i-1}>\dots>a_j\geqslant j\geqslant 1.$

Let us define
$$a^{\langle
i\rangle}=\binom{a_i+1}{i+1}+\binom{a_{i-1}+1}{i}+\dots+\binom{a_j+1}{j+1},
\quad 0^{\langle i\rangle}=0.$$ For example, $a=\binom{a}{1}$, so $a^{\langle
1\rangle}=\binom{a+1}{2}=\frac{(a+1)a}{2}$; for $i\ge a$ we have:
$a=\binom{i}{i}+\binom{i-1}{i-1}+\dots+\binom{i-a+1}{i-a+1}$, so $a^{\langle
i\rangle}=a$.
\begin{thm*}[$g$-theorem, \cite{St1}, \cite{BL}]
Integer numbers $(g_0,g_1,\dots,g_{[\frac{n}{2}]})$ form the $g$-vector of
some simple $n$-dimensional polytope if and only if they satisfy the
following conditions:
$$
g_0=1,\;0\leqslant g_1,\;0\leqslant g_{i+1}\leqslant g_i^{\langle
i\rangle},\; i=1,2,\dots, \left[\frac{n}{2}\right]-1.
$$
\end{thm*}

As it is mentioned in \cite{Z1} even for $4$-dimensional non-simple polytopes
the corresponding problem involving flag $f$-vectors is extremely hard.

Modern tools used to obtain linear and nonlinear inequalities satisfied by
flag $f$-numbers are based on the notion of a {\itshape $cd$-index}
\cite{BK,BE1,BE2,EF,ER,L,St3}, a {\itshape toric $h$-vector}
\cite{BE,BL,K,St2,Sts} and its generalizations \cite{F1,F2,L},  a {\itshape
ring of chain operators} \cite{BH,BLiu,Kalai}.

In \cite{St1} R.~Stanley constructed for a simple polytope $P$ a projective
toric variety $X_P$ so that \{$h_i$\} are the even Betti numbers of the
(singular) cohomology of $X_P$. Then the Poincare duality and the Hard
Lefschetz  theorem for $X_P$ imply the McMullen conditions for $P$.

In \cite{Mc2} P.~McMullen gave the purely geometric proof of these conditions
using the notion of a {\itshape polytope algebra}.

In \cite{St2} Stanley generalized the definition of an $h$-vector to an
arbitrary polytope $P$ in such a way that if a polytope is rational the
generalized  $\hat h_i$ are the intersection cohomology Betti numbers of the
associated toric variety (see the validation of the claim in \cite{Fi}).

The set of numbers $(\hat h_0,\dots,\hat h_n)$ is called a {\itshape toric}
$h$-vector.  It is nonnegative and symmetric: $\hat h_i=\hat h_{n-i}$, and in
the case of rational polytope the Hard Lefschetz theorem in the intersection
cohomology of the associated toric variety proves the unimodality
$$
1=\hat h_0\leqslant \hat h_1\leqslant\dots\leqslant\hat h_{[\frac{n}{2}]}
$$
Generalizing the geometric method introduced by P.~McMullen \cite{Mc2} Kalle
Karu \cite{K} proved the unimodality for an arbitrary polytope.

The toric $h$-vector consists of linear combinations of the flag $f$-numbers.
In general case it does not contain the full information about the flag
$f$-vector. In \cite{L} C.~Lee introduced an ''extended toric'' $h$-vector
that carries all the information about the flag $f$-numbers and consists of a
collection of nonnegative symmetric unimodal vectors. In \cite{F1} and
\cite{F2} J.~Fine introduced the generalized $h$- and $g$-vectors that
contain the toric $h$- and $g$-vectors as subsets, but can have negative
entries.

Let us look on the problem of the description of flag $f$-vectors from the
point of view of the ring of polytopes.

$g$-theorem describes all the polynomials
$\psi(\alpha,t_1,t_2,\dots)\in\Qsym[t_1,t_2,\dots][\alpha]$ that are images
of simple polytopes under the ring homomorphism $f$:
$$
\psi(\alpha,t_1,t_2,\dots)=f(\alpha,t_1,t_2,\dots)(P^n),\quad
P^n\mbox{ -- a simple $n$-dimensional polytope}.
$$

In the general case we know the criterion when the polynomial
$\psi\in\Qsym[t_1,t_2,\dots][\alpha]$ is the image of an integer combination
of polytopes.
\begin{quest}
Find a criterion when the polynomial $\psi\in\Qsym[t_1,t_2,\dots][\alpha]$ is
the image of an $n$-dimensional polytope.
\end{quest}

\appendix
\section{Join}
\begin{prop}
$P^{n_1}\divideontimes Q^{n_2}$ is an $(n_1+n_2+1)$-dimensional polytope.

Faces of $P^{n_1}\divideontimes Q^{n_2}$ up to an affine equivalence are
exactly $F\divideontimes G$, where $\varnothing\subseteq F\subseteq
P^{n_1},\;\varnothing\subseteq G\subseteq Q^{n_2}$ are faces of $P^{n_1}$ and
$Q^{n_2}$ respectively.
\end{prop}
\begin{proof}
The formula for dimension of $P^{n_1}\divideontimes Q^{n_2}$ follows from
Remark \ref{Join}.

The nonempty faces of a polytope are exactly the subsets of the polytope that
minimize some linear function on it. Let $\ib{c}=(\ib{c}_1,\ib{c}_2)$ be a
linear function in $(\mathbb R^{n_1+1}\times\mathbb R^{n_2+1})^*=(\mathbb
R^{n_1+1})^*\times(\mathbb R^{n_2+1})^*$. Then $\ib{c}_1$ has its minimum
$c_1$ on the face $F$ of $P^{n_1}$, and $\ib{c}_2$ has its minimum $c_2$ on
the face $G$ of $Q^{n_2}$. Since any point of $P^{n_1}\divideontimes Q^{n_2}$
has the form $t_1\ib{p}+t_2\ib{q}$, where $\ib{p}\in P^{n_1},\, \ib{q}\in
Q^{n_2},\,t_1,t_2\geqslant 0, t_1+t_2=1$, we have
\begin{itemize}
\item if $c_1<c_2$, then $\ib{c}$ has its minimum on the face
    $F\divideontimes\varnothing$ of $P^{n_1}\divideontimes Q^{n_2}$;

\item if $c_1>c_2$, then $\ib{c}$ has its minimum on
    $\varnothing\divideontimes G$;

\item if $c_1=c_2$, then $\ib{c}$ has its minimum on the set
    $$
    X=\{t_1\ib{p}+t_2\ib{q},\;\ib{p}\in F,\,\ib{q}\in G,\; t_1,t_2\geqslant 0,\,t_1+t_2=1\}.
    $$
    Let $F$ lie in the $k$-simplex
    $\conv\{\ib{p}_1,\dots,\ib{p}_{k+1}\}\subset \aff(F)$, and $G$ lie in
    the $l$-simplex\linebreak
    $\conv\{\ib{q}_1,\dots,\ib{q}_{l+1}\}\subset\aff(G)$. Now let
    $\{\overrightarrow{O\ib{p}_i},\,i=1,\dots,k+1\}$ be a basis in
    $\mathbb R^{k+1}=\aff\{O,\,\ib{p}_i,\,i=1,\dots, k+1\}$ and
    $\{\overrightarrow{O\ib{q}_j},\,j=1,\dots,l+1\}$ be a basis in
    $\mathbb R^{l+1}=\aff\{O,\,\ib{q}_j,\,j=1,\dots, l+1\}$. Then
    $X=F\divideontimes G$.
\end{itemize}

Note that $\ib{c}_1$ can have the constant value on $P^{n_1}$, or $\ib{c}_2$
can have the constant value on $Q^{n_2}$. In these cases $F=P^{n_1}$, or
$G=Q^{n_2}$.

Let $\lambda_P$ be the linear function equal to the sum of all the
coordinates in $\mathbb R^{n_1+1}$, and $\lambda_Q$ be the linear function
equal to the sum of all the coordinates in $\mathbb R^{n_2+1}$. Then

\begin{gather*}
\lambda_P(\ib{p})=1,\quad\lambda_Q(\ib{p})=0\quad\mbox{for any }\ib{p}\in P^{n_1}=P^{n_1}\divideontimes \varnothing\\
\lambda_P(\ib{q})=0,\quad\lambda_Q(\ib{q})=1\quad\mbox{for any }\ib{q}\in
Q^{n_2}=\varnothing\divideontimes Q^{n_2}
\end{gather*}

Now let $F\ne\varnothing$, $F\subseteq P^{n_1}$ be a face of $P^{n_1}$, and
let $\ib{c}_1\in(\mathbb R^{n_1+1})^*$ be a linear function that has its
minimum $c_1$ exactly on $F$. Similarly, let $(G,\ib{c}_2,c_2)$ be the
corresponding data for $Q^{n_2}$. Then let $\ib{c}=(\ib{c}_1,\ib{c}_2)\in
\left(\mathbb R^{n_1+1}\times\mathbb R^{n_2+1}\right)^*$. By an addition or a
subtraction of $\lambda_P$ and $\lambda_Q$ with coefficients we can obtain
any of the cases $c_1<c_2$, $c_1=c_2$, or $c_1>c_2$, so
$F\divideontimes\varnothing$, $F\divideontimes G$, $\varnothing\divideontimes
G$ are faces of $P^{n_1}\divideontimes Q^{n_2}$.

Thus faces of the polytope $P^{n_1}\divideontimes Q^{n_2}$ up to an affine
equivalence have the form
$$
F\divideontimes G,\; \varnothing\subseteq F\subseteq P^{n_1},\; \varnothing\subseteq G\subseteq Q^{n_2}.
$$
\end{proof}
\section{Structure theorems}
\begin{prop}
$\mathcal{P}$ is a polynomial ring generated by indecomposable combinatorial
polytopes.
\end{prop}
\begin{proof}
We will show that any polytope $P^n$ of positive dimension can be represented
in a unique up to the order of factors way as a product of indecomposable
polytopes of positive dimensions $P=P_1\times\dots\times P_k$.

It is clear that such a decomposition exists. Now let
$$
P_1\times\dots\times P_k=Q_1\times\dots\times Q_s,
$$
where
$$
1\leqslant \dim P_1\leqslant\dim P_2\leqslant\dots\leqslant\dim P_k,\quad\mbox{and }1\leqslant \dim Q_1\leqslant\dim Q_2\leqslant\dots\leqslant \dim Q_s.
$$
Let us consider a face $F$ of the form
$\ib{v}_1\times\dots\times\ib{v}_{k-1}\times P_k\simeq P_k$, where $\ib{v}_i$
are vertices. Then we have $F=G_1\times \dots\times G_s$, where $G_i$ is a
face of $Q_i$. Since $P_k$ is indecomposable, we see that all $G_j$ but one
are points, and $P_k$ is combinatorially equivalent to a face of some $Q_r$.
Similarly $Q_r$ is combinatorially equivalent to a face of some $P_t$. Then
$\dim P_k\leqslant\dim Q_r\leqslant\dim P_t\leqslant\dim P_k$, therefore,
$P_k=Q_r$. Similarly $Q_s\simeq P_l$ for some $l$. So $\dim P_k=\dim Q_s$ and
without loss of generality we can set $P_k=Q_s$.

If $P_k=I$ then $P_i=Q_j=I$ for all $i,j$, so the statement is true. Let
$P_k\ne I$. We know that
$$
\ib{v}_1\times\dots\times\ib{v}_{k-1}\times P_k=\ib{w}_1\times\dots\times\ib{w}_{s-1}\times P_k
$$
Let us consider a face $G$ of the form $\ib{v}_1\times\dots\times
\ib{v}_{i-1}\times I\times \ib{v}_{i+1}\times\dots\times \ib{v}_{k-1}\times
P_k$, where $I$ is the edge that joins the vertices $\ib{v}_i$ and
$\ib{v}_i'$ of the polytope $P_i$. Then
$G=\ib{w}_1\times\dots\times\ib{w}_{j-1}\times
J\times\ib{w}_{j+1}\times\dots\times\ib{w}_{s-1}\times P_k$, where $J$ is the
edge that joins the vertices $\ib{w}_j$ and $\ib{w}_j'$ of some polytope
$Q_j$.

Since $\ib{v}_1\times\dots\times\ib{v}_i'\times\dots\times P_k\subset G$,
$P_k$ is indecomposable, and $P_k\ne J$, we have
$$
\ib{v}_1\times\dots\times\ib{v}_i'\times\dots\times P_k=\ib{w}_1\times\dots\times\ib{w}_j'\times\dots\times P_k.
$$
Changing the edges we can achieve any vertex
$\ib{v}_1'\times\dots\times\ib{v}_{k-1}'$ of $P_1\times\dots\times P_{k-1}$,
so for any such a vertex we have
$\ib{v}_1'\times\dots\times\ib{v}_{k-1}'\times P_k=\ib{w}_1'\times\dots\times
\ib{w}_{s-1}'\times P_k$. Therefore for any face $F_1\times \dots\times
F_{k-1}\times P_k$
$$
F_1\times \dots\times F_{k-1}\times P_k=G_1\times\dots\times G_{s-1}\times P_k
$$
Thus for any face $F$ of the polytope $P_1\times\dots\times P_{k-1}$ there is
a face $G$ of the polytope $Q_1\times\dots\times Q_{s-1}$ such that $F\times
P_k=G\times P_k$, and vice versa. Moreover, $F\subset F'$ if and only if
$G\subset G'$. Thus
$$
P_1\times\dots\times P_{k-1}=Q_1\times\dots\times Q_{s-1}
$$
Iterating this step in the end we obtain that $k=s$, and up to the order of
the polytopes $P_i=Q_i$, $i=1,\dots,k$.
\end{proof}

\begin{prop}
$\mathcal{RP}$ is a ring of polynomials generated by all join-indecomposable
polytopes.
\end{prop}
\begin{proof}
Let $P_1\divideontimes P_2\dots\divideontimes P_k=Q_1\divideontimes
Q_2\divideontimes\dots\divideontimes Q_s$, where $P_i, Q_j$ are polytopes,
$$
0\leqslant\dim P_1\leqslant \dim P_2\leqslant \dots\leqslant \dim P_k,\quad 0\leqslant\dim Q_1\leqslant \dim Q_2\leqslant\dots\leqslant \dim Q_s
$$
and each polytope $P_i $ and $Q_j$ can not be decomposed as a join of
polytopes of lower dimensions. Let $\dim P_k\geqslant \dim Q_s$. Consider the
face $\varnothing\divideontimes\varnothing\divideontimes\dots\divideontimes
P_k$. It can be represented in the form $G_1\divideontimes
G_2\divideontimes\dots\divideontimes G_s$, where $G_j$ is a face of $Q_j$.
Since $P_k$ is join-indecomposable, all $G_j$ but one should be equal to
$\varnothing$, and one of them should be equal to $P_k$. Let $P_k=G_r$. Then
$P_k$ is a face of $Q_r$, so $P_k=Q_r$, since $\dim P_k\geqslant\dim
Q_s\geqslant\dim Q_r$. Without loss of generality we can take $r=s$.

We have $P'\divideontimes P_k=Q'\divideontimes P_k$, and the combinatorial
equivalence identifies the face $\varnothing\divideontimes P_k$ of the first
polytope with the same face $\varnothing\divideontimes P_k$ of the second.
Consider the face $P'\divideontimes\varnothing$. It should have the form
$G_1\divideontimes G_2$ where $G_1$ is a face of $Q'$ and $G_2$ is a face of
$P_k$. Then $\varnothing\divideontimes G_2$ should be a face of
$\varnothing\divideontimes P_k$ and $P'\divideontimes \varnothing$, therefore
$G_2=\varnothing$ and $G_1=P'$. Since $\dim P'=\dim Q'$, we see that $P'=Q'$.

So $P_1\divideontimes P_2\divideontimes\dots\divideontimes
P_{k-1}=Q_1\divideontimes Q_2\divideontimes\dots Q_{s-1}$. Let $s\geqslant
k$. Iterating this step in the end we obtain $P_1=Q_1\divideontimes
Q_2\divideontimes\dots\divideontimes Q_{s-k+1}$. Therefore $k=s$ and
$P_1=Q_1$.

Thus each polytope can be decomposed in a unique up to permutations way as
$P_1\divideontimes\dots\divideontimes P_k$, where $P_j$ are
join-indecomposable. Since $\mathcal{RP}$ additively is a free abelian group
generated by combinatorial polytopes and $\varnothing$, it is a ring of
polynomials generated by all join-indecomposable polytopes.
\end{proof}

\section{Examples}
\begin{prop}
Any monomial $M_{\omega}$ such that $M_{\omega}\in \F(\mathcal{RP})$ has the
form $M_{(2k+1)},\;k\geqslant 0$.
\end{prop}
\begin{proof}
Indeed, for the monomial $M_{(a_1,\,\dots,\,a_k)}$ we have
\begin{multline*}
M_{(a_1,\,\dots,\,a_k)}(t,-t,t_1,t_2,\dots)=M_{(a_1,\,\dots,\,a_k)}(t_1,t_2,\dots)+\left(t^{a_1}+(-t)^{a_1}\right)M_{(a_2,\,\dots,\,a_k)}(t_1,t_2,\dots)+\\
+(-1)^{a_2}t^{a_1+a_2}M_{(a_3,\,\dots,\,a_k)}
\end{multline*}
If $k\geqslant 2$ the last summand does not cancel, so $k=1$. For $M_{(n)}$
with $n=2l+1$ we have
$$
M_{(2l+1)}(t_1,\dots,t_{i-1},t,-t,t_{i},\dots)=t_1^{2l+1}+\dots+t_{i-1}^{2l+1}+t^{2l+1}-t^{2l+1}+t_i^{2l+1}+\dots=M_{(2l+1)}(t_1,\dots,t_{i-1},t_i,\dots),
$$
while for $n=2l$ the term containing $t^{2l}+(-t)^{2l}$ does not cancel.
\end{proof}

\begin{prop}
$M_{(2k+1,\,2l+1)}-M_{(2l+1,\,2k+1)}\in\F(\mathcal{RP})$
\end{prop}
\begin{proof}
Indeed,
\begin{multline*}
M_{(2k+1,\,2l+1)}(t_1,\dots,t_{i-1},t,-t,t_i,\dots)=M_{(2k+1,\,2l+1)}(t_1,\dots,t_{i-1},t_i,\dots)+\\
+\sum\limits_{j<i}t_j^{2k+1}(t^{2l+1}+(-t)^{2l+1})+\sum\limits_{j\geqslant
i}(t^{2k+1}+(-t)^{2k+1})t_j^{2l+1}+(-1)^{2l+1}t^{2k+2l+2}
\end{multline*}
Then
\begin{multline*}
\left(M_{(2k+1,\,2l+1)}-M_{(2l+1,\,2k+1)}\right)(t_1,\dots,t_{i-1},t,-t,t_i,\dots)=\\
\left(M_{(2k+1,\,2l+1)}-M_{(2l+1,\,2k+1)}\right)(t_1,\dots,t_{i-1},t_i,\dots)
\end{multline*}
\end{proof}

This proposition can be generalized.
\begin{prop}
Let $(j_1,\dots,j_k)$ be the set of odd positive integer numbers. Then the
quasi-symmetric polynomial
$$
\sum\limits_{\sigma\in S_k}(-1)^{\sigma}M_{(j_{\sigma(1)},\,\dots,\,j_{\sigma(k)})}
$$
belongs to the image of $\F$.
\end{prop}
\begin{proof}
Indeed, since $j_l$ are odd, we have
\begin{multline*}
M_{(j_1,\,\dots,\,j_k)}(t_1,\dots,t_{i-1},t,-t,t_i,\dots)=M_{(j_1,\,\dots,\,j_k)}(t_1,\dots, t_{i-1},t_i,\dots)-\\
-\sum\limits_{l_1<\dots<l_s<i\leqslant l_{s+3}<\dots<l_k}t_{l_1}^{j_1}\dots
t_{l_s}^{j_s}t^{j_{s+1}+j_{s+2}}t_{l_{s+3}}^{j_{s+3}}\dots t_{l_k}^{j_k}
\end{multline*}
For any pair $(j_{s+1},j_{s+2})$ there exists a unique monomial
$M_{(j_1,\,\dots,\,j_s,\,j_{s+2},\,j_{s+1},\,j_{s+3},\,\dots,\,j_k)}$, and
this monomial has the opposite sign. Therefore all the monomials of the form
$t_{l_1}^{j_{\sigma(1)}}\dots
t_{l_s}^{j_{\sigma(s)}}t^{j_{\sigma(s+1)}+j_{\sigma(s+2)}}t_{l_{s+3}}^{j_{\sigma(s+3)}}\dots
t_{l_k}^{j_{\sigma(k)}}$ vanish.
\end{proof}

Let us calculate the image of the Ehrenborg transformation $\F$ in small
dimensions.
\begin{itemize}
\item[$n=0$] The image of the point $x=\pt$ is $M_{(1)}=\sigma_1$. So
    $$
    \F(\mathcal{RP}^2)=\EuScript{R}^*((\mathcal{D}^*)^2)=\mathbb
    Z\langle \sigma_1\rangle.
    $$
\item[$n=1$] The image of the segment $I=\pt*\pt$ is
    $M_{(2)}+2M_{(1,\,1)}=\sigma_1^2$.
    So
    $$
    \F(\mathcal{RP}^4)=\EuScript{R}^*((\mathcal{D}^*)^4)=\mathbb
    Z\langle \sigma_1^2\rangle.
    $$
\item[$n=2$] The image of the $m$-gon $M_m^2$ is
$$
M_{(3)}+mM_{(2,\,1)}+mM_{(1,\,2)}+2mM_{(1,\,1,\,1)}=M_{(3)}+m(\sigma_1\sigma_2-\sigma_3)
$$
Therefore,
$$
\F(\mathcal{RP}^6)=\EuScript{R}^*((\mathcal{D}^*)^6)=\mathbb
Z\langle M_{(3)},\sigma_1\sigma_2-\sigma_3\rangle,
$$
where
$$
M_{(3)}=\F(4\Delta^2-3I^2),\quad \sigma_1\sigma_2-\sigma_3=\F(I^2-\Delta^2),
$$
\item[$n=3$] The image of a $3$-dimensional polytope is
$$
M_{(4)}+f_0M_{(1,\,3)}+f_1M_{(2,\,2)}+f_2M_{(3,\,1)}+f_{01}M_{(1,\,1,\,2)}+f_{02}M_{(1,\,2,\,1)}+f_{12}M_{(2,\,1,\,1)}+f_{012}M_{(1,\,1,\,1,\,1)}
$$
Since $f_0-f_1+f_2=2$, $f_{01}=f_{02}=f_{12}=2f_1$, and $f_{012}=4f_1$,
we obtain
$$
\left(M_{(4)}+2M_{(3,\,1)}\right)+f_0\left(M_{(1,\,3)}-M_{(3,\,1)}\right)+f_1\left(M_{(2,\,2)}+M_{(3,\,1)}+2\left(\sigma_1\sigma_3-2\sigma_4\right)\right),
$$
Then
$$
\F(\mathcal{RP}^8)=\EuScript{R}^*((\mathcal{D}^*)^8)=\mathbb
Z\langle M_{(4)}+2M_{(3,\,1)},M_{(1,\,3)}-M_{(3,\,1)},
M_{(2,\,2)}+M_{(3,\,1)}+2\left(\sigma_1\sigma_3-2\sigma_4\right)\rangle,
$$
where
\begin{gather*}
M_{(4)}+2M_{(3,\,1)}=\F(5\Delta^3-6CI^2+2B\Delta^2),\quad M_{(1,\,3)}-M_{(3,\,1)}=\F(-\Delta^3+3CI^2-2B\Delta^2)\\
M_{(2,\,2)}+M_{(3,\,1)}+2\left(\sigma_1\sigma_3-2\sigma_4\right)=\F(-CI^2+B\Delta^2)
\end{gather*}
\end{itemize}

\end{document}